\documentclass[11pt,english]{article}
\usepackage[T1]{fontenc}
\usepackage[utf8]{inputenc}
\usepackage[letterpaper]{geometry}
\usepackage{amsmath}
\usepackage{amsthm}
\usepackage{tikz}
\usetikzlibrary{arrows}
\usepackage{graphicx}
\usepackage{amssymb}

\makeatletter

\newtheorem{theorem}{Theorem}[section]
\newtheorem{corollary}[theorem]{Corollary}
\newtheorem{lemma}[theorem]{Lemma}

\newtheorem{definition}[theorem]{Definition}

\newtheorem{remark}[theorem]{Remark}

\numberwithin{equation}{section}

\makeatother

\newcommand{\eqnsection}{
	\renewcommand{\theequation}{\thesection.\arabic{equation}}
	\makeatletter   \csname  @addtoreset\endcsname{equation}{section}
	\makeatother}

\def\qed{$\Box $}

\def\R{\mathbb{R}}

\def\E{\mathbb{E}}
\def\P{\mathbb{P}}
\def\0{\mathbf{0}}
\def\1{\mathbf{1}}

\def\Var{{\mathop {{\rm Var\, }}}}
\def\Cov{{\mathop {{\rm Cov\, }}}}
\def\tr{{\mathop {{\rm Tr\, }}}}

\makeatother

\usepackage{babel}

\makeatother

\begin{document}
	
	\title{Hyperbolic Anderson equations with general time-independent
		Gaussian noise: Stratonovich regime}
	
	\author{\sc Xia Chen\thanks {XC is    supported by
			the Simons Foundation \#585506.},\quad  Yaozhong Hu\thanks {YH is supported by an NSERC discovery fund and
			a centennial fund of University of Alberta.} }
	
	\date{} 
	\maketitle
	
	\begin{abstract}
		In this paper, we investigate the 
		hyperbolic Anderson equation
		generated by a time-independent Gaussian noise
		with two objectives: The solvability and intermittency.
		First, we prove that   Dalang's condition is necessary and sufficient
		for existence  of the solution. Second, we establish the
		precise long time  and high moment asymptotics for the solution under
		the usual homogeneity assumption of the covariance of the Gaussian noise.
		Our approach is fundamentally different from the ones existing in
		literature.
		The main contributions  in our approach include the representation of Stratonovich
		moment under Laplace  transform via the  moments of the Brownian motions in
		Gaussian potentials
		and some large deviation skills developed
		in dealing effectively with the Stratonovich chaos expansion.

	\end{abstract}
	
	\begin{quote} {\footnotesize

			}

	\end{quote}

	\section{Introduction} \label{intro}

In this paper we consider the hyperbolic Anderson equation 
\begin{align}\label{intro-1}
	\left\{\begin{array}{ll}\displaystyle
		\frac{\partial^2 u}{\partial t^2}(t,x)=\Delta u(t,x)
		+ \dot{W}(x) 
		u(t,x)\,, \hskip.2in (t, x)\in \R^+\times\R^d\\\\
		u(0, x)=u_0(x)  \hskip.1in\hbox{and} \hskip.1in
		\frac{\partial u}{\partial t}(0,x)=u_1(x)\,, \hskip.2in x\in\R^d\end{array}\right.
\end{align}
run by a  time-independent, mean zero and possibly generalized
Gaussian noise $\dot{W}(x)$ 
with
the covariance function
\begin{align}\label{intro-2}
	\Cov \big(\dot{W}(x), \dot{W}(y)\big)
	=\gamma(x-y)\,, \hskip.2in x,y\in\R^d\,. 
\end{align}
As a  covariance function the non-negative definiteness   of  $\gamma(\cdot)$   implies that it admits
a spectral measure $\mu(d\xi)$ on
$\R^d$ uniquely defined by the relation
\begin{align}\label{intro-3}
	\gamma(x)=\int_{\R^d} e^{i\xi\cdot x}\mu(d\xi)\,, \hskip.3in x\in\R^d\,. 
\end{align}

Throughout this work, we assume that $\gamma(\cdot)\ge 0$ and $d=1,2,3$. 
The system is set up in Stratonovich regime in the sense that the product 
in \eqref{intro-1} is  interpreted as the ordinary (instead of Wick)  one.
The equation    (\ref{intro-1}) will be  approximated appropriately by  
classical  wave equations run by the smoothed
Gaussian noise $\dot{W}_\varepsilon(x)$.  We shall provide the details of the construction of the solution 
in Section \ref{M}.

Our first concern is the condition to ensure the existence of  solution. It is often formulated
in terms of the  integrability of the spectral measure $\mu(d\xi)$.  
In the Skorohod regime, where
the product between
$\dot{W}(  x)$ and $u(t,x)$ in (\ref{intro-1}) is understood as Wick product,
the condition (\cite[Theorem 1.6]{BCC}, \cite[Remark 3.4]{CDST} )
that (\ref{intro-1}) has a unique solution
is 
\begin{align}\label{intro-4}
	\int_{\R^d}\bigg(\frac{1}{1+\vert\xi\vert^2}\bigg)^{3/2}\mu(d\xi)<\infty\,. 
\end{align}

Back to the Stratonovich regime and still in the time independent setting, Balan (\cite{Balan}) recently
proved that  in the dimensions $d=1,2$   Equation (\ref{intro-1}) has a  
solution if
\begin{align}\label{intro-5}
	\int_{\R^d}\bigg(\frac{1}{ 1+\vert\xi\vert^2}\bigg)^{1/2}\mu(d\xi)<\infty\,. 
\end{align}
In the setting of time-space Gaussian noise, Chen, Deya, Song and Tindel  (\cite{CDST-1})
establish the existence/uniqueness under a condition comparable to (\ref{intro-5}).

Our first main result is to obtain  the best condition for the existence of the solution, which is to remove the square root in \eqref{intro-5}. 
We can also allow the spatial dimension to be three as well. 
\begin{theorem}\label{th-1}  Let $d=1,2,3$ and assume that $u_0(x)=1$ and $u_1(x)=0$ in (\ref{intro-1}).
	\begin{enumerate}
		\item[(i)]   Under   Dalang's condition
		\begin{align}\label{intro-6} 
			\int_{\R^d}\frac{1}{1+\vert \xi\vert^2}\mu(d\xi)<\infty
		\end{align}  
		the equation (\ref{intro-1}) has a   solution in the sense of
		Definition \ref{d.mild_solution} given in Section \ref{M}.
		\item[(ii)] If the equation (\ref{intro-1}) has a  square integrable solution $u(t,x)$ that admits the  Stratonovich expansion (see (\ref{M-3}))
		for some $t>0$, then  Dalang's condition \eqref{intro-6}
		must be satisfied.  
	\end{enumerate}
\end{theorem}

Roughly speaking, the system (\ref{intro-1}) in Stratonovich
regime can be viewed as a randomization of the deterministic wave
equation
\begin{align}\label{intro-7}
	\left\{\begin{array}{ll}\displaystyle
		\frac{\partial^2 u}{\partial t^2}(t,x)=\Delta u(t,x)
		+ f(x) 
		u(t,x)\,, \hskip.2in (t, x)\in \R^+\times\R^d\\\\
		u(0, x)=u_0(x)\hskip.1in\hbox{and} \hskip.1in
		\frac{\partial u}{\partial t}(0,x)=u_1(x)\,, \hskip.2in x\in\R^d\end{array}\right.
\end{align}
with a deterministic potential function $f(x)$ on $\R^d$.  To this regard,
it is hard not to notice
the stochastic representation constructed by Dalang, Mueller and Tribe (\cite{DMT}). We devote the subsection
\ref{S-3} below to address this link and to add some new elements to
the representation theory for wave equations.

Our next topic is the ittermittency of the  equation  (\ref{intro-1}). More precisely,
our concern is the asymptotic behaviors of the moments
$$
\E u^p(t,x)  \hskip.1in\hbox{and}\hskip.1in\E \vert u(t,x)\vert^p
$$
as $t\to\infty$ or as $p\to\infty$.

To this end, we assume the
homogeneity for the covariance structure:  
\begin{align}\label{intro-8} 
	\gamma(cx)=c^{-\alpha}\gamma(x)\,, \hskip.2in x\in\R^d\,, \hskip.1in c>0
\end{align} 
for some $\alpha>0$.
Taking $f(\lambda)=(1+\lambda^2)^{-1}$ and $v(d\xi)=\mu(d\xi)$ in   \cite[Lemma 3.10]{CDST}
yields 
$$
\int_{\R^d}\frac{1}{ 1+\vert\xi\vert^2}\mu(d\xi)=\alpha\mu\{\xi\in\R^d;\hskip.05in\vert\xi\vert\le 1\}
\int_0^\infty \frac{1}{ 1+\rho^2}\frac{d\rho}{\rho^{1-\alpha}}
$$
as far as either of the above two sides is finite.  This
shows that under the homogeneity (\ref{intro-8}) on the 
noise covariance condition,
Dalang's  condition (\ref{intro-6}) becomes   ``$\alpha<2$''.
In addition (Remark 1.4, \cite{CDST}), the fact that $\gamma(\cdot)$ is non-negative and
non-negative
definite (for being qualified as covariance function)
requires that $\alpha\le d$. Further, the only setting where
``$\alpha=d$'' is allowed
under $\alpha<2$ is when $\alpha=d=1$, or when $\gamma (\cdot)$ is a
constant multiple of Dirac function (i.e., $\dot{W}$ is
an 1-dimensional spatial white noise, see Corollary \ref{co-0} below for
inttermitency in this case).

\begin{theorem}\label{th-2} Under the homogeneity condition (\ref{intro-8})
	with $0<\alpha<2\wedge d$ or with $\alpha =d=1$
	and under the initial condition $u_0(x)=1$ and $u_1(x)=0$,
	the following limits hold: 
	\begin{align}\label{intro-9} 
		\lim_{t\to\infty}t^{-\frac{4-\alpha}{ 3-\alpha}}\log\E u^p(t,x)
		=\frac{3-\alpha}{ 2}p^{\frac{4-\alpha}{ 3-\alpha}}
		\bigg(\frac{ 2{\cal M}^{1/2}}{ 4-\alpha}\bigg)^{\frac{ 4-\alpha}{ 3-\alpha}}\,, 
		\hskip.2in p=1,2,\cdots \,; 
	\end{align} 
	\begin{align}\label{intro-10} 
		\lim_{p\to\infty}p^{-\frac{4-\alpha}{ 3-\alpha}}\log\E \vert u(t,x)\vert^p
		=\frac{3-\alpha}{  2}t^{\frac{4-\alpha}{3-\alpha}}
		\bigg(\frac{ 2{\cal M}^{1/2}}{ 4-\alpha}\bigg)^{\frac{ 4-\alpha}{ 3-\alpha} }\,,  \hskip.2in \forall \ t>0\,, 
	\end{align} 
	where
	\begin{align}\label{intro-11}
		{\cal M}=\sup_{g\in {\cal F}_d}\bigg\{\bigg(\int_{\R^d\times\R^d}\gamma(x-y)
		g^2(x)g^2(y)dxdy\bigg)^{1/2}-\int_{\R^d}\vert\nabla g(x)\vert^2dx\bigg\}
	\end{align}
	and
	$$
	{\cal F}_d=\bigg\{g\in W^{1,2}(\R^d);\hskip.1in 
	\int_{\R^d}\vert g(x)\vert^2dx=1\bigg\}\,, 
	$$
	where $W^{1,2}$ is the Sobolev space. 
\end{theorem}

An interesting special case is when $\dot{W}(x)$ ($x\in\R$) is  a  white noise that
symbols the derivative of  
a two sided Brownian motion  $W(x)$ on $\R$. The corresponding covariance $\gamma(\cdot)=\delta_0(\cdot)$   is
the Dirac delta function  and   the spectral measure $\mu(d\xi)= d\xi/(2\pi)$  is a multiple of the  Lebesgue  measure on $\R$.
In this case
by   \cite[Theorem C.4, p.307]{Chen-1} (with $p=2$ and $\theta=1$),   we have 
$$
{\cal M}=\frac{1}{ 4}\root 3\of{\frac{3}{ 2}} \,. 
$$
Thus we can write 
\begin{corollary}\label{co-0} When  $\dot{W}(x)$ ($x\in\R$) is an 1-dimensional white noise
	\begin{align}\label{intro-12} 
		\lim_{t\to\infty}t^{-3/2}\log\E u^p(t,x)=\frac{1}{ 2}\root 4\of{\frac{3}{4}} p^{3/2}\,, 
		\hskip.2in p=1,2,\cdots.
	\end{align} 
	\begin{align}\label{intro-13} 
		\lim_{p\to\infty}p^{-3/2}\log\E \vert u(t,x)\vert^p=\frac{1}{ 2}\root 4\of{\frac{3}{ 4}} t^{3/2}\,, 
		\hskip.2in \forall \ t>0\,. 
	\end{align} 
\end{corollary}

In  Skorohod regime (\cite{BCC}), the high moment asymptotic theorem takes
the same form as (\ref{intro-10}),
while the long time  asymptotic theorem takes the form
\begin{align}\label{intro-14} 
	\lim_{t\to\infty}t^{-\frac{4-\alpha}{3-\alpha}}\log\E \vert u(t,x)\vert^p
	=\frac{3-\alpha}{ 2}p(p-1)^{\frac{1}{ 3-\alpha}}
	\bigg(\frac{ 2{\cal M}^{1/2}}{4-\alpha}\bigg)^{\frac{ 4-\alpha}{ 3-\alpha}}
\end{align} 
for  $p\ge 2$.

We now mention some new ideas that are introduced in this paper.  As usual, the solution can be formally
written in terms of Stratonovich expansion (\ref{M-3}).
Therefore, the level of investigation  is largely determined by our capability of handling the Stratonovich multiple
integral
$S_n\big(g_n(\cdot, t,x)\big)$ (see (\ref{M-4}) for its definition) for fixed $n$ and for large $n$ as well.
To this regard, the most significant observation made in this paper is the  moment representation given in Theorem \ref{th-4}
that associates the study of $S_n\big(g_n(\cdot, t,x)\big)$ to the problem of Brownian motions in 
Gaussian potential.
Another notable input is the algorithm development related to the Gaussian moment formula (\ref{M-8}), which is crucial
to, among other things,  the establishment of  a moment inequality  (Lemma \ref{L-5}) for the lower bound of the high moment
asymptotics given in (\ref{intro-10}). Last but not least, some skills on large deviations and Laplacian transforms are developed for dealing with Stratonovich
expansion.

Here is the organization of the paper. In  next section (Section \ref{M}),  
we introduce  the  multiple Stratonovich 
integral and 
formally express the solution  as  Stratonovich expansion.
In Section 3, we establish the Stratonovich integrability for the  functions
$g_n(\cdot, t,x)$, develop the Fubini theorem for the multiple  Strotonovich integration
and represent the Laplace transform of the multiple Stratonovich 
integral  $S_n\big(g_n(\cdot, t,x)\big)$ in terms of Brownian
motions in Gaussian potential. Section \ref{s.4} and   Section  \ref{s.5}  are devoted to  the proofs of Theorem \ref{th-1} and \ref{th-2}, respectively.
Some relevant results about the moment bound of Brownian intersection local times
and about multiple Stratonovich integrals 
are provided in the appendix.

\section{Stratonovich expansion and approximations} \label{M}
As usual by the Duhammel principle the mathematical definition of the hyperbolic Anderson equation (\ref{intro-1}) will be  the following mild form
\begin{align}\label{M-1}
	u(t,x)=u_0(t,x)+\int_{\R^d}\bigg[\int_0^tG(t-s, x-y)u(s,y)ds\bigg]W(dy)\,, 
\end{align}
where 
\begin{enumerate}
	\item[(i)] $G(t, x)$ is the fundamental solution  defined by the
	deterministic wave equation
	\begin{align}\label{M-2}
		\left\{\begin{array}{ll}\displaystyle \frac{\partial^2 G}{\partial t^2}(t,x)
			=\Delta G(t,x) \\\\
			\displaystyle G(0, x)=0\hskip.1in\hbox{and}\hskip.1in\frac{\partial G}{\partial t}(0,x)
			=\delta_0(x)\,, \hskip.2in x\in\R^d\,.\end{array}\right.
	\end{align} 
	\item[(ii)] $u_0(t,x)$ is the solution to the deterministic part of the equation \eqref{intro-1}:  
	\[
	u_0(t,x)=\int_{\R^d}   \frac{\partial }{\partial t} G(t, x-y) u_0(y) dy+\int_{\R^d}  G(t, x-y) u_1(y) dy\,. 
	\]
	Under the initial condition given in Theorem \ref{th-1} and Theorem \ref{th-2},
	$u_0(t,x)\equiv 1$;   
	\item[(iii)] the stochastic integral on the right hand side of (\ref{M-1})
	is interpreted as Stratonovich  one  (see discussion below for details).   
\end{enumerate}

\subsection{Green's function}

The fundamental solution $G(t,x)$ associated with  (\ref{M-2}) plays a key role in determining the behavior of
the system (\ref{M-1}). Let us recall some  basic facts. 
Taking Fourier  transform in (\ref{M-2}) we get the expression for the 
fundamental solution
\begin{align}\label{M-11}
	\int_{\R^d}G(t,x)e^{i\xi\cdot x}dx=\frac{\sin(\vert\xi\vert t)}{\vert\xi\vert}\,, \hskip.2in
	(t,\xi)\in\R^+\times\R^d
\end{align}
in its Fourier transform.
In the dimensions $d=1,2,3$, the fundamental solution $G(t,x)$ 
itself can be expressed explicitly as
\begin{align}\label{M-12}
	G(t,x)=
	\begin{cases}\frac12 1_{\{\vert x\vert\le t\}}&\hskip.3in d=1\\ 
		\displaystyle \frac{1}{ 2\pi}\frac{ 1_{\{\vert x\vert\le t\}}}{\sqrt{t^2-\vert x\vert^2}}&\hskip.3in d=2\\ 
		\displaystyle\frac{1}{ 4\pi t}\sigma_t(dx) &\hskip.3in d=3\,,  
	\end{cases}
\end{align}
where $\sigma_t(dx)$ is the surface measure on the sphere $\{x\in\R^3;\hskip.05in\vert x\vert=t\}$.
We limit our attention to $d=1,2,3$ in this work because the treatment developed   here
requires $G(t,x)\ge 0$.  A scaling property we frequently use 
(especially in the proof of Theorem \ref{th-2}) is
\begin{align}\label{M-13}
	G(t,x)=t^{-(d-1)}G(1,t^{-1}x)\,, \hskip.2in (t,x)\in\R^+\times\R^d\,. 
\end{align}

\subsection{Stratonovich integral} 
Before  giving  the definition of the mild solution we need to give a meaning to 
the Stratonovich integral appeared in \eqref{M-1}. We shall do this by smoothing the noise as follows
\begin{align}\label{M-5}
	\dot{W}_\varepsilon(x)=\int_{\R^d} \dot{W}(y)p_\varepsilon(y-x)dy\,, \hskip.2in \varepsilon>0,\hskip.1in x\in\R^d
	\,, 
\end{align}
where $p_\varepsilon (x)= (2\pi  \varepsilon)^{-d/2} \exp\left(-\frac{|x|^2}{2 \varepsilon}\right)$
is the heat kernel.   The covariance of   $
\dot W_\varepsilon(x)$ is
\begin{equation}
	\E  \left[ \dot W_\varepsilon(x)\dot W_\varepsilon(x)
	\right]=\gamma_{2\varepsilon} (x-y)\,,
\end{equation}
where $\gamma_\varepsilon(x)=\int_{\R^d} \gamma(z) p_{ \varepsilon}(x-z) dz$.   
Given a random field $\Psi(x)$ ($x\in\R^d$) such that
$$
\int_{\R^d}\Psi(x)\dot{W}_\varepsilon(x)dx\in{\cal L}^2(\Omega, {\cal F},\P)\hskip.2in \forall\varepsilon>0\,, 
$$
We define  the Stratonovich integral of $\{\Psi(x), x\in \R \}$ as 
\begin{equation}
	\int_{\R^d}\Psi(x)W(dx)\buildrel \Delta\over
	=\lim_{\varepsilon\to 0^+}\int_{\R^d}\Psi(x)
	\dot{W}_\varepsilon(x)dx
	\label{e.def_integral} 
\end{equation} 
whenever  such limit exists in ${\cal L}^2(\Omega, {\cal F},\P)$. We can also use the convergence in probability in above definition. But as in most works  on SPDE, ${\cal L}^2(\Omega, {\cal F},\P)$ norm is   easier to deal with so that we choose the ${\cal L}^2(\Omega, {\cal F},\P)$ convergence 
throughout  this work.  Notice that   this definition 
implicates  that $u(t,x) $  as  a
solution to (\ref{M-1}) is in $  {\cal L}^2(\Omega, {\cal F},\P)$   for all $(t,x)\in\R^+\times\R^d$.  After  defining the Stratonovich integral, we can give the following definition about  the solution.
\begin{definition}\label{d.mild_solution} 
	A random field $\{u(t,x)\,, t\ge 0\,, x\in \R^d\}$  is called a mild solution 
	to \eqref{intro-1} if  	$ \int_0^tG(t-s, x-y)u(s,y)ds$ is well-defined and is 
	Stratonovich integrable such that  \eqref{M-1} is satisfied. 
\end{definition}

To prove Theorem \ref{th-1}, we shall use the Stratonovich 
expansion (see \cite{humeyer}, \cite{hubook} and
references therein for the multiple
Stratonovich integrals). 
Formally iterating (\ref{M-1})
infinitely many  times we   have  heuristically a solution candidate 
\begin{align}\label{M-3}
	u(t,x)=\sum_{n=0}^\infty S_n\big(g_n(\cdot,t,x)\big)
\end{align}
with $S_0\big(g_0(\cdot, t,x)\big)=   1$. Here is how the notation
$S_n\big(g_n(\cdot,t,x)\big)$ is justified:
The iteration procedure creates the recurrent relation
\begin{equation}
	S_{n+1}\big(g_{n+1}(\cdot, t,x)\big)= \int_{\R^d}\! \left[\int_0^t
	G(t-s, x-y)S_n\big(g_n(\cdot, s,y)\big)ds \right]
	{W}(dy) \,. 
	\label{e.2.9}
\end{equation} 
Iterating this relation formally we have
\begin{align}\label{M-4}
	&S_n\big(g_n(\cdot,t,x)\big)\\
	&=\int_{(\R^d)^n}\bigg[\int_{[0,t]_<^n}d{\bf r}
	G(t-r_n, y_n-x)\cdots G(r_2-r_1, y_2-y_1)\bigg]W(dx_1)\cdots W(dx_n)\nonumber\\
	&=\int_{(\R^d)^n}\bigg[\int_{[0,t]_<^n}d{\bf s}
	\bigg(\prod_{k=1}^n
	G(s_k-s_{k-1}, x_k-x_{k-1})\bigg)\bigg]W(dx_1)\cdots W(dx_n)\nonumber\\
	&=\int_{(\R^d)^n}g_n(x_1,\cdots, x_n, t,x)W(dx_1)\cdots W(dx_n)
	\hskip.2in \hbox{(say)}\,, \nonumber                               
\end{align}
where $[0,t]_<^n:=\left\{ (s_1, \cdots, s_n)\in [0, t]^n \ \  \hbox{satisfying} 
\ \  0<s_1<s_2<\cdots<s_n<t\right\}$,  and the conventions $x_0=x$ and $s_0=0$ are adopted and the above  second equality follows
from the substitutions $s_k=t-r_{n-k+1}$ and $x_k=y_{n-k+1}-x$
($k=1,\cdots, n$).

Thus, the notation ``$S_n\big(g_n(\cdot,t,x)\big)$'' is reasonably introduced
for
a $n$-multiple Gaussian integral of the integrand
\begin{align}\label{M-14}
	g_n(x_1,\cdots, x_n, t,x)=\int_{[0,t]_<^n}\bigg(\prod_{k=1}^n
	G(s_k-s_{k-1}, x_k-x_{k-1})\bigg)ds_1\cdots ds_n\,, 
\end{align}
($n=1,2,\cdots$). In Section 3, the Stratonovich
integrability of $g_n(\cdot, t,x)$
shall be rigorously established (see 
Theorem \ref{prop}  and see also Theorem \ref{t.6.2} for Stratonovich integrablity of general kernels) and the Fubini's theorem posted in (\ref{e.2.9})
shall be mathematically ratified  (Remark \ref{re}).

The above argument gives us the impression  that the  existence of   system (\ref{M-1}) 
can be implied by  the convergence
of the random series defined by  (\ref{M-3}) in a certain  appropriate form. 
This will be justified rigorously in Section \ref{s.4} after we have more understanding 
of the multiple Stratonovich integral $S_n\big(g_n(\cdot, t,x)\big)$ with the specific kernel 
\eqref{M-14}.


The multiple Strtonovich integration is defined as follows. 

\begin{definition}\label{def-2.2}
	Let $f:\R ^n\to \R $ be measurable
	so that   for every $\varepsilon>0$ 
	\[
	\int_{(\R^d)^{n}}f(x_1\cdots, x_{n})\bigg(\prod_{k=1}^n\dot{W}_{\varepsilon}(x_k)\bigg)
	dx_1\cdots dx_n\in   {\cal L}^2(\Omega, {\cal F}, \P)\,.
	\]
	Then we define the $n$-multiple Stratonovich integral of $f$ as
	\begin{align}\label{M-6}
		S_n(f):=&   \int_{(\R^d)^{n}}  f(x_1\cdots, x_{n})W(dx_1)\cdots W(dx_n)\\
		=&\lim_{\varepsilon\to 0^+}
		\int_{(\R^d)^{n}}f(x_1\cdots, x_{n})\bigg(\prod_{k=1}^n\dot{W}_{\varepsilon}(x_k)\bigg)
		dx_1\cdots dx_n \nonumber  
	\end{align}
	whenever the limit exists $  {\cal L}^2(\Omega, {\cal F}, \P)$. 
\end{definition}
\begin{remark}Along with the set-up of our model, the Stratonovich integrand $f$ is
	given as a measure in the dimension three ($d=3$). Indeed  (\cite{pipiras}),  Definition \ref{def-2.2} can be extended
	to the setting of  generalized functions $f$.  A detail is provided near the end of this section for the construction
	needed in $d=3$.
	%
\end{remark} 
The following lemma provides a convenient test
of Stratonovich integrability that we shall use   in this work. 

\begin{lemma}\label{L-0} The $n$-multiple Stratonovich integral $S_n(f)$ 
	exists
	if and only if the limit
	$$
	\begin{aligned}
		\lim_{\varepsilon,\varepsilon'\to 0^+}&
		\E\bigg\{\int_{(\R^d)^n}f(x_1,\cdots x_n)\bigg(\prod_{k=1}^n\dot{W}_{\varepsilon}(x_k)\bigg)
		dx_1\cdots dx_n\bigg\}\\
		&\times\bigg\{\int_{(\R^d)^n}f(x_1,\cdots x_n)\bigg(\prod_{k=1}^n\dot{W}_{\varepsilon'}(x_k)\bigg)
		dx_1\cdots dx_n\bigg\}
	\end{aligned}
	$$
	exists
\end{lemma}

\begin{proof} The existence of the limit in (\ref{M-6}) is
	another way to say that the family
	$$
	{\cal Z}_{\varepsilon}
	=\int_{(\R^d)^n}f(x_1,\cdots x_n)\bigg(\prod_{k=1}^n\dot{W}_{\varepsilon}(x_k)\bigg)
	dx_1\cdots dx_n\,, \quad \varepsilon>0\,, 
	$$
	is a Cauchy sequence in ${\cal L}^2(\Omega, {\cal F},\P)$ as $\varepsilon\to 0^+$, 
	which is equivalent to the lemma.
\end{proof} 

We refer to Theorem \ref{t.6.2} for  the exact conditions on $f$ so that the multiple Stratonovich integral 
$S_n(f)$ exists in $L^2$.

The definition \ref{def-2.2} can be extended to a random field
$f(x_1,\cdots, x_n)$ in an obvious way. For most of the time in this paper,
however,
we deal with a deterministic integrand and demand some effective ways to compute the expectation of   multiple Stratonovich
integral  of deterministic integrands. To this end let us
recall  \cite[p.201, Lemma 5.2.6]{MR}  that
\begin{align}\label{M-8}
	\left\{\begin{array}{ll}\displaystyle
		\E\prod_{k=1}^{2n}g_k=\sum_{{\cal D}\in\Pi_n}\prod_{(j,k)\in{\cal D}}\E g_jg_k\\\\
		\displaystyle\E\prod_{k=1}^{2n-1}g_k=0\,, \end{array}\right.
\end{align}
where $(g_1,\cdots,g_{2n})$ is a
mean zero normal vector,
and $\Pi_n$ is the set  of all pair partitions  of $\{1,2,\cdots, 2n\}$. As a side remark,
$\#(\Pi_n)={(2n)!\over 2^n n!}$.
Applying  (\ref{M-8})  to $g_k=W_{\varepsilon_k}(x_k)$ in the case of deterministic integrand $f$,
and taking the $\varepsilon$-limit, we have
\begin{align}\label{M-9}
	\E\left[ \int_{(\R^d)^{2n-1}}f(x_1,\cdots, x_{2n-1})W(dx_1)\cdots W(dx_{2n-1})\right] =0
\end{align}
and
\begin{align}\label{M-10}
	&\E\left[\int_{(\R^d)^{2n}} f(x_1,\cdots, x_{2n})W(dx_1)\cdots W(dx_{2n})\right]\\
	&\qquad\qquad =\sum_{{\cal D}\in\Pi_n}
	\int_{(\R^d)^{2n}}\bigg(\prod_{(j,k)\in {\cal D}}\gamma(x_j-x_k)\bigg)
	f(x_1,\cdots, x_{2n})dx_1\cdots dx_{2n}\nonumber
\end{align}
under the Stratonovich integrability  on the left hand sides.
In particular, the expectation of a $(2n)$-multiple Stratonovich
integral is non-negative if the integrand is non-negative.

Since   Dalang's condition (\ref{intro-6})
encompasses the cases where the covariance function $\gamma(\cdot)$ exists only as a
generalized function (e.g., $\gamma(\cdot)=\delta_0(\cdot)$ in $d=1$), the meaning of the multiple integral on
the right hand side of (\ref{M-10}) needs to be clarified.
Indeed, by (\ref{M-8})
$$
\begin{aligned}
	&\E\int_{(\R^d)^{2n}} f(x_1,\cdots, x_{2n})\prod_{k=1} ^{2n} \dot{W}_\epsilon
	(x_k)\\
	&= \int_{(\R^d)^{2n}}\bigg(\prod_{(j,k)\in {\cal D}}\gamma_{2\epsilon}(x_j-x_k)\bigg)
	f(x_1,\cdots, x_{2n})dx_1\cdots dx_{2n}
\end{aligned}
$$
where, we recall 
$$
\gamma_\epsilon(x)=\int_{\R^d}\gamma(y)p_\epsilon(x-y)dy\,, \hskip.2in\epsilon>0,\hskip.1in x\in\R^d\,. 
$$
Inspired by (\ref{M-6}), we therefore define 
\begin{align}\label{M-15}
	&\int_{(\R^d)^{2n}}\bigg(\prod_{(j,k)\in {\cal D}}\gamma(x_j-x_k)\bigg)
	f(x_1,\cdots, x_{2n})dx_1\cdots dx_{2n}\\
	&\buildrel \Delta\over =\lim_{\epsilon\to 0^+}\int_{(\R^d)^{2n}}\bigg(\prod_{(j,k)\in {\cal D}}\gamma_{2\epsilon}(x_j-x_k)\bigg)
	f(x_1,\cdots, x_{2n})dx_1\cdots dx_{2n}\nonumber
\end{align}
whenever the limit exists.

According to Theorem \ref{t.6.2} and Remark \ref{r.6.3}, the ${\cal L}^2$-convergence in
(\ref{M-6}), Definition \ref{def-2.2} implies the 
${\cal L}^p$-convergence for any $p\in [2, \infty)$. Consequently, for any
integers $l_1,\cdots, l_m\ge 1$ and the $l_j$-multiple variate functions
$f_j$ ($1\le j\le m$), the Stratonovich integrability of
$f_1,\cdots, f_m$
implies the Strotonovich integrability of $f_1\otimes\cdots \otimes f_m$
and
\begin{align}\label{M-16}
	S_{l_1+\cdots +l_m}\big(f_1\otimes\cdots \otimes f_m\big)
	=\prod_{j=1}^mS_{l_j}(f_j)\,. 
\end{align}
According to (\ref{M-9}), in particular,
\begin{align}\label{M-17}
	\E\prod_{j=1}^mS_{l_j}(f_j)=0\hskip.1in
	\hbox{whenever $l_1+\cdots +l_m$ is odd}\,. 
\end{align}    

Given two Stratonovich integrable functions $f(x_1,\cdots, x_n)$
and $g(x_1,\cdots,x_n)$, by (\ref{M-10}) and (\ref{M-16}) (with $m=2$),
\begin{align}\label{M-18}  
	& \E S_n(f)S_n(g)\\
	& =\sum_{{\cal D}\in\Pi_n}\int_{(\R^d)^{2n}}dx_1\cdots dx_{2n}
	\bigg(\prod_{(j,k)\in {\cal D}}\gamma(x_j-x_k)\bigg)f(x_1,\cdots, x_n)g(x_{n+1},\cdots, x_{2n}).\nonumber
\end{align}

To end this section   we take the  chance to address an inconvenient fact from (\ref{M-12}) where $G(t,x)$ is
defined as a measure rather than a function in $d=3$ dimensional Euclidean space. In this case,  we can combine 
$g_n(x_1,\cdots, x_n, t,x)$ and $dx_1\cdots dx_n$ together to have that 
$$
\begin{aligned}
	g_n(x_1,\cdots, x_n, t,x)dx_1\cdots dx_n&=\int_{[0,t]_<^n}\bigg(\prod_{k=1}^n\frac{1}{ 4\pi (s_k-s_{k-1})}
	\sigma_{s_k-s_{k-1}}(x_{k-1}, dx_k)\bigg)ds_1\cdots ds_n\\
	&\buildrel\Delta\over =\mu_n^{t,x}(dx_1\cdots dx_n)
\end{aligned}
$$
defines a measure on $(\R^3)^n$, where $\sigma_t(x, dy)$ represents
the surface measure on the sphere $\{y\in\R^3;\hskip.05in\vert y-x\vert=t\}$. 
For example,  in defining $S_n\big(g_n(\cdot,t,x)\big)$ by (\ref{M-6})  we use the
convention
$$
\int_{(\R^3)^n}g_n(x_1,\cdots, x_n, t,x)\bigg(\prod_{k=1}^n\dot{W}_\varepsilon(x_k)\bigg)dx_1\cdots dx_n
=\int_{(\R^3)^n}\bigg(\prod_{k=1}^n\dot{W}_\varepsilon(x_k)\bigg)\mu_n^{t,x}(dx_1\cdots dx_n)  \,. 
$$
It will be verified  in the future that as $\varepsilon\downarrow 0+$, the above sequence converges 
in $L^2(\Omega, \mathcal{F}, \mathbb{P})$ and the limit is denoted still by 
$$
\int_{(\R^3)^n}g_n(x_1,\cdots, x_n, t,x) {W} (dx_1)\cdots {W}(dx_n)
$$
with
\begin{align}\label{M-19}
	\E\bigg[\int_{(\R^3)^{2n}}&g_n(x_1,\cdots, x_n, t,x)W(dx_1)\cdots W(dx_{n})\bigg]^2\\
	&=\sum_{{\cal D}\in\Pi_n}
	\int_{(\R^3)^{2n}}\bigg(\prod_{(j,k)\in {\cal D}}\gamma(x_j-x_k)\bigg)
	\mu_n^{t,x}(dx_1\cdots dx_{n}) \mu_n^{t,x}(dx_{n+1},\cdots, dx_{2n})
	\,,   \nonumber
\end{align}
and the integral on the right hand side of (\ref{M-19}) will be justified (Lemma \ref{int}) together with dimensions $d=1,2$ by the approximation procedure
proposed  in (\ref{M-15}).

\section{Stratonovich moments}\label{S}

In the following discussion, $B(t), B_1(t), B_2(t),\cdots$ are independent $d$-dimensional
Brownian motions. We assume independence between $\dot{W}$ and the Brownian motions
and use the notation $\E_x$ for the expectation with respect to the Brownian motions
with starting point $x$.
We adopt the notation $\varepsilon=(\varepsilon_1,\cdots,\varepsilon_n)$ and $\epsilon'=(\epsilon_{n+1},\cdots, \epsilon_{2n})$ for $\varepsilon_1,\cdots,\varepsilon_{2n}>0$
and set
\begin{align}\label{P-1}
	S_{n,\varepsilon}\big(g_n(\cdot, t,x)\big)=\int_{(\R^d)^n}g_n(x_1,\cdots,x_n, t,x)\bigg(\prod_{k=1}^n \dot{W}_{\varepsilon_k}(x_k)\bigg)
	dx_1\cdots dx_n\,. 
\end{align}
For any pair partition ${\cal D}\in\Pi_n$, set
\begin{align}\label{P-2}
	F_{\epsilon,\epsilon'}^{\cal D}(t_1,t_2)&=
	\int_{(\R^d)^{2n}}dx_1\cdots dx_{2n}\bigg(\prod_{(j,k)\in {\cal D}}
	\gamma_{\epsilon_j+\epsilon_k}(x_j-x_k)\bigg)\\
	&\qquad \times g_n(x_1,\cdots, x_n, t,0)g_n(x_{n+1},\cdots, x_{2n}, t,0)\,. \nonumber
\end{align}
Again, $d=1,2,3$.

\subsection{Strantonovich moment representation}\label{S-2}

\begin{lemma}\label{L-1} Let $n=1,2,\cdots$.
	Under  Dalang's condition (\ref{intro-6}),
	\item[(i)]  For any $n\ge 1$, $\epsilon_1,\cdots, \epsilon_{n}>0$ and $\lambda>0$
	\begin{align}\label{P-3}
		&\int_0^\infty e^{-\lambda t}S_{n,\epsilon}\big(g_n(\cdot, t,x)\big) dt\\
		& =\frac{\lambda}{ 2}\Big(\frac{1}{ 2}\Big)^n
		\int_0^\infty\exp\Big\{-\frac{\lambda^2}{ 2}t\Big\}\E_x\int_{[0,t]_<^n}ds_1\cdots ds_n\prod_{k=1}^n
		\dot{W}_{\epsilon_k}\big(B(s_k)\big)\hskip.1in a.s.\nonumber
	\end{align}
	\item[(ii)]  For any $\lambda_1,\lambda_2>0$
	\begin{align}\label{P-4}
		&\int_0^\infty\int_0^\infty e^{-\lambda_1 t_1-\lambda_2 t_2}F_{\epsilon,\epsilon'}^{\cal D}(t_1,t_2)dt_1dt_2\\
		& \le{\lambda_1\lambda_2\over 4}\Big({1\over 2}\Big)^{2n}
		\int_0^\infty\!\!\int_0^\infty dt_1dt_2
		\exp\Big\{-{\lambda_1^2 t_1+\lambda_2^2 t_2\over 2}\Big\}\nonumber\\
		&\qquad \times\E_0 \int_{[0,t_1]_<^n\times[0,t_2]_<^n}ds_1\cdots ds_{2n}\nonumber\prod_{(j,k)\in{\cal D}}
		\gamma\big(B_{v(j)}(s_j)-B_{v(k)}(s_k)\big)\,, \nonumber
	\end{align}
	where the map $v$: $\{1, 2,\cdots, 2n\}\longrightarrow\{1,2\}$ is defined as:
	$v(k)=1$ for $1\le k\le n$ and $v(k)=2$ for $n+1\le k\le n$.
	\item (iii) For any $\lambda_1,\lambda_2>0$
	\begin{align}\label{P-5}
		&\lim_{\epsilon,\epsilon\to 0}\int_0^\infty\!\!\int_0^\infty e^{-\lambda_1 t_1-\lambda_2 t_2}
		F^{\cal D}_{\epsilon,\epsilon'}(t_1,t_2)dt_1dt_2\\
		&={\lambda_1\lambda_2\over 4}\Big({1\over 2}\Big)^{2n}
		\int_0^\infty\!\!\int_0^\infty dt_1dt_2
		\exp\Big\{-{\lambda_1^2 t_1+\lambda_2^2 t_2\over 2}\Big\}\nonumber\\
		&\qquad \times \E_0 \int_{[0,t_1]_<^n\times[0,t_2]_<^n}ds_1\cdots ds_{2n}\prod_{(j,k)\in{\cal D}}
		\gamma\big(B_{v(j)}(s_j)-B_{v(k)}(s_k)\big)\,. \nonumber
	\end{align}
\end{lemma}

\begin{remark}\label{R-1} Under  Dalang's condition (\ref{intro-6}), the intersection local times (Lemma A.1, \cite{Chen-2})
	$$
	\int_0^{t_1}\!\int_0^{t_2}\gamma\big(B(s)-B(r)\big)dsdr\hskip.1in\hbox{and}\hskip.1in
	\int_0^{t_1}\!\int_0^{t_2}\gamma\big(B_1(s)-B_2(r)\big)dsdr\,, \hskip.2in t_1,t_2>0
	$$
	are properly defined, so are the multiple time integral on the right hand sides of (\ref{P-4}) and (\ref{P-5}) in the spirit
	of Fubini's theorem. By Lemma \ref{L-10},
	the moments of the intersection local times have (at most) polynomial increasing rate in $t_1, t_2$. Consequently,
	the right hand sides of (\ref{P-4}) and (\ref{P-5}) are finite for any $\lambda_1,\lambda_2>0$.
\end{remark}

\proof The reason behind (\ref{P-3}) is the  simple fact that
\begin{align}\label{P-6} 
	\int_0^\infty e^{-\lambda t}G(t,x)dt={1\over 2}\int_0^\infty e^{-\lambda^2t/2}p(t,x)dt\,, \hskip.2in x\in\R^d
\end{align}
for any $\lambda>0$, where $p(t,x)$ is the density of $B(t)$:
$$
p(t,x)={1\over (2\pi t)^{d/2}}\exp\Big\{-{\vert x\vert^2\over 2t}\Big\}\,, \hskip.2in (t,x)\in\R^+\times\R^d\,. 
$$
Indeed,   both sides have  the same  Fourier transform
$$
\begin{aligned}
	&\int_{\R^d}e^{i\xi\cdot x}\bigg[\int_0^\infty e^{-\lambda t}G(t,x)dt\bigg]dx
	=\int_0^\infty e^{-\lambda t}{\sin\vert\xi\vert t\over \vert\xi\vert}dt={1\over \lambda^2
		+\vert\xi\vert^2}\\
	&\qquad  ={1\over 2}\int_0^\infty\!\!e^{-\lambda^2t/2}\exp\Big\{-{1\over 2}\vert\xi\vert^2 t\Big\}dt
	=\int_{\R^d}\!e^{i\xi\cdot x}\bigg[{1\over 2}\int_0^\infty e^{-\lambda^2 t/2}p(t,x)dt\bigg]dx
\end{aligned}
$$
for every $\xi\in\R^d$.

Recall the identity (Lemma 2.2.7, p.39 in \cite{Chen-1})
\begin{align}\label{P-7}
	\int_0^\infty e^{-\lambda t}\int_{[0,t]_<^n} ds_1\cdots ds_n
	\prod_{k=1}^n\varphi(s_k-s_{k-1})
	=\lambda^{-1}\prod_{k=1}^n\int_0^\infty \varphi(t)e^{-\lambda t}dt
\end{align}
with the convention $s_0=0$. Using it twice,
\begin{align}\label{P-8}
	&\int_0^\infty e^{-\lambda t} g_n(x_1,\cdots, x_n, t,x)dt\\
	&=\int_0^\infty dt e^{-\lambda t}
	\int_{[0,t]_<^n}ds_1\cdots ds_n\bigg(\prod_{k=1}^n G(s_k-s_{k-1}, x_k-x_{k-1})
	\bigg)\nonumber\\
	&=\lambda^{-1}\prod_{k=1}^n \int_0^\infty e^{-\lambda t}G(t, x_k-x_{k-1})dt
	=\Big(\frac12\Big)^n\lambda^{-1}\prod_{k=1}^n \int_0^\infty e^{-\lambda^2 t/2}p(t, x_k-x_{k-1})dt\nonumber\\
	&=\frac{\lambda}{ 2}\Big(\frac12\Big)^n\int_0^\infty dt e^{-\lambda^2 t/2}
	\int_{[0,t]_<^n}ds_1\cdots ds_n\bigg(\prod_{k=1}^n p(s_k-s_{k-1}, x_k-x_{k-1})\bigg)\,. \nonumber
\end{align}   

Hence,
$$
\begin{aligned}
	&\int_0^\infty e^{-\lambda t}S_{n,\varepsilon}\big(g_n(\cdot, t,x)\big)dt\\
	&=\int_0^\infty dt e^{-\lambda t}\int_{(\R^d)^n}dx_1\cdots dx_ng_n(x_1,\cdots, x_n, t,x)\bigg(\prod_{k=1}^n \dot{W}_{\varepsilon_k}(x_k)\bigg)\\
	&=\frac{\lambda}{ 2}\Big(\frac12\Big)^n\int_0^\infty dt \exp\Big\{-\frac{\lambda^2}{ 2}t\Big\}
	\int_{[0,t]_<^n}ds_1\cdots ds_n\\
	&\qquad \times\int_{(\R^d)^n}dx_1\cdots dx_n\bigg(\prod_{k=1}^n p(s_k-s_{k-1}, x_k-x_{k-1})\bigg)
	\bigg(\prod_{k=1}^n \dot{W}_{\varepsilon_k}(x_k)\bigg)\,. 
\end{aligned}
$$
Given $(s_1,\cdots, s_n)\in [0,t]_<^n$, the random vector $\big(B(s_1),\cdots,B(s_n)\big)$ has
the joint density
$$
f_{s_1,\cdots, s_n}(x_1,\cdots, x_n)\buildrel \Delta\over
=\prod_{k=1}^n p(s_k-s_{k-1}, x_k-x_{k-1})\,. 
$$
So we have
$$
\begin{aligned}
	&\int_{(\R^d)^n}d{\bf x}\bigg(\prod_{k=1}^n p(s_k-s_{k-1}, x_k-x_{k-1})\bigg)
	\bigg(\prod_{k=1}^n \dot{W}_{\varepsilon_k}(x_k)\bigg)=\E_x\prod_{k=1}^n \dot{W}_{\varepsilon_k}\big(B(s_k)\big)\,. 
\end{aligned}
$$
This completes the proof (\ref{P-3}).

By  (\ref{P-8})  we have 
$$
\begin{aligned}
	&\int_0^\infty\!\!\int_0^\infty e^{-\lambda_1 t_1-\lambda_2 t_2}
	F^{\cal D}_{\epsilon,\epsilon'}(t_1,t_2)dt_1dt_2\\
	&={\lambda_1\lambda_2\over 4}\Big({1\over 2}\Big)^{2n}
	\int_0^\infty\!\!\int_0^\infty dt_1dt_2
	\exp\Big\{-{\lambda_1^2 t_1+\lambda_2^2 t_2\over 2}\Big\}
	\int_{[0,t_1]_<^n\times[0,t_2]_<^n}ds_1\cdots ds_{2n}\\
	&\qquad \times\int_{(\R^d)^{2n}}dx_1\cdots dx_{2n}
	\bigg(\prod_{(j,k)\in {\cal D}}
	\gamma_{\epsilon_j+\epsilon_k}(x_j-x_k)\bigg)\\
	&\qquad \times\bigg(p(s_1,x_1)\prod_{k=2}^np(s_k-s_{k-1}, x_k-x_{k-1})\bigg)\\
	&\qquad \times
	\bigg(p(s_{n+1},x_{n+1})\prod_{k=n+2}^{2n}p(s_k-s_{k-1}, x_k-x_{k-1})\bigg)\,. 
\end{aligned}
$$
For fixed $(s_1,\cdots, s_{2n})$, the function
$$
\begin{aligned}
	&f(x_1,\cdots, x_{2n})\\
	&=\bigg(p(s_1,x_1)\prod_{k=2}^np(s_k-s_{k-1}, x_k-x_{k-1})\bigg)
	\bigg(p(s_{n+1},x_{n+1})\prod_{k=n+2}^{2n}p(s_k-s_{k-1}, x_k-x_{k-1})\bigg)
\end{aligned}
$$
is the density of the random vector
$\big(B_1(s_1),\cdots, B_1(s_n); B_2(s_{n+1}),\cdots B_2(s_{2n})\big)$.
We have
$$
\begin{aligned}
	&\int_0^\infty\!\!\int_0^\infty e^{-\lambda_1 t_1-\lambda_2 t_2}
	F^{\cal D}_{\epsilon, \epsilon'}(t_1,t_2)dt_1dt_2\\
	&={\lambda_1\lambda_2\over 4}\Big({1\over 2}\Big)^{2n}
	\int_0^\infty\!\!\int_0^\infty dt_1dt_2
	\exp\Big\{-{\lambda_1^2 t_1+\lambda_2^2 t_2\over 2}\Big\}
	\int_{[0,t_1]_<^n\times[0,t_2]_<^n}ds_1\cdots ds_{2n}\nonumber \\
	&\qquad   \times\E_0\prod_{(j,k)\in{\cal D}}
	\gamma_{\epsilon_j+\epsilon_k}\big(B_{v(j)}(s_j)-B_{v(k)}(s_k)\big)\,. 
\end{aligned}
$$
By Fourier transform
$$
\begin{aligned}
	&\E_0\prod_{(j,k)\in{\cal D}}\gamma_{\epsilon_j+\epsilon_k}
	\big(B_{\nu(j)}(s_j)-B_{v(k)}(s_k)\big)\\
	&=\int_{(\R^d)^n}\bigg(\prod_{(j,k)\in{\cal D}}
	\mu(d\xi_{j,k})\bigg)\exp\bigg\{-\sum_{(j,k)\in{\cal D}}
	{\epsilon_j+\epsilon_k\over 2}\vert\xi_{j,k}\vert^2\bigg\}\nonumber\\
	&\qquad  \times\E_0\exp\bigg\{i\sum_{(j,k)\in{\cal D}}\xi_{j,k}\cdot
	\big(B_{v(j)}(s_j)-B_{v(k)}(s_k)\big)\bigg\}\\
	&=\int_{(\R^d)^n}\bigg(\prod_{(j,k)\in{\cal D}}
	\mu(d\xi_{j,k})\bigg)\exp\bigg\{-\sum_{(j,k)\in{\cal D}}
	{\epsilon_j+\epsilon_k\over 2}\vert\xi_{j,k}\vert^2\bigg\}\nonumber\\
	&\qquad  \times\exp\bigg\{-{1\over 2}\Var\bigg(\sum_{(j,k)\in{\cal D}}\xi_{j,k}\cdot
	\big(B_{v(j)}(s_j)-B_{v(k)}(s_k)\big)\bigg)
	\bigg\}\,. 
\end{aligned}
$$
We have
\begin{align}\label{P-9}
	&\int_0^\infty\!\!\int_0^\infty e^{-\lambda_1 t_1-\lambda_2 t_2}
	F^{\cal D}_{\epsilon,\epsilon'}(t_1,t_2)dt_1dt_2\\
	&={\lambda_1\lambda_2\over 4}\Big({1\over 2}\Big)^{2n}
	\int_0^\infty\!\!\int_0^\infty dt_1dt_2
	\exp\Big\{-{\lambda_1^2 t_1+\lambda_2^2 t_2\over 2}\Big\}
	\int_{[0,t_1]_<^n\times[0,t_2]_<^n}ds_1\cdots ds_{2n}\nonumber\\
	&\qquad \times\int_{(\R^d)^n}\bigg(\prod_{(j,k)\in{\cal D}}
	\mu(d\xi_{j,k})\bigg)\exp\bigg\{-\sum_{(j,k)\in{\cal D}}
	{\epsilon_j+\epsilon_k\over 2}\vert\xi_{j,k}\vert^2\bigg\}\nonumber\\
	&\qquad  \times\exp\bigg\{-{1\over 2}\Var\bigg(\sum_{(j,k)\in{\cal D}}\xi_{j,k}\cdot
	\big(B_{v(j)}(s_j)-B_{v(k)}(s_k)\big)\bigg)\bigg\}
	\,. \nonumber
\end{align}
Therefore
$$
\begin{aligned}
	&\int_0^\infty\!\!\int_0^\infty e^{-\lambda_1 t_1-\lambda_2 t_2}
	F^{\cal D}_{\epsilon,\epsilon'}(t_1,t_2)dt_1dt_2\\
	&\le{\lambda_1\lambda_2\over 4}\Big({1\over 2}\Big)^{2n}
	\int_0^\infty\!\!\int_0^\infty dt_1dt_2
	\exp\Big\{-{\lambda_1^2 t_1+\lambda_2^2 t_2\over 2}\Big\}
	\int_{[0,t_1]_<^n\times[0,t_2]_<^n}ds_1\cdots ds_{2n}\\
	&\qquad  \times\int_{(\R^d)^n}\bigg(\prod_{(j,k)\in{\cal D}}
	\mu(d\xi_{j,k})\bigg)\exp\bigg\{-{1\over 2}\Var\bigg(\sum_{(j,k)\in{\cal D}}\xi_{j,k}\cdot
	\big(B_{v(j)}(s_j)-B_{v(k)}(s_k)\big)\bigg)\bigg\}\\
	&=\int_{[0,t_1]_<^n\times[0,t_2]_<^n}ds_1\cdots ds_{2n}\prod_{(j,k)\in{\cal D}}
	\gamma\big(B_{v(j)}(s_j)-B_{v(k)}(s_k)\big)\,. 
\end{aligned}
$$
We have proved (\ref{P-4}).  Finally, taking limit in (\ref{P-9}) 
$$
\begin{aligned}
	&\lim_{\epsilon,\epsilon'\to 0}\int_0^\infty\!\!\int_0^\infty e^{-\lambda_1 t_1-\lambda_2 t_2}
	F^{\cal D}_{\epsilon,\epsilon'}(t_1,t_2)dt_1dt_2\\
	&={\lambda_1\lambda_2\over 4}\Big({1\over 2}\Big)^{2n}
	\int_0^\infty\!\!\int_0^\infty dt_1dt_2
	\exp\Big\{-{\lambda_1^2 t_1+\lambda_2^2 t_2\over 2}\Big\}
	\int_{[0,t_1]_<^n\times[0,t_2]_<^n}ds_1\cdots ds_{2n}\\
	&\qquad  \times\int_{(\R^d)^n}\bigg(\prod_{(j,k)\in{\cal D}}
	\mu(d\xi_{j,k})\bigg)\exp\bigg\{-{1\over 2}\Var\bigg(\sum_{(j,k)\in{\cal D}}\xi_{j,k}\cdot
	\big(B_{v(j)}(s_j)-B_{v(k)}(s_k)\big)\bigg)\bigg\}\\
	&=\int_{[0,t_1]_<^n\times[0,t_2]_<^n}ds_1\cdots ds_{2n}\prod_{(j,k)\in{\cal D}}
	\gamma\big(B_{v(j)}(s_j)-B_{v(k)}(s_k)\big)
\end{aligned}
$$
which leads to (\ref{P-5}).
\qed

\begin{theorem}\label{th-4}  
	Under   Dalang's condition (\ref{intro-6}), the function $g_n(\cdot, t, x)$ defined in (\ref{M-14}) is Stratonovich 
	integrable in the sense of Definition \ref{def-2.2}. Furthermore,
	\begin{align}\label{P-10}
		&\int_0^\infty e^{-\lambda t}S_n\big(g_n(\cdot, t,x)\big) dt
		=\frac{1}{ n!}\frac{\lambda}{ 2}\Big(\frac{1}{ 2}\Big)^n
		\int_0^\infty\exp\Big\{-\frac{\lambda^2}{ 2}t\Big\}
		\E_x \bigg[\int_0^t \dot{W}\big(B(s)\big)ds\bigg]^ndt
	\end{align}
	almost surely  for any $\lambda>0$.
\end{theorem}  

\begin{proof}  We first explain the time integral appearing on the right hand side of (\ref{P-10}).   It is defined
	as
	$$
	\int_0^t\dot{W}\big(B(s)\big)ds\buildrel\Delta\over
	=\lim_{\varepsilon\to 0^+}
	\int_0^t\dot{W}_\varepsilon\big(B(s)\big)ds\hskip.2in
	\hbox{in ${\cal L}^2(\Omega, {\cal F}, \P_x\otimes\P)$}\,, 
	$$
	where the existence of the limit on the right hand is
	established in Lemma A.1, \cite{Chen-2} under  
	Dalang's condition (\ref{intro-6}). Conditioning on the Brownian motion $B$,
	it is
	a mean zero normal random variable with the variance
	$$
	\int_0^t\!\!\int_0^t\gamma(B_s-B_r)dsdr 
	$$
	whose distribution does not depend on the starting point $x$ of the Brownian motion. So we have
	\begin{align}\label{P-11}
		&\E\otimes\E_x\bigg[\int_0^t\dot{W}\big(B(s)\big)ds\bigg]^n  \\
		&=\begin{cases} \frac{n!}{ 2^{n/2}(n/2)!}
			\E_0\bigg[\int_0^t\!\!\int_0^t\gamma(B_s-B_r)dsdr\bigg]^{\frac{n}{ 2}}&\quad  \hbox{when $n$ is even}\,;\nonumber \\ \\
			0 &\quad \hbox{when $n$ is odd}\,. \end{cases} 
	\end{align}
	The   above $n$-th moment is finite ((\ref{A-1}), Lemma \ref{L-10} below) for all $n=1,2,\cdots$.
	Consequently, the quenched moment
	$$
	\E_x\bigg[\int_0^t\dot{W}\big(B(s)\big)ds\bigg]^n
	$$
	exists almost surely. In addition, the bound provided in (\ref{A-1}) in the
	Lemma \ref{L-10} below makes the right hand side of (\ref{P-10}) well-defined
	for any $\lambda>0$.

	Taking $\epsilon_1=\cdots =\epsilon_{2n}=\delta$ in (\ref{P-3}), we have
	$$
	\begin{aligned}
		&\int_0^\infty e^{-\lambda t}S_{n,\epsilon}\big(g_n(\cdot, t,x)\big) dt
		=\frac{\lambda}{ 2}\Big(\frac{1}{ 2}\Big)^n{1\over n!}
		\int_0^\infty\exp\Big\{-\frac{\lambda^2}{ 2}t\Big\}\E_x\bigg[\int_0^t\dot{W}_\delta\big(B(s)\big)ds\bigg]^ndt\,. 
	\end{aligned}
	$$
	We now let $\delta\to 0^+$ on  both sides. Notice that
	$$
	\lim_{\delta\to 0^+}\E_x\bigg[\int_0^t\dot{W}_\delta\big(B(s)\big)ds\bigg]^n=\E_x\bigg[\int_0^t\dot{W}\big(B(s)\big)ds\bigg]^n\hskip.1in
	\hbox{in ${\cal L}^2(\Omega, {\cal F}, \P)$}\,. 
	$$
	In addition, by Cauchy-Schwartz and Jensen inequalities
	$$
	\begin{aligned}
		&\E\bigg\{\E_x\bigg[\int_0^t\dot{W}_\delta\big(B(s)\big)ds\bigg]^n\bigg\}^2\le\E_0\otimes\E\bigg[\int_0^t\dot{W}_\delta\big(B(s)\big)ds\bigg]^{2n}\\
		&\le\int_{\R^d}p_\epsilon(y)\E_0\otimes\E\bigg[\int_0^t\dot{W}\big(y+B(s)\big)ds\bigg]^{2n}dy=\E_0\otimes\E\bigg[\int_0^t\dot{W}\big(B(s)\big)ds\bigg]^{2n}\\
		&=\E_0\bigg[\int_0^t\!\int_0^t\gamma\big(B(s)-B(r)\big)dsdr\bigg]^n\,. 
	\end{aligned}
	$$
	By Lemma \ref{L-10}, the right hand side has at most a polynomial
	increasing rate in $t$. By the dominated convergence theorem, we have 
	$$
	\lim_{\delta\to 0^+}\int_0^\infty\exp\Big\{-\frac{\lambda^2}{ 2}t\Big\}\E_x\bigg[\int_0^t\dot{W}_\delta\big(B(s)\big)ds\bigg]^ndt=
	\int_0^\infty\exp\Big\{-\frac{\lambda^2}{ 2}t\Big\}\E_x\bigg[\int_0^t\dot{W}\big(B(s)\big)ds\bigg]^ndt
	$$
	in  ${\cal L}^2(\Omega, {\cal F}, \P)$.

	The Stratonovich integrability of $g_n(\cdot, t,x)$ shall be established
	in Theorem \ref{prop} below to make sense of left hand side of (\ref{P-10}).
	By stationarity in $x$, all we need is the following convergence
	$$
	\lim_{\delta\to 0^+}\int_0^\infty e^{-\lambda t}S_{n,\epsilon}\big(g_n(\cdot, t,0)\big) dt
	=\int_0^\infty e^{-\lambda t}S_n\big(g_n(\cdot, t,0)\big) dt \hskip.1in
	\hbox{in ${\cal L}^2(\Omega, {\cal F}, \P)$}.
	$$
	This is given in Part (ii), Theorem \ref{prop}.   \end{proof}

\begin{corollary}\label{co-1}
	Assume  Dalang's condition (\ref{intro-6}). Let $p\ge 1$ and $n\ge 1$ be any integers.
	Given $\lambda_1,\cdots,\lambda_p>0$,
	\begin{align}\label{P-12}
		&\int_{(\R^+)^p}dt_1\cdots dt_p
		\exp\Big\{-\sum_{j=1}^p\lambda_jt_j\Big\}\sum_{l_1+\cdots+l_p=2n}
		\E \prod_{j=1}^p  S_{l_j}\big(g_{l_j}(\cdot, t_j,0)\big)\\
		&=\Big(\frac{1}{ 2}\Big)^{3n}\frac{1}{ n!}\Big(\prod_{j=1}^p\frac{\lambda_j}{ 2}\Big)
		\int_{(\R^+)^p}dt_1\cdots dt_p\exp\Big\{-\frac{1}{ 2}\sum_{j=1}^p\lambda_j^2 t_j\Big\}\nonumber\\
		&\qquad \times\E_0\bigg[\sum_{j,k=1}^p\int_0^{t_j}\!\int_0^{t_k}
		\gamma\big(B_j(s)-B_k(r)\big)dsdr\bigg]^n\,, \nonumber
	\end{align}
	where $B_1(t),\cdots, B_p(t)$ are independent $d$-dimensional Brownian motions starting at 0.
	
\end{corollary}

\begin{proof}  By Theorem \ref{th-4},
	$$
	\begin{aligned}
		&\int_{(\R^+)^p}dt_1\cdots dt_p\exp\Big\{-\sum_{j=1}^p\lambda_jt_j\Big\}\sum_{l_1+\cdots+l_p=2n}
		\prod_{j=1}^pS_{l_j}\big(g_{l_j}(\cdot, t,0)\big)\\
		&=\sum_{l_1+\cdots+l_p=2n}\prod_{j=1}^p\int_0^\infty e^{-\lambda_jt}S_{l_j}\big(g_{l_j}(\cdot, t_j,0)\big)dt\\
		&=\sum_{l_1+\cdots+l_p=2n}\bigg(\prod_{j=1}^p\frac{\lambda_j}{ 2}\Big(\frac{1}{ 2}\Big)^{l_j}\bigg)
		\prod_{j=1}^p\frac{1}{l_j!}\int_0^\infty\!\! dt e^{-\lambda_j^2 t/2}
		\!\E_0\bigg[\!\int_0^t \dot{W}\big(B(s)\big)ds\bigg]^{l_j}\\
		&=\Big(\frac{1}{2}\Big)^{2n}\Big(\prod_{j=1}^p\frac{\lambda_j}{ 2}\Big)
		\int_{(\R^+)^p}dt_1\cdots dt_p\exp\Big\{-\frac{1}{ 2}\sum_{j=1}^p\lambda_j^2t_j\Big\}
		\frac{1}{ (2n)!}\E_0\bigg[\sum_{j=1}^p\int_0^{t_j}\dot{W}\big(B_j(s)\big)ds\bigg]^{2n}
		\,, \end{aligned}
	$$
	where the last step follows from Newton's multi-nominal  formula. 
	By the fact that conditioning on the Brownian motions,
	$$
	\sum_{j=1}^p\int_0^{t_j}\dot{W}\big(B_j(s)\big)ds
	$$
	is normal with zero mean and the variance
	$$
	\sum_{j,k=1}^p\int_0^{t_j}\!\int_0^{t_k}\gamma\big(B_j(s)-B_k(r)\big)dsdr
	$$
	we have
	\begin{align}\label{P-13}
		\E\bigg[\sum_{j=1}^p\int_0^{t_j}\dot{W}\big(B_j(s)\big)ds\bigg]^{2n}
		=\frac{(2n)!}{2^n n!}\bigg[\sum_{j,k=1}^p\int_0^{t_j}\!\int_0^{t_k}\gamma\big(B_j(s)-B_k(r)\big)dsdr\bigg]^n\,. 
	\end{align}
	Thus, we have proved (\ref{P-12}). 
\end{proof}

\begin{corollary} \label{co-2}
	
	\begin{enumerate}
		\item[(1)] For any $\lambda_1,\lambda_2>0$
		\begin{align}\label{P-14}
			&\int_0^\infty \!\!\int_0^\infty dt_1dt_2 e^{-\lambda_1t_1-\lambda_2t_2}\E S_{n}\big(g_n(\cdot, t_1, 0)\big)S_{n}\big(g_n(\cdot, t_2, 0)\big)\\
			&=\Big({1\over 4}\Big)^n{\lambda_1\lambda_2\over 4}{1\over (n!)^2}\E\otimes\E_0\bigg[\int_0^{t_1}\dot{W}\big(B_1(s)ds\bigg]^n\bigg[\int_0^{t_2}\dot{W}\big(B_2(s)\big)ds\bigg]^n\,. 
			\nonumber
		\end{align}
		\item[(2)] For any $\epsilon=(\epsilon_1,\cdots,\epsilon_n)$ and $\epsilon'=(\epsilon_{n+1},\cdots,\epsilon_{2n})$
		\begin{align}\label{P-15}
			&\int_0^\infty \!\!\int_0^\infty dt_1dt_2 e^{-\lambda_1t_1-\lambda_2t_2}\E S_{n,\epsilon}\big(g_n(\cdot, t_1, 0)\big)S_{n,\epsilon'}\big(g_n(\cdot, t_2, 0)\big)\\
			&\le \int_0^\infty \!\!\int_0^\infty dt_1dt_2 e^{-\lambda_1t_1-\lambda_2t_2}\E S_{n}\big(g_n(\cdot, t_1, 0)\big)S_{n}\big(g_n(\cdot, t_2, 0)\big)\,. \nonumber
		\end{align}
	\end{enumerate}
\end{corollary}

\proof (\ref{P-14}) is a direct consequence of Theorem \ref{th-4}.
By the definition of $S_{n,\epsilon}\big((g_n(\cdot, t,x)\big)$ given in (\ref{P-1}),
$$
\begin{aligned}
	&\E S_{n,\epsilon}\big(g_n(\cdot, t_1, 0)\big)S_{n,\epsilon'}\big(g_n(\cdot, t_2, 0)\big)\\
	&=\E\int_{(\R^d)^{2n}}dx_1\cdots dx_{2n}g_n(x_1,\cdots, x_n, t_1,0)g_n(x_{n+1},\cdots,x_{2n}, t_2, 0)\prod_{k=1}^{2n}\dot{W}_{\epsilon_k}(x_k)\\
	&=\sum_{{\cal D}\in\Pi_n}\int_{(\R^d)^{2n}}dx_1\cdots dx_{2n}\bigg(\prod_{(j,k)\in{\cal D}}\gamma_{\epsilon_j+\epsilon_k}(x_j-x_k)\bigg)\\
	&\qquad \times g_n(x_1,\cdots, x_n, t_1,0)g_n(x_{n+1},\cdots,x_{2n}, t_2, 0)\\
	&=\sum_{{\cal D}\in\Pi_n}
	F_{\epsilon,\epsilon'}^{\cal D}(t_1,t_2)\,, 
\end{aligned}
$$
where the second equality follows from (\ref{M-8}) with $g_k=\dot{W}_{\epsilon_k}(x_k)$ and $F_{\epsilon,\epsilon'}^{\cal D}(t_1,t_2)$ is given in (\ref{P-2}).
By (\ref{P-3}),  we see 
$$
\begin{aligned}
	&\int_0^\infty \!\!\int_0^\infty dt_1dt_2 e^{-\lambda_1t_1-\lambda_2t_2}\E S_{n,\epsilon}\big(g_n(\cdot, t_1, 0)\big)S_{n,\epsilon'}\big(g_n(\cdot, t_2, 0)\big)\\
	&\le {\lambda_1\lambda_2\over 4}\Big({1\over 2}\Big)^{2n}
	\int_0^\infty\!\!\int_0^\infty dt_1dt_2
	\exp\Big\{-{\lambda_1^2 t_1+\lambda_2^2 t_2\over 2}\Big\}\\
	&\qquad \times\E_0 \sum_{{\cal D}\in\Pi_n}\int_{[0,t_1]_<^n\times[0,t_2]_<^n}ds_1\cdots ds_{2n}\prod_{(j,k)\in{\cal D}}
	\gamma\big(B_{v(j)}(s_j)-B_{v(k)}(s_k)\big)\,. 
\end{aligned}
$$
For any  permutation  $\sigma$ on $\{1,\cdots, 2n\}$ with $\sigma(\{1,\cdots, n\})=\{1,\cdots, n\}$ and $\sigma(\{n+1,\cdots, 2n\})=\{n+1,\cdots, 2n\}$
$$
\begin{aligned}
	&\sum_{{\cal D}\in\Pi_n}\int_{[0,t_1]_<^n\times[0,t_2]_<^n}ds_1\cdots ds_{2n}\prod_{(j,k)\in{\cal D}}
	\gamma\big(B_{v(j)}(s_{\sigma(J)})-B_{v(k)}(s_{\sigma(k)})\big)\\
	&=\sum_{{\cal D}\in\Pi_n}\int_{[0,t_1]_<^n\times[0,t_2]_<^n}ds_1\cdots ds_{2n}\prod_{(j,k)\in{\cal D}}
	\gamma\big(B_{v(j)}(s_j)-B_{v(k)}(s_k)\big)\,. 
\end{aligned}
$$
Therefore
$$
\begin{aligned}
	&\sum_{{\cal D}\in\Pi_n}\int_{[0,t_1]_<^n\times[0,t_2]_<^n}ds_1\cdots ds_{2n}\prod_{(j,k)\in{\cal D}}
	\gamma\big(B_{v(j)}(s_j)-B_{v(k)}(s_k)\big)\\
	&={1\over (n!)^2}\sum_{{\cal D}\in\Pi_n}\int_{[0,t_1]^n\times[0,t_2]^n}ds_1\cdots ds_{2n}\prod_{(j,k)\in{\cal D}}
	\gamma\big(B_{v(j)}(s_j)-B_{v(k)}(s_k)\big)\,. 
\end{aligned}
$$
A crucial observation is that
$$
\begin{aligned}
	&\int_{[0,t_1]^n\times[0,t_2]^n}ds_1\cdots ds_{2n}\prod_{(j,k)\in{\cal D}}
	\gamma\big(B_{v(j)}(s_j)-B_{v(k)}(s_k)\big)\\
	&=\prod_{(j,k)\in{\cal D}}
	\int_0^{t_{v(j)}}\!\int_0^{t_{v(k)}}
	\gamma\big(B_{v(j)}(s)-B_{v(k)}(r))\big)dsdr\,. 
\end{aligned}
$$
Applying (\ref{M-8}) conditionally on the Brownian motions to the
$2n$-dimensional
normal vector
$$
\bigg(\overbrace{\int_0^{t_1}\dot{W}\big(B_1(s)\big)ds,\cdots,
	\int_0^{t_1}\dot{W}\big(B_1(s)\big)ds}^n,\hskip.1in
\overbrace{\int_0^{t_2}\dot{W}\big(B_2(s)\big)ds, \cdots,
	\int_0^{t_2}\dot{W}\big(B_2(s)\big)ds}^n\bigg)
$$
the right hand side is equal to
$$
\E\bigg[\int_0^{t_1}\dot{W}\big(B_1(s)\big)ds\bigg]^n\bigg[\int_0^{t_2}
\dot{W}\big(B_2(s)\big)ds\bigg]^n\,. 
$$
In summary
$$
\begin{aligned}
	&\int_0^\infty \!\!\int_0^\infty dt_1dt_2 e^{-\lambda_1t_1-\lambda_2t_2}\E S_{n,\epsilon}\big(g_n(\cdot, t_1, 0)\big)S_{n,\epsilon'}\big(g_n(\cdot, t_2, 0)\big)\\
	&\qquad \le {\lambda_1\lambda_2\over 4}\Big({1\over 2}\Big)^{2n}{1\over(n!)^2}
	\int_0^\infty\!\!\int_0^\infty dt_1dt_2
	\exp\Big\{-{\lambda_1^2 t_1+\lambda_2^2 t_2\over 2}\Big\}\\
	&\qquad\qquad  \times\E_0\otimes\E\bigg[\int_0^{t_1}\dot{W}\big(B_1(s)\big)ds\bigg]^n\bigg[\int_0^{t_2}\dot{W}\big(B_2(s)\big)ds\bigg]^n\,. 
\end{aligned}
$$
Finally, (\ref{P-15}) follows from (\ref{P-14}). \qed

\subsection{Stratonovich integrability and Fubini's theorem}

Recall that the function $F_{\epsilon,\epsilon'}^{\cal D}(t_1,t_2)$ is defined in (\ref{P-2}).

\begin{lemma}\label{int}
	Under   Dalang's condition (\ref{intro-6}), the limit
	\begin{align}\label{P-16}
		\lim_{\epsilon,\epsilon'\to 0}F_{\epsilon,\epsilon'}^{\cal D}(t_1,t_2)
	\end{align}
	exists for any $n\ge 1$, $t_1, t_2>0$ and any pair partition ${\cal D}\in \Pi_n$.
	Further, the limiting function is continuous in $t_1, t_2$.
\end{lemma}

\begin{remark} In view of (\ref{M-15}), Lemma \ref{int} justifies the definition
	$$
	\begin{aligned}
		&\int_{(\R^d)^{2n}}dx_1\cdots dx_{2n}\bigg(\prod_{(j,k)\in {\cal D}}
		\gamma(x_j-x_k)\bigg)
		g_n(x_1,\cdots, x_n, t_1,x)g_n(x_{n+1},\cdots, x_{2n}, t_2,0)\\
		&\buildrel \Delta\over =\lim_{\epsilon\to 0^+}\int_{(\R^d)^{2n}}
		dx_1\cdots dx_{2n}\bigg(\prod_{(j,k)\in {\cal D}}
		\gamma_\epsilon(x_j-x_k)\bigg)
		g_n(x_1,\cdots, x_n, t_1,x)g_n(x_{n+1},\cdots, x_{2n}, t_2 ,0)\,. 
	\end{aligned}
	$$
\end{remark}

\proof Clearly, $F^{\cal D}_{\epsilon,\epsilon'}(t_1,t_2)$ is
non-negative, non-decreasing and continuous
on $\R^+\times\R^+$. 

By (\ref{P-5}), Lemma \ref{L-1}, the limit
$$
\lim_{\epsilon,\epsilon'\to 0}\int_0^\infty\!\!\int_0^\infty e^{-\lambda_1 t_1-\lambda_2 t_2}F_{\epsilon,\epsilon'}^{\cal D}(t_1,t_2)dt_1dt_2
$$
exists for any $\lambda_1,\lambda_2>0$. 

By continuity theorem for Laplace  transform   \cite[Theorem 5.2.2]{K}, therefore, the function
$F^{\cal D}_{\epsilon}(t_1, t_2)$ weakly converges
to a non-negative, non-decreasing and right continuous function
$F^{\cal D}(t_1,t_2)$ on $(\R^+)^2$, i.e.,
$$
\lim_{\epsilon,\epsilon'\to 0}F^{\cal D}_{\epsilon,\epsilon'}(t_1, t_2)=F^{\cal D}(t_1, t_2)
$$
for any continuous point $(t_1, t_2)$ of $F^{\cal D}$ and
\begin{align}\label{P-17}
	\lim_{\epsilon,\epsilon'\to 0}\int_0^\infty\!\!\int_0^\infty e^{-\lambda_1 t_1-\lambda_2 t_2}
	F^{\cal D}_{\epsilon,\epsilon'}(t_1,t_2)dt_1dt_2=\int_0^\infty\!\!\int_0^\infty e^{-\lambda_1 t_1-\lambda_2 t_2}
	F^{\cal D}(t_1,t_2)dt_1dt_2\,. 
\end{align}
(Actually, Theorem 5.22, \cite{K}
is stated for probability measures on $(\R^+)^d$. The case of general measures on
$(\R^+)^d$ can be derived as in the proof of  \cite[Theorem 2a, Section 1, Chapter XIII]{Feller}.
Although this theorem 
only considers measures on $\R^+$ 
its extension to $(\R^+)^2$ is routine).

To establish the existence for
the limit in (\ref{P-16}) and therefore to complete the proof,
all we need is to show that $F^{\cal D}(t_1, t_2)$ is continuous   on
$(\R^+)^2$ so
\begin{align}\label{P-18}
	\lim_{\epsilon,\epsilon'\to 0}F^{\cal D}_{\epsilon,\epsilon'}(t_1, t_2)=F^{\cal D}(t_1, t_2)\,, \hskip.2in \forall t_1, t_2>0\,. 
\end{align}

We will do it by establishing
\begin{align}\label{P-19}
	\lim_{\delta_1,\delta_2\to 0^+}\sup_{\epsilon,\epsilon'}\big\{F^{\cal D}_{\epsilon,\epsilon'}(t_1, t_2)
	-F^{\cal D}_{\epsilon,\epsilon'}(t_1-\delta_1, t_2-\delta_2)\big\}=0\,.
\end{align}

Write
$$
\begin{aligned}
	&F^{\cal D}_{\epsilon,\epsilon'}(t_1, t_2)\\
	&=\int_{(\R^d)^{2n}}dx_1\cdots dx_{2n}\bigg(\prod_{(j,k)\in {\cal D}}\gamma_{\epsilon_j+\epsilon_k}(x_j-x_k)\bigg)
	\int_{[0, t_1]_<^n\times[0, t_2]_<^n}ds_1\cdots ds_{2n}\\
	&\qquad \times \bigg(G(s_1,x_1)\prod_{l=2}^n G(s_l-s_{l-1}, x_l-x_{l-1}\bigg)\\
	&\qquad\quad \times 
	\bigg(G(s_{n+1},x_{n+1})\prod_{k=n+2}^{2n} G(s_l-s_{l-1}, x_k-x_{k-1}\bigg)\\
	&={\cal E}_{\epsilon,\epsilon'}([0,t_1]_<^n\times [0,t_2]_<^n)\hskip.2in \hbox{(say)}\,. 
\end{aligned}
$$
To prove (\ref{P-19}), all we need is
$$
\lim_{\delta_1,\delta_2\to 0^+}\sup_{\epsilon,\epsilon'}{\cal E}_{\epsilon,\epsilon'}\Big(\big\{[0,t_1]_<^n\times [0,t_2]_<^n\big\}\setminus \big\{[0, t_1-\delta_1]_<^n\times [0,t_2-\delta_1]_<^n\big\}\Big)=0\,. 
$$

By the extension $G(t,x)=0$ for $t<0$, we can extend 
${\cal E}_{\epsilon,\epsilon'}(\cdot)$
from a measure on $(\R^+)_<^n\times(\R^+)_<^n$ to a measure
on $(\R^+)^n\times(\R^+)^n$
in an obvious way. Then by the relation
$$
\begin{aligned}
	&\big\{[0,t_1]_<^n\times [0,t_2]_<^n\big\}\setminus \big\{[0, t_1-\delta_1]_<^n\times [0,t_2-\delta_1]_<^n\big\}\\
	&\qquad \subseteq\Big([0,t_1]_<^n\times\big\{[0, t_2]_<^n\setminus [0,t_2-\delta_2]_<^n\big\}\Big)
	\cup\Big(\big\{[0, t_1]_<^n\setminus [0, t_1-\delta_1]_<^n\big\}\times [0, t_2]_<^n\Big)\\
	&\qquad \subseteq \Big([0,t_1]_<^n\times\big\{[0, t_2]_<^{n-1}\times [t_2-\delta_2, t_2]\big\}\Big)
	\cup\Big(\big\{[0, t_1]_<^{n-1}\times [t_1-\delta_1, t_1]\big\}\times [0, t_2]_<^n\Big)
\end{aligned}
$$
the problem is further reduced to
\begin{align}\label{P-20}
	\lim_{\delta\to 0^+}\sup_{\epsilon,\epsilon'}{\cal E}_{\epsilon,\epsilon'}\Big([0,t_1]_<^n\times[0, t_2]_<^{n-1}\times [t_2-\delta, t_2]\Big)
	=0
\end{align}
and
\begin{align}\label{P-21}
	\lim_{\delta\to 0^+}\sup_{\epsilon,\epsilon'}{\cal E}_{\epsilon,\epsilon'}\Big([0, t_1]_<^{n-1}\times [t_1-\delta, t_1]\times [0, t_2]_<^n\Big)
	=0\,. 
\end{align}

Due to similarity, we only prove (\ref{P-20}). By Fubini's theorem
$$
\begin{aligned}
	&{\cal E}_{\epsilon,\epsilon'}\Big([0,t_1]_<^n\times[0, t_2]_<^{n-1}\times [t_2-\delta, t_2]\Big)\\
	&=\int_{(\R^d)^{2n-1}}dx_1\cdots dx_{2n-1}\bigg(\prod_{(j,k)\in {\cal D}'}\gamma_{\epsilon_j+\epsilon_k}(x_j-x_k)\bigg)
	\int_{[0, t_1]_<^n\times[0, t_2]_<^{n-1}}ds_1\cdots ds_{2n-1}\\
	&\qquad \times \bigg(G(s_1,x_1)\prod_{l=2}^n G(s_l-s_{l-1}, x_l-x_{l-1}\bigg)
	\bigg(G(s_{n+1},x_{n+1})\prod_{k=n+2}^{2n-1} G(s_l-s_{l-1}, x_k-x_{k-1}\bigg)\\
	&\qquad \times\int_{\R^d}dx_{2n}\gamma_{\epsilon_{j_0}+\epsilon_{2n}}(x_{2n}-x_{j_0})\int_{t_2-\delta}^{t_2}G(s_{2n}-s_{2n-1}, x_{2n}-x_{2n-1})ds_{2n}\,, 
\end{aligned}
$$
where $1\le j_0\le 2n-1$ satisfies $(j_0, 2n)\in{\cal D}$ and where
${\cal D}'\in \Pi_{n-1}$ is given by
${\cal D}'={\cal D}\setminus(j_0, 2n)$.

By Fourier transform and Fubini's theorem
$$
\begin{aligned}
	&\int_{\R^d}dx_{2n}\gamma_{\epsilon_{j_0}+\epsilon_{2n}}(x_{2n}-x_{j_0})\int_{t_2-\delta}^{t_2}G(s_{2n}-s_{2n-1}, x_{2n}-x_{2n-1})ds_{2n}\\
	&\qquad =\int_{\R^d}dx_{2n}\gamma_{\epsilon_{j_0}+\epsilon_{2n}}(x_{2n}-x_{j_0})\int_{0\vee (t_2-s_{2n-1}-\delta)}^{t_2-s_{2n-1}}G(s, x_{2n}-x_{2n-1})ds\\
	&\qquad =\int_{\R^d}\mu(d\xi)\exp\Big\{-{\epsilon_{j_0}+\epsilon_{2n}\over 2}\vert\xi\vert^2\Big\}\int_{0\vee (t_2-s_{2n-1}-\delta)}^{t_2-s_{2n-1}}ds\\
	&\qquad\quad  \times\int_{\R^d}\exp\big\{i\xi\cdot(x_{2n}-x_{j_0})\big\}G(s, x_{2n}-x_{2n-1})dx_{2n}\,. 
\end{aligned}
$$
Using (\ref{M-11}), the right hand side is equal to
$$
\begin{aligned}
	&\int_{\R^d}\mu(d\xi)\exp\Big\{-{\epsilon_{j_0}+\epsilon_{2n}\over 2}\vert\xi\vert^2+i\xi\cdot(x_{2n-1}-x_{j_0})\Big\}\int_{0\vee (t_2-s_{2n-1}-\delta)}^{t_2-s_{2n-1}}ds\\
	&\qquad \qquad \times\int_{\R^d}\exp\big\{i\xi\cdot(x_{2n}-x_{2n-1})\big\}G(s, x_{2n}-x_{2n-1})dx_{2n}\\
	&\qquad =\int_{\R^d}\mu(d\xi)\exp\Big\{-{\epsilon_{j_0}+\epsilon_{2n}\over 2}\vert\xi\vert^2+i\xi\cdot(x_{2n-1}-x_{j_0})\Big\}\int_{0\vee (t_2-s_{2n-1}-\delta)}^{t_2-s_{2n-1}}
	{\sin(\vert\xi\vert s)\over\vert\xi\vert}ds\\
	&\qquad \le\int_{\R^d}\mu(d\xi)\bigg\vert\int_{0\vee (t_2-s_{2n-1}-\delta)}^{t_2-s_{2n-1}}
	{\sin(\vert\xi\vert s)\over\vert\xi\vert}ds\bigg\vert
	\,. 
\end{aligned}
$$
Notice that
$$
\begin{aligned}
	&\int_{0\vee (t_2-s_{2n-1}-\delta)}^{t_2-s_{2n-1}}
	{\sin(\vert\xi\vert s)\over\vert\xi\vert}ds
	={\cos\big(0\vee (t_2-s_{2n-1}-\delta)\vert \xi\vert -\cos(t_2-s_{2n-1})\vert\xi\vert\over\vert\xi\vert^2}\\
	&\qquad ={2\over\vert\xi\vert^2}\sin{\vert\xi\vert\big((t_2-s_{2n-1})-0\vee (t_2-s_{2n-1}-\delta)\big)\over 2}\\
	&\qquad\qquad \times \sin{\vert\xi\vert\big((t_2-s_{2n-1})+0\vee (t_2-s_{2n-1}-\delta)\big)\over 2}\,. 
\end{aligned}
$$
By the bounds $0\le (t_2-s_{2n-1})-0\vee (t_2-s_{2n-1}-\delta)\le \delta$ and $\vert \sin\theta\vert\le\vert\theta\vert$
$$
\int_{\R^d}\mu(d\xi)\bigg\vert\int_{0\vee (t_2-s_{2n-1}-\delta)}^{t_2-s_{2n-1}}
{\sin(\vert\xi\vert s)\over\vert\xi\vert}ds\bigg\vert\le 4t_2\delta\mu(\vert\xi\vert\le N)+2\int_{\{\vert\xi\vert\ge N\}}{1\over\vert\xi\vert^2}\mu(d\xi)
$$
for any $N>0$. 

In summary, there is a $\beta(\delta)>0$ independent of $(\epsilon, \epsilon')$ such that 
$$
\int_{\R^d}dx_{2n}\gamma_{\epsilon_{j_0}+\epsilon_{2n}}(x_{2n}-x_{j_0})\int_{t_2-\delta}^{t_2}G(s_{2n}-s_{2n-1}, x_{2n}-x_{2n-1})ds_{2n}\le\beta(\delta)
$$
and that $\beta(\delta)\to 0$ as $\delta\to 0^+$. Consequently,
$$
{\cal E}_{\epsilon,\epsilon'}\Big([0,t_1]_<^n\times[0, t_2]_<^{n-1}\times [t_2-\delta, t_2]\Big)\le \beta(\delta)\Lambda_{\epsilon,\epsilon'}(t_1, t_2)\,, 
$$
where
$$
\begin{aligned}
	&\Lambda_{\epsilon,\epsilon'}(t_1, t_2)=\int_{(\R^d)^{2n-1}}dx_1\cdots dx_{2n-1}\bigg(\prod_{(j,k)\in {\cal D}'}\gamma_{\epsilon_j+\epsilon_k}(x_j-x_k)\bigg)
	\int_{[0, t_1]_<^n\times[0, t_2]_<^{n-1}}ds_1\cdots ds_{2n-1}\\
	&\qquad \times \bigg(G(s_1,x_1)\prod_{l=2}^n G(s_l-s_{l-1}, x_l-x_{l-1}\bigg)
	\bigg(G(s_{n+1},x_{n+1})\prod_{k=n+2}^{2n-1} G(s_l-s_{l-1}, x_k-x_{k-1}\bigg)\,. 
\end{aligned}
$$
To establish (\ref{P-20}) and therefore to complete the proof, it suffices to show that
\begin{align}\label{P-22}
	\sup_{\epsilon,\epsilon'}\Lambda_{\epsilon,\epsilon'}(t_1, t_2)<\infty\,. 
\end{align}

Indeed, by a computation similar to the one used for (\ref{P-4})
$$
\begin{aligned}
	&\int_0^\infty\!\!\int_0^\infty e^{-t-\tilde{t}}\Lambda_{\epsilon,\epsilon'}(t, \tilde{t})dtd\tilde{t}\\
	&\quad \le\Big({1\over 2}\Big)^{2n+1}
	\int_0^\infty\int_0^\infty dtd\tilde{t}\exp\Big\{-{t+\tilde{t}\over 2}\Big\}\E_0\int_{[0, t]_<^n\times[0, \tilde{t}]_<^{n-1}}ds_1\cdots ds_{2n-1}\\
	&\qquad \times\prod_{(j,k)\in{\cal D}'}\gamma\big(B_{v(j)}(s_j)-B_{v(k)}(s_k)\big)
\end{aligned}
$$
for any $\epsilon,\epsilon'$. The above 
right hand side is finite by the fact (Lemma \ref{L-10}) that the moments of Brownian intersection local
times have polynomial increasing rates in time.

Finally, by non-negativity and monotonicity of $\Lambda_{\epsilon,\epsilon'}(t, \tilde{t})$ in $t$ and $\tilde{t}$, (\ref{P-20}) follows from the bound
$$
\sup_{\epsilon,\epsilon'}\Lambda_{\epsilon,\epsilon'}(t_1, t_2)\le\exp\{t_1+t_2\}
\sup_{\epsilon,\epsilon'}\int_0^\infty\!\!\int_0^\infty e^{-t-\tilde{t}}\Lambda_{\epsilon}(t, \tilde{t})dtd\tilde{t}<\infty\,. 
$$
This completes the proof. 
\qed

Keep in mind that the proof of Theorem \ref{th-4} depends on the Stratonovich integrability of $g_n(\cdot, t,x)$ and the ${\cal L}^2$-convergence of
the Laplace transform
$$
\int_0^\infty e^{-\lambda t}S_{n,\epsilon}\big(g_n(\cdot, t, x)\big)dt \hskip.1in\hbox{as $\epsilon\to 0$}
$$
that are installed in the following:

\begin{theorem}\label{prop}
	Under  Dalang's condition (\ref{intro-6}),
	
	\begin{enumerate}
		\item[(i)]  
		the ${\cal L}^2$-limit
		\begin{align}\label{P-23}
			\lim_{\epsilon_1,\cdots,\epsilon_n\to 0^+}
			\int_{(\R^d)^n}g_n(x_1,\cdots, x_n, t,x)\bigg(
			\prod_{k=1}^n\dot{W}_{\epsilon_k}(x_k)\bigg)dx_1\cdots dx_n
		\end{align}
		exists for any $n\ge 1$ and $(t,x)\in\R^+\times\R^d$.
		Consequently,  $g_n(\cdot, t,x)$
		is integrable in the sense of Definition \ref{def-2.2} and the limit in (\ref{P-23})
		is $S_n\big(g_n(\cdot, t,x)\big)$.
		\item[(ii)] for any $\lambda>0$,
		\begin{align}\label{P-24}
			\lim_{\epsilon\to 0}\int_0^\infty e^{-\lambda t}
			S_{n,\epsilon}\big(g_n(\cdot, t,x)\big)dt=\int_0^\infty e^{-\lambda t}
			S_{n}\big(g_n(\cdot, t,x)\big)dt\hskip.2in\hbox{in ${\cal L}^2(\Omega, {\cal F}, \P)$}\,. 
		\end{align}
	\end{enumerate}
\end{theorem}

\proof By Lemma \ref{L-0}, all we need is to show
\begin{align}\label{P-25}
	\lim_{\epsilon,\epsilon'\to 0^+}\E S_{n,\epsilon}\big(g_{n,\epsilon}(\cdot, t,x)\big)S_{n,\epsilon'}\big(g_n(\cdot, t,x)\big)\,, 
\end{align}
exists, where $S_{n,\epsilon}\big(g_{n,\epsilon}(\cdot, t,x)\big)$ is defined in (\ref{P-1})
and where $\epsilon'=(\epsilon_{n+1},\cdots, \epsilon_{2n})$.
We have
$$
\begin{aligned}
	&\E S_{n,\epsilon}\big(g_{n,\epsilon}(\cdot, t,x)\big)S_{n,\epsilon'}\big(g_n(\cdot, t,x)\big)\\
	&=\int_{(\R^d)^{2n}}dx_1\cdots dx_{2n}g_n(x_1,\cdots, x_n, t,0)g_n(x_{n+1},\cdots, x_{2n}, t,0)\E\bigg(\prod_{k=1}^{2n}\dot{W}_{\epsilon_k}(x_k)\bigg)\\
	&=\sum_{{\cal D}\in \Pi_n}\int_{(\R^d)^{2n}}dx_1\cdots dx_{2n}\bigg(\prod_{(j,k)\in {\cal D}}
	\gamma_{\epsilon_j+\epsilon_k}(x_j-x_k)\bigg)g_n(x_1,\cdots, x_n, t,0)g_n(x_{n+1},\cdots, x_{2n}, t,0)\\
	&=\sum_{{\cal D}\in\Pi_n}F_{\epsilon,\epsilon'}^{\cal D}(t_1,t_2)\,, 
\end{aligned}
$$
where the second  step follows from (\ref{M-8}) with $g_k=\dot{W}_{\epsilon_k}(x_k)$ ($k=1,\cdots 2n$). Therefore, the existence of the limit in (\ref{P-25})
follows from Lemma \ref{int}.

We now come to Part (ii). Notice that
$$
\begin{aligned}
	&\E\bigg[\int_0^\infty e^{-\lambda t}S_{n,\epsilon}\big(g_n(\cdot, t,x)\big)dt
	-\int_0^\infty e^{-\lambda t}S_{n}\big(g_n(\cdot, t,x)\big)dt\bigg]^2\\
	&\qquad =\int_0^\infty\!\int_0^\infty e^{-\lambda(t_1+t_2)}
	\E S_{n,\epsilon}\big(g_n(\cdot, t_1,0)\big)S_{n,\epsilon}\big(g_n(\cdot, t_2,0)\big)dt_1dt_2\\
	&\qquad\qquad-2\int_0^\infty\!\int_0^\infty e^{-\lambda(t_1+t_2)}
	\E S_{n}\big(g_{n,\epsilon}(\cdot, t_1,0)\big)S_{n}\big(g_n(\cdot, t_2,0)\big)dt_1dt_2\\
	& \qquad\qquad +\int_0^\infty\!\int_0^\infty e^{-\lambda(t_1+t_2)}
	\E S_{n}\big(g_n(\cdot, t_1,0)\big)S_{n}\big(g_n(\cdot, t_2,0)\big)dt_1dt_2\,. 
\end{aligned}
$$
For the first term
$$
\begin{aligned}
	&\int_0^\infty\!\int_0^\infty dt_1dt_2e^{-\lambda(t_1+t_2)}
	\E S_{n,\epsilon}\big(g_n(\cdot, t_1,0)\big)S_{n,\epsilon}\big(g_n(\cdot, t_2, 0)\big)dt_1dt_2\\
	&\qquad =\sum_{{\cal D}\in\Pi_n}\int_0^\infty\!\int_0^\infty F_{\epsilon,\epsilon}^{\cal D}(t_1,t_2)dt_1dt_2\,. 
\end{aligned}
$$

The function $F^{\cal D}(t_1,t_2)$ appearing in (\ref{P-18}) is identified as
$$
\begin{aligned}
	F^{\cal D}(t_1, t_2)=&\int_{(\R^d)^{2n}}dx_1\cdots dx_{2n}\bigg(\prod_{(j,k)\in{\cal D}}\gamma(x_j-x_k)\bigg)\\
	&\qquad \times g_n(x_1,\cdots, x_n, t, 0)g_n(x_{n+1},\cdots, x_{2n}, t, 0)\,. 
\end{aligned}
$$

By (\ref{P-17}) with $\lambda_1=\lambda_2=\lambda$, therefore,
$$
\begin{aligned}
	&\lim_{\epsilon\to 0}\int_0^\infty\!\int_0^\infty dt_1dt_2e^{-\lambda(t_1+t_2)}
	\E S_{n,\epsilon}\big(g_n(\cdot, t_1,0)\big)S_{n,\epsilon}\big(g_n(\cdot, t_2,0)\big)dt_1dt_2\\
	&\quad=\sum_{{\cal D}\in\Pi_n}\int_0^\infty\!\int_0^\infty dt_1dt_2e^{-\lambda(t_1+t_2)}
	\int_{(\R^d)^{2n}}dx_1\cdots dx_{2n}\bigg(\prod_{(j,k)\in {\cal D}}
	\gamma(x_j-x_k)\bigg)\\
	&\qquad\times g_n(x_1,\cdots, x_n, t,0)g_n(x_{n+1},\cdots, x_{2n}, t,0)\\
	&\quad=\int_0^\infty\!\int_0^\infty dt_1dt_2e^{-\lambda(t_1+t_2)}
	\E S_{n}\big(g_n(\cdot, t_2,0)\big)S_{n}\big(g_n(\cdot, t_2,0)\big)dt_1dt_2\,, 
\end{aligned}
$$
where the last step follows from Stratonovich integrability stated in Part (i)
and the identity in (\ref{M-18}).

Using Part (i),
$$
\E S_{n,\epsilon}\big(g_{n}(\cdot, t_1,0)\big)S_{n}\big(g_{n}(\cdot, t_2, 0\big)
=\lim_{\epsilon'\to 0}\E S_{n,\epsilon}\big(g_{n}(\cdot, t_1,0)\big)S_{n,\epsilon'}\big(g_n(\cdot, t_2,0)\big)\,. 
$$
By the fact that $\E S_{n,\epsilon}\big(g_n(\cdot, t_1,0)\big)S_{n}\big(g_n(\cdot, t_2,0)\big)\ge 0$ and
by Fatou's lemma, 
$$
\begin{aligned}
	&\liminf_{\epsilon\to 0}\int_0^\infty\!\int_0^\infty dt_1dt_2e^{-\lambda(t_1+t_2)}\E S_{n,\epsilon}\big(g_n(\cdot, t_1,0)\big)S_{n}\big(g_n(\cdot, t_2,0)\big)\\
	&\qquad\ge\int_0^\infty\!\int_0^\infty dt_1dt_2e^{-\lambda(t_1+t_2)}\liminf_{\epsilon\to 0}\E S_{n,\epsilon}\big(g_n(\cdot, t_1,0)\big)S_{n}\big(g_n(\cdot, t_2,0)\big)\\
	&\qquad=\int_0^\infty\!\int_0^\infty dt_1dt_2e^{-\lambda(t_1+t_2)}\lim_{\epsilon,\epsilon'\to 0}\E S_{n,\epsilon}\big(g_n(\cdot, t_1,0)\big)S_{n,\epsilon'}\big(g_n(\cdot, t_2,0)\big)\\
	&\qquad=\int_0^\infty\!\int_0^\infty dt_1dt_2e^{-\lambda(t_1+t_2)}\E S_{n}\big(g_n(\cdot, t_1,0)\big)S_{n}\big(g_n(\cdot, t_2,0)\big)\,, 
\end{aligned}
$$
where we have used Part (i) in the last two steps.

Summarizing our argument,
$$
\lim_{\epsilon\to 0}\E\bigg[\int_0^\infty e^{-\lambda t}S_{n,\epsilon}\big(g_n(\cdot, t,0)\big)dt
-\int_0^\infty e^{-\lambda t}S_{n}\big(g_n(\cdot, t,0)\big)dt\bigg]^2=0\,. 
$$
This competes the proof. \qed

We now establish Fubini's theorem for the multiple Stratonovich integral with the integrand $g_n(\cdot, t,x)$.

\begin{lemma}\label{L}   Under   Dalang's condition (\ref{intro-6}),  we have 
	\begin{align}\label{P-26}
		\lim_{\varepsilon_2,\cdots,\varepsilon_n\to 0^+}S_{n,\varepsilon}\big(g_n(\cdot, t,x)\big)=\int_{\R^d}\bigg(\int_0^t G(t-s, y-x)
		S_{n-1}\big(g_{n-1}(\cdot, s, y))ds\bigg)\dot{W}_{\varepsilon_1}(y)dy 
	\end{align}
	and  \begin{align}\label{P-27}
		\lim_{\varepsilon_1\to 0^+}\int_{\R^d}\bigg(\int_0^t G(t-s, y-x)
		S_{n-1}\big(g_{n-1}(\cdot, s, y))ds\bigg)\dot{W}_{\varepsilon_1}(y)dy 
		=S_n\big(g_n(\cdot, t,x)\big)\,, 
	\end{align}
	where the limits are taken in ${\cal L}^2(\Omega, {\cal F}, \P)$.
\end{lemma}

\begin{remark}\label{re} The identity \eqref{P-27} mathematically confirms the relation
	\eqref{e.2.9}.
\end{remark}  

\begin{proof}   Part (i) in Theorem \ref{prop} shows that
	$$
	\lim_{\epsilon\to 0^+}S_{n,\epsilon}\big(g_n(\cdot, t,x)\big)=S_n\big(g_n(\cdot, t,x)\big)\hskip.1in\hbox{in ${\cal L}^2(\Omega, {\cal F}, \P)$}\,. 
	$$
	Therefore, all we need  for establishing (\ref{P-26}) is to prove it with
	the convergence in ${\cal L}^1(\Omega, {\cal F},\P)$ instead.
	By   the Fubini theorem
	$$
	S_{n,\varepsilon}\big(g_n(\cdot, t,x)\big)
	=\int_{\R^d}\bigg(\int_0^t G(t-s, y-x)S_{n-1,\tilde{\varepsilon}}\big(g_{n-1}(\cdot, s, y)\big)ds\bigg)
	\dot{W}_{\varepsilon_1} (y)dy\,, 
	$$
	where
	$$
	\begin{aligned}S_{n-1,\tilde{\varepsilon}}\big(g_{n-1}(\cdot, t,x)\big)&=
		\int_{(\R^d)^{n-1}}\int_{[0,t]_<^{n-1}}ds_1\cdots ds_{n-1}\\
		&\times\bigg(\prod_{k=1}^{n-1}G(s_k-s_{k-1}, x_k-x_{k-1})\bigg)
		\prod_{k=1}^{n-1}\dot{W}_{\varepsilon_{k+1}}(x_k)dx_k
	\end{aligned}
	$$
	with notation $\tilde{\varepsilon}=(\varepsilon_2,\cdots,\varepsilon_n)$.
	So we have
	$$
	\begin{aligned}
		&S_{n,\varepsilon}\big(g_n(\cdot,t,x)\big)-\int_{\R^d}\bigg(\int_0^t G(t-s, y-x)
		S_{n-1}\big(g_{n-1}(\cdot, s, y))ds\bigg)\dot{W}_{\varepsilon_1}(y)dy\\
		&=\int_0^tds\int_{\R^d}
		\Big[S_{n-1,\tilde{\varepsilon}}\big(g_{n-1}(\cdot,s, y)\big)- S_{n-1}\big(g_{n-1}(\cdot, s, y))\Big]
		W_{\varepsilon_1}(y)G(t-s, y-x)dy\,. 
	\end{aligned}
	$$
	By the Cauchy-Schwartz inequality
	$$
	\begin{aligned}
		&\E\bigg\vert S_{n,\varepsilon}\big(g_n(\cdot,t,x)\big)-\int_{\R^d}\bigg(\int_0^t G(t-s, y-x)
		S_{n-1}\big(g_{n-1}(\cdot, s, y)\big)ds\bigg)\dot{W}_{\varepsilon_1}(y)dy\bigg\vert\\
		&\le\bigg\{\E\int_0^tds\int_{\R^d}
		\Big[S_{n-1,\tilde{\varepsilon}}\big(g_{n-1}(\cdot,s, y)\big)- S_{n-1}\big(g_{n-1}(\cdot, s, y)\big) \Big]^2
		G(t-s, y-x)dy\bigg\}^{1/2}\\
		&\times\bigg\{\E\int_0^tds\int_{\R^d}
		\vert\dot{W}_{\varepsilon_1}(y)\vert^2G(t-s, y-x)dy\bigg\}^{1/2}\\
		&=\bigg\{\E\int_0^tds\E\Big[S_{n-1,\tilde{\varepsilon}}\big(g_{n-1}(\cdot, s, 0)\big)- S_{n-1}\big(g_{n-1}(\cdot, s, 0)\big)\Big]^2
		\int_{\R^d}G(t-s, y-x)dy\bigg\}^{1/2}\\
		&\times \Big\{\E\vert\dot{W}_{\varepsilon_1}(0)\vert^2\Big\}^{1/2}\bigg\{
		\int_0^tds\int_{\R^d} G(t-s, y-x)dy\bigg\}^{1/2}\,, 
	\end{aligned}
	$$
	where the last step follows from the facts that
	$\E\vert\dot{W}_{\varepsilon_1}(y)\vert^2=\E\vert\dot{W}_{\varepsilon_1}(0)\vert^2$ and
	$$
	\E\Big[S_{n-1,\tilde{\varepsilon}}\big(g_{n-1}(\cdot, s, y)\big)- S_{n-1}\big(g_{n-1}(\cdot, s, y)\big)\Big]^2
	=\E\Big[S_{n-1,\tilde{\varepsilon}}\big(g_{n-1}(\cdot, s, 0)\big)- S_{n-1}\big(g_{n-1}(\cdot, s, 0)\big)\Big]^2\,. 
	$$
	Further,
	$$
	\int_{\R^d} G(t-s, y-x)dy=t-s
	$$
	We have the bound
	$$
	\begin{aligned}
		&\E\bigg\vert S_{n,\varepsilon}\big(g_n(\cdot,t,x)\big)-\int_{\R^d}\bigg(\int_0^t G(t-s, y-x)
		S_{n-1}\big(g_{n-1}(\cdot, s, y)\big)ds\bigg)\dot{W}_{\varepsilon_1}(y)dy\bigg\vert\\
		&\le \frac{1}{ 2}t^{3/2}\Big\{\E\vert\dot{W}_{\varepsilon_1}(0)\vert^2\Big\}^{1/2}
		\bigg\{\int_0^t
		\E\Big[S_{n-1,\tilde{\varepsilon}}\big(g_{n-1}(\cdot, s, 0)\big)- S_{n-1}\big(g_{n-1}(\cdot, s, 0)\big)\Big]^2ds\bigg\}^{1/2}\,. 
	\end{aligned}
	$$
	By Part (i) of Theorem \ref{prop} (with $n$ being replaced by $n-1$),
	$$
	\lim_{\varepsilon_2,\cdots,\varepsilon_n\to 0^+}
	\E\Big[S_{n-1,\tilde{\varepsilon}}\big(g_{n-1}(\cdot, s, 0)\big)- S_{n-1}\big(g_{n-1}(\cdot, s, 0)\Big]^2
	=0\,, \hskip.2in 0\le s\le t\,. 
	$$
	In addition
	$$
	\begin{aligned}
		&\E\Big[S_{n-1,\tilde{\varepsilon}}\big(g_{n-1}(\cdot, s, 0)\big)- S_{n-1}\big(g_{n-1}(\cdot, s, 0)\big)\Big]^2\\
		&\le \E\big[S_{n-1,\tilde{\varepsilon}}\big(g_{n-1}(\cdot, s, 0)\big)\big]^2+\E\big[S_{n-1}\big(g_{n-1}(\cdot, s, 0)\big)\big]^2\\
		&\le \E\big[S_{n-1,\tilde{\varepsilon}}\big(g_{n-1}(\cdot, t, 0)\big)\big]^2+\E\big[S_{n-1}\big(g_{n-1}(\cdot, t, 0)\big)\big]^2\,. 
	\end{aligned}
	$$
	By dominated convergence we see 
	$$
	\lim_{\varepsilon_2,\cdots,\varepsilon_n\to 0^+}\int_0^t
	\E\Big[S_{n-1,\tilde{\varepsilon}}\big(g_{n-1}(\cdot, s, 0)\big)- S_{n-1}\big(g_{n-1}(\cdot, s, 0)\big)\Big]^2ds  
	=0\,. 
	$$
	This proves the (\ref{P-26}). Finally, (\ref{P-27}) follows from (\ref{P-26}) and
	Theorem \ref{prop}.
\end{proof}

\subsection{Link to Dalang-Mueller-Tribe's work}\label{S-3}

The discussion in this sub-section does not contribute to the proof of the main theorems in this paper. Rather, it
helps the interested reader to better understand the true nature of Stratonovich solution and
provide a new representation to the Laplace transform of the deterministic system (\ref{intro-7})
for possible future investigation.

Let $N(t)$ ($t\ge 0$) be  a Poisson process with parameter 1 and $\{\tau_k\}_{k\ge 1}$ be the jumping times of $N(t)$ with
definition $\tau_0=0$.  The stochastic process $X_t$ ($t\ge 0$) is defined as follows: First, the random sequence
$\{X_{\tau_k}\}_{k\ge 1}$ is a random sequence whose finite-dimensional distribution of $(X_{\tau_1},\cdots, X_{\tau_k})$
has the conditional distribution (conditioning on $\{\tau_1,\cdots,\tau_n\}$)
$$ 
\bigg(\prod_{k=1}^n(\tau_k-\tau_{k-1})^{-1}G(\tau_k-\tau_{k-1}, x_k-x_{k-1})\bigg)dx_1\cdots dx_n\,. 
$$
Set $X_{\tau_0}=X_0=x$. The process $X_t$ is defined as the linear interpolation of $\{X_{\tau_k}\}_{k\ge 0}$. 

Dalang, Mueller and Tribe (Theorem 3.2, \cite{DMT}) prove that the function
\begin{align}\label{P-28}
	u(t,x)=e^t
	\E_x\Bigg[u_0(t-\tau_{N(t)}, X_{\tau_{N(t)}})\prod_{k=1}^{N(t)}(\tau_k-\tau_{k-1})f(X_{\tau_k})\Bigg]
\end{align}
solves the wave equation (\ref{intro-7}), where $u_0(t,x)$ appears in (\ref{M-1}). For the purpose of comparison,
we consider the case when $u_0(t,x)=1$ and write
$$
\begin{aligned}
	u(t,x)&=\sum_{n=0}^\infty e^t \P\{N(t)=n\}\E_x\Bigg[\prod_{k=1}^{n}(\tau_k-\tau_{k-1})f(X_{\tau_k})\bigg\vert N(t)=n\Bigg]\\
	&=\sum_{n=0}^\infty {t^n\over n!}\E_x\Bigg[\prod_{k=1}^{n}(\tau_k-\tau_{k-1})f(X_{\tau_k})\bigg\vert N(t)=n\Bigg]\,. 
\end{aligned}
$$
By the classic fact that conditioning on $\{N(t)=n\}$, the $n$-dimensional vector $(\tau_1,\cdots, \tau_n)$ is uniformly
distribution on $[0,t]_<^n$,
\begin{align}\label{P-29}
	u(t,x)&=\sum_{n=0}^\infty\int_{[0,t]_<^n}ds_1\cdots ds_n\int_{(\R^d)^n}dx_1\cdots dx_n\bigg(\prod_{k=1}^nG(s_k-s_{k-1}, x_k-x_{k-1})\bigg)\prod_{k=1}^nf(x_k)\nonumber\\
	&=\sum_{n=0}^\infty\int_{(\R^d)^n}dx_1\cdots dx_n g_n(x_1,\cdots, x_n, t, x)\prod_{k=1}^nf(x_k)
\end{align}
with the convention $s_0=0$ and $x_0=x$. Comparing  this with (\ref{M-2}) and (\ref{M-4}) we see the deterministic root of 
stochastic model (\ref{intro-1}) in the Stratonovich setting.

Similar to (\ref{P-10}), the same computation leads to
\begin{align}\label{P-30}
	&\int_0^\infty e^{-\lambda t}\int_{(\R^d)^n}dx_1\cdots dx_n g_n(x_1,\cdots, x_n, t, x)\prod_{k=1}^nf(x_k)\\
	&\qquad =\frac{1}{ n!}\frac{\lambda}{ 2}\Big(\frac{1}{ 2}\Big)^n
	\int_0^\infty\exp\Big\{-\frac{\lambda^2}{ 2}t\Big\}
	\E_x \bigg[\int_0^t f\big(B(s)\big)ds\bigg]^ndt\,. \nonumber
\end{align}
Summing both sides over $n$, we obtain the following representation
\begin{align}\label{P-31}
	\int_0^\infty e^{-\lambda t}u(t,x)dt={\lambda\over 2}
	\int_0^\infty\exp\Big\{-\frac{\lambda^2}{ 2}t\Big\}
	\E_x\exp\bigg\{\int_0^tf\big(B(s)\big)ds\bigg\}dt
\end{align}
in the sense that   finite  of one side leads to   finite  of the other
side, and to the equality.

The classic semi-group theory (see, e.g., Section 4.1, \cite{Chen-1}) claims
an asymptotically linear growth of the logarithmic exponential moment
$$
\log\E_x\exp\bigg\{\int_0^tf\big(B(s)\big)ds\bigg\}\hskip.2in (t\to\infty)
$$
for a class of functions $f$. In this case, the right hand side
of (\ref{P-31}) is finite for large $\lambda$.

On the other hand, the representation (\ref{P-31}) unlikely makes sense
for the stochastic wave equation (\ref{intro-1}). Under the assumption in
Theorem \ref{th-2}, we have (\cite{Chen-2})
$$
\log\E_x\exp\bigg\{\int_0^t\dot{W}\big(B(s)\big)ds\bigg\}
\sim C(\gamma)t(\log t)^{2\over 4-\alpha}\hskip.2in a.s.
$$
for some constant $C(\gamma)>0$ as $t\to\infty$. So (\ref{P-31}) almost surely blows up
for any $\lambda>0$.

\section{Proof of Theorem \ref{th-1}}\label{s.4} 

\begin{proof}[Proof  of Theorem \ref{th-1}]

	To show that the Stratonovich expansion (\ref{M-8}) converges in ${\cal L}^2(\Omega, {\cal F}, \P)$, by the triangle inequality
	and by the fact that $u(t,x)$ (if defined) is stationary in $x$, all we need is
	\begin{align}\label{E-1}
		\sum_n\Big\{\E\big[S_n\big(g_n(\cdot, t, 0)\big)\big]^2\Big\}^{1/2}<\infty\,,  \hskip.2in
		\forall t>0\,. 
	\end{align}
	
	The procedure starts at Corollary \ref{co-2}. By  the  Cauchy-Schwartz inequality
	$$
	\begin{aligned}
		&\E\otimes\E_0\bigg[\int_0^{t_1}\dot{W}\big(B_1(s)ds\bigg]^n\bigg[\int_0^{t_2}\dot{W}\big(B_2(s)\big)ds\bigg]^n\\
		&\quad \le\bigg\{\E\otimes\E_0\bigg[\int_0^{t_1}\dot{W}\big(B(s)ds\bigg]^{2n}\bigg\}^{1/2}\bigg\{\E\otimes\E_0\bigg[\int_0^{t_2}\dot{W}\big(B(s)ds\bigg]^{2n}\bigg\}^{1/2}\\
		&\quad =\bigg\{\E_0\bigg[\int_0^{t_1}\!\!\int_0^{t_1}\gamma\big(B(s)-B(r)\big)dsdr\bigg]^n\bigg\}^{1/2}\bigg\{\E_0\bigg[\int_0^{t_2}\!\!\int_0^{t_2}
		\gamma\big(B(s)-B(r)\big)dsdr\bigg]^n\bigg\}^{1/2}\,. 
	\end{aligned}
	$$

	Let $t>0$ be fixed. Taking  $\lambda_1=\lambda_2=nt^{-1}$
	in (\ref{P-14}), Corollary \ref{co-2}:
	$$
	\begin{aligned}
		&\int_0^\infty\!\int_0^\infty dt_1dt_2\exp\Big\{-\frac{n}{ t}(t_1+t_2)\Big\}
		\E \left[ S_l\big(g_n(\cdot, t_1,0)\big)S_{2n-l}\big(g_n(\cdot, t_2,0)\big)\right] \\
		&\quad \le\frac{(2n)!}{(n!)^2}\Big(\frac{n}{ 2t}\Big)^2\Big(\frac{1}{ 2}\Big)^{3n}
		\int_0^\infty\!\int_0^\infty dt_1dt_2\exp\Big\{-\frac{n^2}{ 2t^2}(t_1+t_2)\Big\}\\
		&\qquad \times\bigg\{\E_0\bigg[\int_0^{t_1}\!\!\int_0^{t_1}\gamma\big(B(s)-B(r)\big)dsdr\bigg]^n\bigg\}^{1/2}\bigg\{\E_0\bigg[\int_0^{t_2}\!\!\int_0^{t_2}\gamma\big(B(s)-B(r)\big)dsdr\bigg]^n\bigg\}\\
		&\quad =\frac{(2n)!}{(n!)^2}\Big(\frac{n}{ 2t}\Big)^2\Big(\frac{1}{ 2}\Big)^{3n}\bigg\{\int_0^\infty d\tilde{t}\exp\Big\{-\frac{n^2}{ 2t^2}\tilde{t}\Big\}
		\bigg(\E_0\bigg[\int_0^{\tilde{t}}\!\!\int_0^{\tilde{t}}\gamma\big(B(s)-B(r)\big)dsdr\bigg]^n\bigg)^{1/2}\bigg\}^2\,. 
	\end{aligned} 
	$$
	
	Recall ((1.5), Theorem 1.1, \cite{Chen-3}) that under   Dalang's condition (\ref{intro-6}), the limit
	$$
	\lim_{\tilde{t}\to\infty}{1\over \tilde{t}}\log\E_0\exp\bigg\{{1\over \tilde{t}}\int_0^{\tilde{t}}\!\!\int_0^{\tilde{t}}\gamma(B(s)-B(r))dsdr\bigg\}
	$$
	exists and is finite.  This means there is a constant $C$ such that
	\[
	\E_0\exp\bigg\{\frac{1}{\tilde t} \int_0^{\tilde{t}}
	\!\!\int_0^{\tilde{t}}\gamma(B(s)-B(r))dsdr\bigg\}\le \exp\{C\tilde t\}\,. 
	\]
	By the relation
	$$
	{1\over n! \tilde{t}^n}\E_0\bigg[  \int_0^{\tilde{t}}\!\!\int_0^{\tilde{t}}
	\gamma(B(s)-B(r))dsdr\bigg]^n\le \E_0\exp\bigg\{\frac{1}{\tilde t} \int_0^{\tilde{t}}
	\!\!\int_0^{\tilde{t}}\gamma(B(s)-B(r))dsdr\bigg\}
	$$
	for any $\tilde{t}>0$, we have the bound that is uniform in $\tilde{t}$ and $n$:
	$$
	\E_0\bigg[\int_0^{\tilde{t}}\!\!\int_0^{\tilde{t}}\gamma\big(B(s)-B(r)\big)dsdr\bigg]^n
	\le n!\tilde{t}^n \exp\{C\tilde{t}\}\,. 
	$$
	Hence,  
	$$
	\begin{aligned}
		&\int_0^\infty \exp\Big\{-\frac{n^2}{ 2t^2}\tilde{t}\Big\}
		\bigg(\E_0\bigg[\int_0^{\tilde{t}}\!\!\int_0^{\tilde{t}}\gamma\big(B(s)-B(r)\big)dsdr\bigg]^n\bigg)^{1/2}d\tilde{t}\\
		&\qquad  \le (n!)^{1/2}\int_0^\infty \exp\Big\{-\frac{n^2}{ 4t^2}\tilde{t}\Big\}\tilde{t}^{n/2}d\tilde{t}=(n!)^{1/2}\Big(\frac{4t^2}{ n^2}\Big)^{n+1}\Gamma\Big({n\over 2}+1\Big)\,. 
	\end{aligned}
	$$
	Thus,  by the Stirling formula we get the bound
	\begin{align}\label{E-2}
		&\int_0^\infty\!\int_0^\infty dt_1dt_2\exp\Big\{-\frac{n}{ t}(t_1+t_2)\Big\}
		\E \left[ S_n\big(g_n(\cdot, t_1,0)\big)S_{n}\big(g_{n}(\cdot, t_2,0)\big)\right] 
		\le \frac{C^n}{n!}t^{2n+4}\,. 
	\end{align}

	By the fact that
	the moment
	$$
	\E \left[ S_n\big(g_n(\cdot, t_1,0)\big)S_{n}\big(g_{n}(\cdot, t_2,0)\big)\right] 
	$$
	is non-negative and non-decreasing in $t_1$ and $t_2$ we have  
	$$
	\begin{aligned}
		&\int_0^\infty\!\int_0^\infty dt_1dt_2\exp\Big\{-\frac{n}{ t}(t_1+t_2)\Big\}
		\E \left[ S_n\big(g_n(\cdot, t_1,0)\big)S_{n}\big(g_{n}(\cdot, t_2,0)\big)\right] dt_1 dt_2\\
		&\qquad \ge\E \left[ S_n\big(g_n(\cdot, t,0)\big)\right]^2
		\int_t^\infty\!\int_t^\infty dt_1dt_2\exp\Big\{-\frac{n}{ t}(t_1+t_2)\Big\}dt_1dt_2\\
		&\qquad =\frac{t^2}{ n^2}e^{-2n}\E \left[ S_n\big(g_n(\cdot, t,0)\big)\right]^2\,. 
	\end{aligned}
	$$
	Comparing this with \eqref{E-2}  we get the bound 
	\begin{align}\label{E-3}
		\E \left[ S_n\big(g_n(\cdot, t,0)\big)\right]^2 
		\le  \frac{C_2^n}{ n!}t^{2n+2}\,, \hskip.2in n=1,2,\cdots
	\end{align}
	This leads to (\ref{E-1}) and therefore
	to the ${\cal L}^2$-convergence of the Stratonovich expansion in (\ref{M-3}).

	In view of (\ref{P-15}), the bound (\ref{E-3}) remains true for
	$\E \left[ S_{n,\epsilon}\big(g_n(\cdot, t,0)\big)\right]^2$
	for any $\epsilon=(\epsilon_1,\cdots,\epsilon_n)$, i.e.,
	$$
	\E \left[ S_{n,\epsilon}\big(g_n(\cdot,t,0)\big)\right]^2 
	\le  \frac{C_2^n}{ n!} t^{2n+2}\,, \hskip.2in n=1,2,\cdots
	$$
	for any $t>0$. Let $\epsilon_2,\cdots,\epsilon_n\to 0^+$ on the left hand side. By (\ref{P-26}), Lemma \ref{L},
	\begin{equation}  
		\begin{split}
			&\E\left|  \int_{\R^d}\bigg(\int_0^t G(t-s, y-x)
			S_{n-1}\big(g_{n-1}(\cdot, s, y))ds\bigg)\dot{W}_{\varepsilon_1}(y)dy \right|^2  \le \frac{C_2^n}{ n!}t^{2n+2}
		\end{split}\label{E-4} 
	\end{equation}
	for any $n=1,2\cdots$, any $t> 0$ and any  $\epsilon_1>0$.

	To show that $\{u(t,x)\}$  is a solution in the sense of Definition 
	\ref{d.mild_solution}  and therefore to complete the proof of Part (i) of Theorem \ref{th-1}, we only need to show 
	\begin{enumerate}
		\item[(1)] For any $t>0$ and $x\in\R^d$, the random field $V(y)\equiv \int_0^t G_{t-s}(x-y)u(s,y) ds$ is Stratonovich integrable, or
		$$
		\lim_{\epsilon_1\to 0^+}\int_{\R^d} \bigg(\int_0^t G_{t-s}(x-y)u(s,y) ds \bigg)W_{\varepsilon_1} (y) dy=\int_{\R^d} \bigg(\int_0^t G_{t-s}(x-y)u(s,y) ds\bigg)W (dy)
		$$
		in ${\cal L}^2(\Omega, {\cal F},\P)$.
		\item[(2)] Equation \eqref{M-1} is satisfied with $u_0(t,x)=1$.
	\end{enumerate} 
	Because of \eqref{E-3} and \eqref{E-4}, to show (1) and (2)  one has only to show that for all fixed $n\ge 1$, $\int_{\R^d} \big(\int_0^t G_{t-s}(x-y) S_{n-1}\big(g_{n-1}(\cdot,s,y)\big) ds \big)W_\varepsilon(y) dy$ converges to $S_n\big(g_n(\cdot, t,x)\big)$ in ${\cal L}^2(\Omega, {\cal F}, \P)$.  This is done in Lemma \ref{L}, Equation \eqref{P-27}.   
	
	To prove Part (ii) of Theorem \ref{th-1}, all we need is to show that Dalang's condition is necessary for
	$$
	\E\big[S_2\big(g_2(\cdot, t, 0)\big)\big]^2<\infty
	$$
	with any $t>0$. Indeed,
	$$
	\begin{aligned}
		&\E\big[S_2\big(g_2(\cdot, t, 0)\big)\big]^2\\
		&=\sum_{{\cal D}\in \Pi_2}\int_{(\R^d)^4}dx_1dx_2dx_3dx_4\bigg(\prod_{{\cal D}\in\Pi_2}\gamma(x_j-x_k)\bigg)
		g_2(x_1,x_2, t,0)g_2(x_3, x_4, t, 0)\cr
		&\ge\int_{(\R^d)^4}dx_1dx_2dx_3dx_4\gamma(x_1-x_2)\gamma(x_3-x_4)
		g_2(x_1,x_2, t,0)g_2(x_3, x_4, t, 0)\\
		&=\bigg(\int_{(\R^d)^2}\gamma(x_2-x_1)g(x_1, x_2, t,0)dx_1dx_2\bigg)^2
	\end{aligned}
	$$
	and
	$$
	\begin{aligned} 
		&\int_{(\R^d)^2}\gamma(x_2-x_1)g(x_1, x_2, t,0)dx_1dx_2\\
		&=\int_{[0,t]_<^2}ds_1ds_2\int_{(\R^d)^2}\gamma(x_2-x_1)G(s_1,x_1)G(s_2-s_1, x_2-x_1)dx_1dx_2\\
		& =\int_{[0,t]_<^2}\bigg(\int_{\R^d}G(s_1,x)dx\bigg)\bigg(\int_{\R^d}\gamma(x)G(s_2-s_1, x)dx\bigg)ds_1ds_2\\
		&=\int_{[0,t]_<^2} s_1\bigg[\int_{\R^d}{\sin(\vert\xi\vert (s_2-s_1))\over\vert\xi\vert}\mu(d\xi)\bigg]ds_1
		=\int_{\R^d}{\mu(d\xi)\over\vert\xi\vert}\int_0^t s_1\bigg[\int_0^{t-s_1}\sin (\vert\xi\vert s_2)ds_2\bigg]ds_1\\
		&=\int_0^t s_1\bigg[\int_{\R^d}{1-\cos\big(\vert \xi\vert (t-s_1)\big)\over\vert\xi\vert^2}\mu(d\xi)\bigg]ds_1\,. 
	\end{aligned}
	$$
	Clearly, the finiteness on the right hand side leads to   Dalang's condition (\ref{intro-6}).
\end{proof}

\section{Proof of Theorem \ref{th-2}}\label{s.5}

From the expansion (\ref{M-3}) and the stationarity of
the Stratonovich moment in $x$, a formal algebra leads to
$$
\E u^p(t,x)=\sum_{n=0}^\infty\sum_{l_1+\cdots +l_p=n}\E\prod_{j=1}^p
S_{l_j}\big(g_{l_j}(\cdot, t,0)\big)
=\sum_{n=0}^\infty\sum_{l_1+\cdots +l_p=2n}\E\prod_{j=1}^p
S_{l_j}\big(g_{l_j}(\cdot, t,0)\big)\,, 
$$
where the second equality follows from the fact ((\ref{M-17})) that
$$
\E\prod_{j=1}^p
S_{l_j}\big(g_{l_j}(\cdot, t,0)\big)=0
$$
whenever $l_1+\cdots+l_p$ is odd. Moreover,  the expansion
for $\E u^p(t,x)$ appears as
a positive series. Consequently $\E u^p(1,x)> 0$. 

Mathematically, under   Dalang's condition (\ref{intro-6})
the Stratonovich expansion (\ref{M-3}) converges in
${\cal L}^p(\Omega,{\cal F},\P)$ for any $p>0$. Indeed, it is enough to exam
this for all even numbers $p$. This follows from the estimate
$$
\E\bigg\vert\sum_{n=N+1}^{N+m}S_n\big(g_n(\cdot, t,0)\big)\bigg\vert^p
\le\sum_{n:2n\ge N+1}\sum_{l_1+\cdots +l_p=2n}\E\prod_{j=1}^p
S_{l_j}\big(g_{l_j}(\cdot, t,0)\big)\,. 
$$
Therefore, the claimed ${\cal L}^p$-convergence relies on the fact
$$
\sum_{n=0}^\infty\sum_{l_1+\cdots +l_p=2n}\E\prod_{j=1}^p
S_{l_j}\big(g_{l_j}(\cdot, t,0)\big)<\infty
$$
which appears as a direct consequence of (\ref{LDP-1}) and
(\ref{LDP-3}) below.

By (\ref{intro-8}) and (\ref{M-13}), in addition, one can verify that
\begin{align}\label{LDP-1}
	\sum_{l_1+\cdots +l_p=2n}\E\prod_{j=1}^p
	S_{l_j}\big(g_{l_j}(\cdot, t,0)\big)=t^{(4-\alpha)n}
	\sum_{l_1+\cdots +l_p=2n}\E\prod_{j=1}^p
	S_{l_j}\big(g_{l_j}(\cdot, 1,0)\big)\,,  \hskip.2in\forall t>0\,. 
\end{align}
Therefore, (\ref{E-3}) can be written as
\begin{align}\label{LDP-2}
	\E u^p(t,x)=\sum_{n=0}^\infty t^{(4-\alpha)n}\bigg(\sum_{l_1+\cdots +l_p=2n}\E\prod_{j=1}^p
	S_{l_j}\big(g_{l_j}(\cdot, 1,0)\big)\bigg)
\end{align}
for each $p=1,2,\cdots$.

\bigskip

\begin{proof}[Proof of Theorem \ref{th-2}]  First, we claim 
	\begin{align}\label{LDP-3}
		&\lim_{n\to\infty}\frac{1}{ n}\log
		(n!)^{3-\alpha}\bigg(\sum_{l_1+\cdots+l_p=2n}
		\E \prod_{j=1}^p S_{l_j}\big(g_{l_j}(\cdot, 1,0)\big)\bigg)=\log\Big(\frac{1}{ 2}\Big)^{3-\alpha} p^{4-\alpha}\bigg(\frac{2{\cal M}^{1/2}}{ 4-\alpha}\bigg)^{4-\alpha}
	\end{align}
	for each integer $p\ge 1$. 
	In next subsections we shall prove  the upper bound part of this claim in   (\ref{LDP-14}) and  the lower bound  part in \eqref{LDP-21}.  
	
	After we established    (\ref{LDP-3})  the proof of    (\ref{intro-9}) is easy and can be seen through
	the following computation:  From  (\ref{LDP-2})  and then \eqref{LDP-3}   it follows 
	\begin{align}\label{LDP-4}
		\lim_{t\to\infty}t^{-\frac{4-\alpha}{ 3-\alpha}}\log u^p(t,x)&=\lim_{t\to\infty}t^{-\frac{4-\alpha}{ 3-\alpha}}\log\sum_{n=0}^\infty t^{(4-\alpha)n}
		\bigg(\sum_{l_1+\cdots +l_p=2n}\E\prod_{j=1}^pS_{l_j}\big(g_{l_j}(\cdot, 1,0)\big)\bigg)\\
		&=\lim_{t\to\infty}t^{-\frac{4-\alpha}{ 3-\alpha}}\log\sum_{n=0}^\infty \frac{t^{(4-\alpha)n}}{ (n!)^{3-\alpha}}
		\bigg( \Big(\frac{1}{ 2}\Big)^{3-\alpha} p^{4-\alpha}\bigg(\frac{ 2{\cal M}^{1/2}}{ 4-\alpha}
		\bigg)^{4-\alpha}\bigg)^n\nonumber\\
		&=\frac{3-\alpha}{ 2}p^{\frac{4-\alpha}{ 3-\alpha}}\Big(\frac{2{\cal M}^{1/2}}{ 4-\alpha}
		\Big)^{\frac{4-\alpha}{ 3-\alpha}}\,, \nonumber
	\end{align}
	where the last step follows  from the following  elementary fact  of the 
	asymptotics of the Mittag-Leffler function (Lemma A.3, \cite{BCC}): 
	\begin{align}\label{LDP-5}
		\lim_{b\to\infty}b^{-1/\gamma}\log\sum_{n=0}^\infty\frac{\theta^nb^n}{ (n!)^\gamma}=\gamma\theta^{1/\gamma}\,, 
		\hskip.2in \theta>0
	\end{align}
	with $\gamma =3-\alpha$ and with $b=t^{4-\alpha}$.

	The proof for the upper bound of (\ref{intro-9}) is given in \eqref{LDP-7}
	of  Lemma \ref{L-4};
	and the lower bound is established in (\ref{LDP-26}).  
\end{proof}

\subsection{Upper bounds of  (\ref{intro-9}) and (\ref{intro-10})}\label{5-1}

\begin{lemma}\label{L-4} Under the condition in Theorem \ref{th-2},  we have the following statements. 
	\begin{enumerate}
		\item[(1)] for any $\lambda_1,\cdots, \lambda_p>0$ and $p=1,2,\cdots$ 
		\begin{align}\label{LDP-6}
			&\limsup_{n\to\infty}\frac{1}{ n}\log \frac{1}{ n!}\int_{(\R^+)^p}dt_1\cdots dt_p
			\exp\Big\{-\sum_{j=1}^p\lambda_j t_j\Big\}
			\bigg(\sum_{l_1+\cdots+l_p=2n}
			\E \prod_{j=1}^pS_{l_j}\big(g_{l_j}(\cdot, t_j,0)\big)\bigg)\\
			&\qquad \le\log 2{\cal M}^{\frac{ 4-\alpha}{ 2}}+\frac{4-\alpha}{ 2}
			\sum_{j=1}^p\frac{ \lambda_j^{-2}\log\lambda_j^{-2}}{ \lambda_1^{-2}+\cdots+\lambda_p^{-2}}\,. \nonumber
		\end{align}
		\item[(2)] for any $t>0$,
		\begin{align}\label{LDP-7} 
			\limsup_{p\to\infty}p^{-\frac{4-\alpha}{3-\alpha}}\log\E \vert u(t,0)\vert^p
			\le \frac{3-\alpha}{ 2}t^{\frac{4-\alpha}{ 3-\alpha}}
			\bigg(\frac{2{\cal M}^{1/2}}{ 4-\alpha}\bigg)^{\frac{4-\alpha}{ 3-\alpha}}\,. 
		\end{align}  
	\end{enumerate}
\end{lemma}

\begin{proof} The proof starts with the moment representation in Corollary \ref{co-1}.
	On the right hand side of (\ref{P-11}), we
	perform the estimation   by Fourier transform
	$$
	\begin{aligned}
		&\sum_{j,k=1}^p\int_0^{t_j}\!\int_0^{t_k}
		\gamma\big(B_j(s)-B_k(r)\big)dsdr\\
		&\qquad =\int_{\R^d}\mu(d\xi)\bigg\vert\sum_{j=1}^p\int_0^{t_j}e^{i\xi\cdot B_j(s)}ds
		\bigg\vert^2\\
		&\qquad =(t_1+\cdots +t_p)^2\int_{\R^d}\mu(d\xi)\bigg\vert\sum_{j=1}^p
		{t_j\over t_1+\cdots +t_p}
		{1\over t_j}\int_0^{t_j}e^{i\xi\cdot B_j(s)}ds\bigg\vert^2\\
		&\qquad \le(t_1+\cdots +t_p)\sum_{j=1}^p
		t_j\int_{\R^d}\mu(d\xi)\bigg\vert
		{1\over t_j}\int_0^{t_j}e^{i\xi\cdot B_j(s)}ds\bigg\vert^2\\
		&\qquad =(t_1+\cdots +t_p)\sum_{j=1}^p
		{1\over t_j}\int_0^{t_j}\!\int_0^{t_j}\gamma\big(B_j(s)-B_j(r)\big)dsdr\\
		&\qquad \buildrel d\over =(t_1+\cdots +t_p)\sum_{j=1}^p
		t_j^{2-\alpha\over 2}\int_0^1\!\int_0^1\gamma\big(B_j(s)-B_j(r)\big)dsdr\,. 
	\end{aligned}
	$$
	The advantage of the above inequality is to  replace the sum of dependent  quantities  by the sum  of independent ones, where the last step
	follows from scaling
	$$
	\int_0^{t_j}\!\int_0^{t_j}\gamma\big(B_j(s)-B_j(r)\big)dsdr\buildrel d\over =t_j^{4-\alpha\over 2}
	\int_0^1\!\int_0^1\gamma\big(B_j(s)-B_j(r)\big)dsdr\,, \hskip.2in j=1,\cdots, p
	$$
	and the independence of  the Brownian motions.
	
	Combining the above result with  (\ref{P-14}) gives  
	\begin{align} \label{e.5.8} 
		&\int_{(\R^+)^p}dt_1\cdots dt_p
		\exp\Big\{-\sum_{j=1}^p\lambda_jt_j\Big\}\sum_{l_1+\cdots+l_p=2n}
		\E \prod_{j=1}^p  S_{l_j}\big(g_{l_j}(\cdot, t_j,0)\big) \nonumber\\
		&\qquad \le\Big({1\over 2}\Big)^{3n}{1\over n!}
		\bigg(\prod_{j=1}^p{\lambda_j\over 2}\bigg)
		\int_{(\R^+)^p}dt_1\cdots dt_p
		\exp\Big\{-{1\over 2}\sum_{j=1}^p\lambda_j^2t_j\Big\}(t_1+\cdots+t_p)^n
		\nonumber\\
		&\qquad \qquad \times\E_0\bigg[\sum_{j=1}^p
		t_j^{2-\alpha\over 2}\int_0^1\!\int_0^1\gamma\big(B_j(s)-B_j(r)\big)dsdr\bigg]^n
		\nonumber\\
		&\qquad =\bigg(\prod_{j=1}^p{\lambda_j\over 2}\bigg)
		\Big({1\over 2}\Big)^{3n}\sum_{l_1+\cdots+l_p=n}{1\over l_1!\cdots l_p!}
		\bigg\{\prod_{j=1}^p\E_0\bigg[\int_0^1\!\int_0^1
		\gamma\big(B(s)-B(r)\big)dsdr\bigg]^{l_j}\bigg\} \nonumber \\
		&\qquad \qquad \times\int_{(\R^+)^p}dt_1\cdots dt_p
		(t_1+\cdots +t_p)^n\exp\Big\{-{1\over 2}\sum_{j=1}^p \lambda_j^2 t_j\Big\}
		\prod_{j=1}^pt_j^{{2-\alpha\over 2}l_j}\,.  
	\end{align}  
	By   \cite[Theorem 1.1]{Chen-3}, we see 
	$$
	\begin{aligned}
		&\lim_{t\to\infty}{1\over t}\log\E_0\exp\bigg\{\bigg(\int_0^t\!\int_0^t
		\gamma\big(B(s)-B(r)\big)dsdr\bigg)^{1/2}\bigg\}\\
		&\qquad=\sup_{g\in{\cal F}_d}\bigg\{\bigg(\int_{\R^d\times\R^d}\gamma(x-y)g^2(x)g^2(y)dxdy\bigg)^{1/2}
		-{1\over 2}\int_{\R^d}\vert\nabla g(x)\vert^2dx\bigg\}\\
		&\qquad=2^{\alpha\over 4-\alpha}{\cal M}\,. 
	\end{aligned}
	$$
	By (\ref{intro-8}) and Brownian scaling,
	$$
	\int_0^t\!\int_0^t
	\gamma\big(B(s)-B(r)\big)dsdr\buildrel d\over = t^{4-\alpha\over 2}\int_0^1\!\int_0^1
	\gamma\big(B(s)-B(r)\big)dsdr
	$$
	we can rewrite it as
	$$
	\lim_{t\to\infty}{1\over t}\log\E_0\exp\bigg\{t^{4-\alpha\over 4}\bigg(\int_0^1\!\int_0^1
	\gamma\big(B(s)-B(r)\big)dsdr\bigg)^{1/2}\bigg\}=2^{\alpha\over 4-\alpha}{\cal M}\,. 
	$$

	On the other hand, by Taylor's  expansion and the positivity of $\gamma$, 
	we have  
	$$
	\begin{aligned}
		&{1\over (2n)!}t^{{4-\alpha\over 2}n}\E_0\bigg[\int_0^1\!\int_0^1
		\gamma\big(B(s)-B(r)\big)dsdr\bigg]^n\\
		&\qquad \le\E_0\exp\bigg\{t^{4-\alpha\over 4}\bigg(\int_0^1\!\int_0^1
		\gamma\big(B(s)-B(r)\big)dsdr\bigg)^{1/2}\bigg\}\,.
	\end{aligned}
	$$
	For any $\theta>0$, taking $t=\theta n$ and by Stirling's formula,
	$$
	\begin{aligned}
		&\limsup_{n\to\infty}{1\over n}\log (n!)^{-\alpha/2}\E_0\bigg[\int_0^1\!\int_0^1
		\gamma\big(B(s)-B(r)\big)dsdr\bigg]^n\\
		&\qquad \le -{4-\alpha\over 2}+\log 4+2^{\alpha\over 4-\alpha}{\cal M}\theta
		-{4-\alpha\over 2}\log\theta\,. 
	\end{aligned}
	$$
	Picking  the minimizer
	$$
	\theta=2^{-{\alpha\over 4-\alpha}}{4-\alpha\over 2{\cal M}}
	$$
	yields 
	$$
	\limsup_{n\to\infty}{1\over n}\log (n!)^{-\alpha/2}\E_0\bigg[\int_0^1\!\int_0^1
	\gamma\big(B(s)-B(r)\big)dsdr\bigg]^n
	\le\log 2^4\Big({{\cal M}\over 4-\alpha}\Big)^{4-\alpha\over 2}\,. 
	$$
	Consequently, for any given $\delta>0$ there is $C_\delta>0$ such that
	$$
	\E_0\bigg[\int_0^1\!\int_0^1
	\gamma\big(B(s)-B(r)\big)dsdr\bigg]^n\le C_\delta (n!)^{\alpha/2}
	\bigg((1+\delta)2^4\Big({{\cal M}\over 4-\alpha}\Big)^{4-\alpha\over 2}\bigg)^n\,, 
	\hskip.2in n=1,2,\cdots
	$$
	Substituting this bound into \eqref{e.5.8} gives 
	\begin{align}\label{LDP-8}
		&\int_{(\R^+)^p}dt_1\cdots dt_p
		\exp\Big\{-\sum_{j=1}^p\lambda_jt_j\Big\}\sum_{l_1+\cdots+l_p=2n}
		\E \prod_{j=1}^pS_{l_j}\big(g_{l_j}(\cdot, t_j,0)\big)\\
		&\qquad\le \bigg(\prod_{j=1}^p{C_\delta\lambda_j\over 2}\bigg)
		\bigg(2(1+\delta)\Big({{\cal M}\over 4-\alpha}\Big)^{4-\alpha\over 2}\bigg)^n
		\sum_{l_1+\cdots+l_p=n}\bigg(\prod_{j=1}^p(l_j!)^{-{2-\alpha\over 2}}\bigg)\nonumber\\
		&\qquad\qquad  \times\int_{(\R^+)^p}dt_1\cdots dt_p
		(t_1+\cdots +t_p)^n\exp\Big\{-{1\over 2}\sum_{j=1}^p \lambda_j^2t_j\Big\}
		\prod_{j=1}^pt_j^{{2-\alpha\over 2}l_j}\,. \nonumber
	\end{align}

	For each $(l_1,\cdots, l_p)$,   we can write the above multiple integral as 
	$$
	\begin{aligned}
		&\int_{(\R^+)^p}dt_1\cdots dt_p
		(t_1+\cdots +t_p)^n\exp\Big\{-{1\over 2}\sum_{j=1}^p \lambda_j^2t_j\Big\}
		\prod_{j=1}^pt_j^{{2-\alpha\over 2}l_j}\\
		&\qquad=\sum_{k_1+\cdots+k_p=n}{n!\over k_1!\cdots k_p!}\int_{(\R^+)^p}dt_1\cdots dt_p
		\bigg(\prod_{j=1}^p t_j^{k_j+{2-\alpha\over 2}l_j}\bigg)
		\exp\Big\{-{1\over 2}\sum_{j=1}^p \lambda_j^2t_j\Big\}\\
		&\qquad=\sum_{k_1+\cdots+k_p=n}{n!\over k_1!\cdots k_p!}\prod_{j=1}^p
		\Big({2\over\lambda_j^2}\Big)^{k_j+{2-\alpha\over 2}l_j+1}\Gamma\Big(k_j+{2-\alpha\over 2}l_j+1\Big)\,. 
	\end{aligned}
	$$
	In the sequel, we shall use  the Stirling formula of the following form: 
	$$
	n^ne^{-n}\le \Gamma(n+1)\le n^{n+1}e^{-{n+1}}\,, \hskip.2in n=1,2,\cdots
	$$
	By using this type of Stirling's  formula and by routine simplification
	$$
	\begin{aligned}
		&\sum_{l_1+\cdots+l_p=n}
		\bigg(\prod_{j=1}^p(l_j!)^{-{2-\alpha\over 2}}\bigg)
		\int_{(\R^+)^p}dt_1\cdots dt_p(t_1+\cdots +t_p)^n\exp\Big\{-{1\over 2}\sum_{j=1}^p \lambda_j^2t_j\Big\}
		\prod_{j=1}^pt_j^{{2-\alpha\over 2}l_j}\\
		&\qquad\le n!C^p\Big({2-\alpha\over 2}\Big)^{{2-\alpha\over 2}n}
		\Big(2\sum_{j=1}^p{1\over\lambda_j^2}\Big)^{{4-\alpha\over 2}n}\\
		&\qquad\qquad\qquad \times \sum_{\stackrel{\scriptstyle k_1+\cdots +k_p =n}{\scriptstyle l_1+\cdots +l_p =n}}\prod_{j=1}^p
		\bigg({k_j+{2-\alpha\over 2}l_j\over k_j}\bigg)^{k_j}
		\bigg({k_j+{2-\alpha\over 2}l_j\over {2-\alpha\over 2}l_j}\bigg)^{{2-\alpha\over 2}l_j}
		\theta_j^{k_j+{2-\alpha\over 2}l_j}\,, 
	\end{aligned}
	$$
	where $C>0$ is a constant independent of $n$ and $p$, and
	$$
	\theta_j=\Big({1\over\lambda_1^2}+\cdots +{1\over\lambda_p^2}\Big)^{-1}{1\over\lambda_j^2}\,, \hskip.2in
	j=1,\cdots, p\,. 
	$$
	
	It is straightforward to check that the Lagrange problem
	\[
	\begin{aligned}
		\max\bigg\{&\prod_{j=1}^p\Big({x_j+y_j\over x_j}
		\Big)^{x_j}\Big({x_j+y_j\over y_j}\Big)^{y_j}\theta_j^{x_j+y_j};
		\hskip.1in x_1+\cdots +x_p=n,\hskip.05in \\
		& \hbox{and}\hskip.1in y_1+\cdots +y_p={2-\alpha\over 2} n\,,  x_1,\cdots, x_p, y_1,\cdots, y_p> 0\bigg\}
	\end{aligned}
	\]
	has the solution
	$$
	x_j=\theta_jn\hskip.1in\hbox{and}\hskip.1in y_j={2-\alpha\over 2}\theta_jn\,, \hskip.2in j=1,\cdots, p\,. 
	$$
	Therefore, since $\sum_{j=1}^p\theta_j=1$
	$$
	\begin{aligned}
		&\prod_{j=1}^p
		\bigg({k_j+{2-\alpha\over 2}l_j\over k_j}\bigg)^{k_j}
		\bigg({k_j+{2-\alpha\over 2}l_j\over {2-\alpha\over 2}l_j}\bigg)^{{2-\alpha\over 2}l_j}
		\theta_j^{k_j+{2-\alpha\over 2}l_j}\\
		&\qquad\le \prod_{j=1}^p\Big({4-\alpha\over 2}\Big)^{\theta_jn}
		\bigg({4-\alpha\over 2-\alpha}\bigg)^{{2-\alpha\over 2}\theta_j n}
		\theta_j^{{4-\alpha\over 2}\theta_jn}\\
		&\qquad=\Big({4-\alpha\over 2}\Big)^{n}\bigg({4-\alpha\over 2-\alpha}\bigg)^{{2-\alpha\over 2}n}
		\prod_{j=1}^p\theta_j^{{4-\alpha\over 2}\theta_jn}
		=\bigg({4-\alpha\over 2}\bigg)^{{4-\alpha\over 2}n}\bigg({2\over 2-\alpha}\bigg)^{{2-\alpha\over 2}n}
		\prod_{j=1}^p\theta_j^{{4-\alpha\over 2}\theta_jn}
	\end{aligned}
	$$
	uniformly over $l_1,\cdots, l_p;k_1,  \cdots, k_p$. 
	
	Summarizing our steps since (\ref{LDP-8}) and noticing
	$$
	\sum_{l_1+\cdots+l_p=n}1=\left(\begin{array}{c} n+p-1\\ p-1\end{array}\right)
	$$
	we have the bound
	\begin{align}\label{LDP-9}
		\int_{(\R^+)^p}dt_1\cdots dt_p&
		\exp\Big\{-\sum_{j=1}^p\lambda_jt_j\Big\}\sum_{l_1+\cdots+l_p=2n}
		\E \prod_{j=1}^p S_{l_j}\big(g_{l_j}(\cdot, t_j,0)\big)\\
		&\le C^pn!\left(\begin{array}{c} n+p-1\\ p-1\end{array}\right)^2
		\bigg(\prod_{j=1}^p{C_\delta\lambda_j\over 2}\bigg)
		\bigg(2(1+\delta)\Big({{\cal M}\over 4-\alpha}\Big)^{4-\alpha\over 2}\bigg)^n\nonumber\\
		&\qquad \times \bigg((4-\alpha)\sum_{j=1}^p{1\over \lambda_j^2}\bigg)^{{4-\alpha\over 2}n}
		\prod_{j=1}^p\theta_j^{{4-\alpha\over 2}\theta_jn}\,. \nonumber
	\end{align}
	This leads to (\ref{LDP-6}) as $\delta>0$ can be made arbitrarily small.
	
	The bound (\ref{LDP-9}) can also be used to   the proof of (\ref{LDP-7}).
	To see  this  we can  allow $p$ tends to infinity only along integer points.
		This does not compromise the claim there 
		by the following interpolation argument:
		For any real
		and large $p\ge 1$, let $\langle p/2\rangle$ be the smallest integer larger than  or equal to $p/2$.
		Then,  by H\"older's inequality we have 
		$$
		\Big\{\E \vert u(t,0)\vert^p\Big\}^{1/p}\le\Big\{\E u^{2\langle p/2\rangle}(t,0)
		\Big\}^{1\over 2\langle p/2\rangle}\,. 
		$$
		Thus, it suffices to show  (\ref{LDP-7}) along the positive integers $p$ and
		with $\E u^p(t,0)$ instead of $\E\vert u(t,0)\vert^p$.

		By monotonicity of $g_n(\cdot, t,0)$ in $t$,
		$$
		\begin{aligned}
			\sum_{l_1+\cdots+l_p=2n} \E &\prod_{j=1}^p S_{l_j}\big(g_{l_j}(\cdot, t_j,0)\big)
			\ge \sum_{l_1+\cdots+l_p=2n} \E \prod_{j=1}^p S_{l_j}\Big(g_{l_j}\big(\cdot, \min_{1\le j\le p}t_j,0\big)\Big)\\
			&=\bigg(\sum_{l_1+\cdots+l_p=2n} \E \prod_{j=1}^p S_{l_j}\big(g_{l_j}(\cdot, 1,0)\big)\bigg)
			\Big(\min_{1\le j\le p}t_j\Big)^{(4-\alpha)n}\,, 
		\end{aligned}
		$$
		where the   last step follows from (\ref{LDP-1}). Thus,
		$$
		\begin{aligned}
			&\int_{(\R^+)^p}dt_1\cdots dt_p
			\exp\Big\{-\sum_{j=1}^pt_j\Big\}\sum_{l_1+\cdots+l_p=2n}\E \prod_{j=1}^p
			S_{l_j}\big(g_{l_j}(,\cdot, t_j,0)\big)\\
			&\qquad \ge \bigg(\sum_{l_1+\cdots+l_p=2n} \E \prod_{j=1}^p S_{l_j}\big(g_{l_j}(\cdot, 1,0)\big)\bigg)
			\int_{(\R^+)^p}dt_1\cdots dt_p\exp\Big\{-\sum_{j=1}^pt_j\Big\}
			\Big(\min_{1\le j\le p}t_j\Big)^{(4-\alpha)n}\,. 
		\end{aligned}
		$$
		By the fact that given i.i.d. exponential times $\tau_1,\cdots,\tau_p$ of parameter 1,
		$\displaystyle \min_{1\le j\le p}\tau_j$ is an exponential time with parameter $p$,
		\begin{align}\label{LDP-11}
			\int_{(\R^+)^p}dt_1\cdots dt_p &\exp\Big\{-\sum_{j=1}^pt_j\Big\}
			\Big(\min_{1\le j\le p}t_j\Big)^{(4-\alpha)n}
			=p\int_0^\infty e^{-pt} t^{(4-\alpha)n}dt\\
			&=p^{-(4-\alpha)n}\Gamma\Big(1+(4-\alpha)n\Big)\,. \nonumber
		\end{align}
		In summary,  we have 
		\begin{align}\label{LDP-12}
			&\sum_{l_1+\cdots+l_p=2n} \E \prod_{j=1}^p S_{l_j}\big(g_{l_j}(\cdot, 1,0)\big)\\
			&\qquad \le {p^{(4-\alpha)n}\over \Gamma\big(1+(4-\alpha)n\big)}\int_{(\R^+)^p}dt_1\cdots dt_p
			\exp\Big\{-\sum_{j=1}^pt_j\Big\}\sum_{l_1+\cdots+l_p=2n}\E \prod_{j=1}^p
			S_{l_j}\big(g_{l_j}(,\cdot, t_j,0)\big)\nonumber\\
			&\qquad \le\Big({CC_\delta\over 2}\Big)^p\left(\begin{array}{c} n+p-1\\ p-1\end{array}\right)^2
			{n!p^{(4-\alpha)n}\over \Gamma\big(1+(4-\alpha)n\big)}\bigg(2(1+\delta){\cal M}^{4-\alpha\over 2}\bigg)^n\,, 
			\nonumber
		\end{align}
		where the second step follows directly from the bound (\ref{LDP-9}) with
		$\lambda_1=\cdots=\lambda_p=1$.
		
		Using (\ref{LDP-1}) and \eqref{LDP-2} we  then have
		$$
		\begin{aligned}
			\E u^p(t,0)&\le\Big({CC_\delta\over 2}\Big)^p\sum_{n=0}^\infty
			\left(\begin{array}{c} n+p-1\\ p-1\end{array}\right)^2
			{n!(pt)^{(4-\alpha)n}\over \Gamma\big(1+(4-\alpha)n\big)}\bigg(2(1+\delta)
			{\cal M}^{4-\alpha\over 2}\bigg)^n\\
			&\le\Big({CC_\delta\over 2}{\theta\over \theta -1}\Big)^{2p}\sum_{n=0}^\infty
			{n!(pt) ^{(4-\alpha)n}\over \Gamma\big(1+(4-\alpha)n\big)}\bigg(2\theta^2
			(1+\delta){\cal M}^{4-\alpha\over 2}\bigg)^n\,, 
		\end{aligned}
		$$
		where $\theta>1$ is arbitrary, and the second  step follows from the estimate
		$$
		\theta^{-n}\left(\begin{array}{c} n+p-1\\ p-1\end{array}\right)\le\sum_{k=0}^\infty\theta^{-k}
		\left(\begin{array}{c} k+p-1\\ p-1\end{array}\right)=\Big({\theta\over \theta-1}\Big)^p\,. 
		$$
		By the Stirling formula, $\Gamma\big(1+(4-\alpha)n\big)$ is replaceable by
		$$
		(n!)^{4-\alpha}(4-\alpha)^{(4-\alpha)n}\,. 
		$$
		By  the asymptotics of the Mittag-Leffler function (\ref{LDP-5}) with $\gamma= 3-\alpha$ and $b=p^{4-\alpha}$,  and $\theta$
		being replaced by $t ^{ 4-\alpha } \bigg(2\theta^2
		(1+\delta)\Big({{\cal M}^{1/2}\over 4-\alpha}\Big)^{4-\alpha}\bigg) $, we have 
		$$
		\begin{aligned}
			\limsup_{p\to\infty}&p^{-{4-\alpha\over 3-\alpha}}\log \E u^p(t,0)\\
			& \le \lim_{p\to\infty}p^{-{4-\alpha\over 3-\alpha}}\log \sum_{n=0}^\infty
			{(pt)^{(4-\alpha)n}\over (n!)^{3-\alpha}}\bigg(2\theta^2
			(1+\delta)\Big({{\cal M}^{1/2}\over 4-\alpha}\Big)^{4-\alpha}\bigg)^n\\
			&=(3-\alpha)t^{4-\alpha\over 3-\alpha}\bigg(2\theta^2
			(1+\delta)\Big({{\cal M}^{1/2}\over 4-\alpha}\Big)^{4-\alpha}\bigg)^{1\over 3-\alpha}\,. 
		\end{aligned}
		$$
		Letting $\delta\to 0^+$ and $\theta\to 1^+$ on the right hand side gives (\ref{LDP-7}).
	\end{proof}

	We end  this subsection  by the following statement: First, taking $\lambda_1=\cdots =\lambda_p=1$
	in (\ref{LDP-6}) leads to
	\begin{align}\label{LDP-13}
		&\limsup_{n\to\infty}{1\over n}\log {1\over n!}\int_{(\R^+)^p}dt_1\cdots dt_p
		\exp\Big\{-\sum_{j=1}^p t_j\Big\}
		\bigg(\sum_{l_1+\cdots+l_p=2n}
		\E \prod_{j=1}^pS_{l_j}\big(g_{l_j}(\cdot, t_j,0)\big)
		\bigg)\nonumber\\
		&\qquad \le\log 2{\cal M}^{4-\alpha\over 2}\,. 
	\end{align}
	
	Second, applying  (\ref{LDP-12}) to the setting of fixed integer $p\ge 1$ yields 
	\begin{align}\label{LDP-14}
		&\limsup_{n\to\infty}{1\over n}\log
		(n!)^{3-\alpha}\bigg(\sum_{l_1+\cdots+l_p=2n}
		\E \prod_{j=1}^p S_{l_j}\big(g_{l_j}(\cdot, 1,0)\big)\bigg)\nonumber \\
		&\qquad\qquad \le\log\Big({1\over 2}
		\Big)^{3-\alpha} p^{4-\alpha}\bigg({2{\cal M}^{1/2}\over 4-\alpha}
		\bigg)^{4-\alpha}\,. 
	\end{align}

	\subsection {Lower bound for (\ref{intro-9})}
	
	In this subsection we start by the lower bound correspondent to (\ref{LDP-13}).
	
	\begin{lemma}\label{L-7} Under the condition in Theorem \ref{th-2},  we have 
		\begin{align}\label{LDP-15}
			&\liminf_{n\to\infty}{1\over n}\log {1\over n!}\int_{(\R^+)^p}dt_1\cdots dt_p
			\exp\Big\{-\sum_{j=1}^pt_j\Big\}\sum_{l_1+\cdots+l_p=2n}
			\E \prod_{j=1}^pS_{l_j}\big(g_{l_j}(\cdot, t_j,0)\big)\nonumber\\
			&\qquad \ge\log 2{\cal M}^{4-\alpha\over 2}
		\end{align}
		for $p=1,2,\cdots$.
	\end{lemma}

	\begin{proof} Notice
		$$
		\begin{aligned}
			&\sum_{j,k=1}^p\int_0^{t_j}\!\int_0^{t_k}
			\gamma\big(B_j(s)-B_k(r)\big)dsdr
			=\int_{\R^d}\mu(d\xi)\bigg\vert\sum_{j=1}^p\int_0^{t_j}e^{i\xi\cdot B_j(s)}ds
			\bigg\vert^2\\
			&\qquad \ge\bigg[\int_{\R^d}\mu(d\xi)f(\xi)\bigg(\sum_{j=1}^p\int_0^{t_j}e^{i\xi\cdot B_j(s)}ds\bigg)\bigg]^2\\
			&\qquad =\bigg[\sum_{j=1}^p\int_{\R^d}\mu(d\xi)f(\xi)\bigg(\int_0^{t_j}e^{i\xi\cdot B_j(s)}ds\bigg)\bigg]^2
		\end{aligned}
		$$
		for any non-negative $f(\xi)$ with
		\begin{align}\label{LDP-16}
			\int_{\R^d}\vert f(\xi)\vert^2\mu(d\xi)=1\,. 
		\end{align}
		Therefore, we have  
		$$
		\begin{aligned}
			&\E_0\bigg[\sum_{j,k=1}^p\int_0^{t_j}\!\int_0^{t_k}
			\gamma\big(B_j(s)-B_k(r)\big)dsdr\bigg]^n
			\ge\E_0\bigg[\sum_{j=1}^p\int_{\R^d}\mu(d\xi)f(\xi)\bigg(\int_0^{t_j}e^{i\xi\cdot B_j(s)}ds\bigg)\bigg]^{2n}\\
			&\qquad=\sum_{l_1+\cdots +l_p=2n}{(2n)!\over l_1!\cdots l_p!}\prod_{j=1}^p\E_0
			\bigg[\int_{\R^d}\mu(d\xi)f(\xi)\bigg(\int_0^{t_j}e^{i\xi\cdot B_j(s)}ds\bigg)\bigg]^{l_j}\\
			&\qquad=(2n)!\sum_{l_1+\cdots +l_p=2n}\prod_{j=1}^p\int_{(\R^d)^{l_j}}
			\mu(d\xi)\bigg(\prod_{k=1}^{l_j}f(\xi_k)\bigg)\int_{[0,t_j]_<^{l_j}}d{\bf s}
			\prod_{k=1}^{l_j}\exp\bigg\{-{s_k-s_{k-1}\over 2}\Big\vert\sum_{i=k}^{l_j}\xi_i\Big\vert^2\bigg\}\,. 
		\end{aligned}
		$$

		Taking $\lambda_1=\cdots=\lambda_p=1$ in  Corollary \ref{co-1} and inserting the above computation into the obtained expression yield  
		\begin{align}\label{LDP-17}
			&\int_{(\R^+)^p}dt_1\cdots dt_p
			\exp\Big\{-\sum_{j=1}^pt_j\Big\}\sum_{l_1+\cdots+l_p=2n}
			\E \prod_{j=1}^pS_{l_j}\big(g_{l_j}(\cdot, t_j,0)\big)\\
			&\qquad \ge \Big({1\over 2}\Big)^p\Big({1\over 2}\Big)^{3n}{(2n)!\over n!}  
			\int_{(\R^+)^p}dt_1\cdots dt_p
			\exp\Big\{-\sum_{j=1}^pt_j\Big\}  \nonumber\\
			&\qquad\qquad \times\sum_{l_1+\cdots +l_p=2n}\prod_{j=1}^p\int_{(\R^d)^{l_j}}
			\mu(d\xi)\bigg(\prod_{k=1}^{l_j}f(\xi_k)\bigg)\int_{[0,t_j]_<^{l_j}}d{\bf s}
			\prod_{k=1}^{l_j}\exp\bigg\{-{s_k-s_{k-1}\over 2}\Big\vert\sum_{i=k}^{l_j}\xi_i\Big\vert^2\bigg\} \nonumber\\
			&\qquad \ge \Big({1\over 2}\Big)^{p+3n} 
			{(2n)!\over n!}
			\sum_{l_1+\cdots +l_p=2n}\prod_{j=1}^p\int_{(\R^d)^{l_j}}
			\mu(d\xi)\bigg(\prod_{k=1}^{l_j}f(\xi_k)\bigg)\nonumber\\
			&\qquad\qquad \times\int_0^\infty dt e^{-t }\int_{[0,t]_<^{l_j}}d{\bf s}
			\prod_{k=1}^{l_j}\exp\bigg\{-{s_k-s_{k-1}\over 2}\Big\vert\sum_{i=k}^{l_j}\xi_i\Big\vert^2\bigg\}
			\nonumber\\
			&\qquad =\Big({1\over 2}\Big)^{p+3n}{(2n)!\over n!}
			\sum_{l_1+\cdots +l_p=2n}\prod_{j=1}^p\int_{(\R^d)^{l_j}}
			\mu(d\xi)\bigg(\prod_{k=1}^{l_j}f(\xi_k)\bigg)\prod_{k=1}^{l_j} \int_0^\infty e^{-t }
			\exp\bigg\{-{t\over 2}\Big\vert\sum_{i=k}^{l_j}\xi_i\Big\vert^2\bigg\}dt\nonumber\\
			&\qquad =\Big({1\over 2}\Big)^{p+3n}{(2n)!\over n!}
			\sum_{l_1+\cdots +l_p=2n}\prod_{j=1}^p\int_{(\R^d)^{l_j}}
			\mu(d\xi)\bigg(\prod_{k=1}^{l_j}f(\xi_k)\bigg) \prod_{k=1}^{l_j}
			\bigg\{1+\frac{1}{2} \Big\vert\sum_{i=k}^{l_j}\xi_i\Big\vert^2\bigg\}^{-1}  \nonumber\\
			&\qquad \ge \Big({1\over 2}\Big)^{2n}{(2n)!\over n!}
			\sum_{l_1+\cdots +l_p=2n}\prod_{j=1}^p\int_{(\R^d)^{l_j}}
			\mu(d\xi)\bigg(\prod_{k=1}^{l_j}f(\xi_k)\bigg) \prod_{k=1}^{l_j}
			\bigg\{1+ \Big\vert\sum_{i=k}^{l_j}\xi_i\Big\vert^2\bigg\}^{-1}\,. \nonumber
		\end{align}

		By the computation in   \cite[(3.7)-(3.9)]{BCR}
		\begin{align}\label{LDP-18}
			\liminf_{n\to\infty}{1\over n}&\log \int_{(\R^d)^n}\mu(d\xi)\bigg(\prod_{k=1}^{n}f(\xi_k)\bigg)
			\prod_{k=1}^n\bigg\{1+\Big\vert\sum_{i=k}^{n}\xi_i\Big\vert^2\bigg\}^{-1}\\
			&\ge\log
			\sup_{\|\varphi\|_2=1}\int_{\R^d}\mu(d\xi)f(\xi)\bigg[\int_{\R^d}d\eta {\varphi(\eta)\varphi(\eta+\xi)\over
				\sqrt{(1+\vert\eta\vert^2)(1+\vert\xi+\eta\vert^2)}}\bigg]\nonumber\\
			&\buildrel\Delta\over =\log\rho(f)\,. \nonumber
		\end{align}
		For a given $\delta>0$, therefore, there is $C_\delta>0$ such that
		$$
		\int_{(\R^d)^n}\mu(d\xi)\bigg(\prod_{k=1}^{n}f(\xi_k)\bigg)
		\prod_{k=1}^n\bigg\{1+\Big\vert\sum_{i=k}^{n}\xi_i\Big\vert^2\bigg\}^{-1}
		\ge C_\delta^{-1}\Big((1-\delta)\rho(f)\Big)^n\,, \hskip.2in n=1,2,\cdots
		$$
		Together with (\ref{LDP-17}), by the Stirling formula one has 
		$$
		\begin{aligned}
			\liminf_{n\to\infty}{1\over n}&\log \int_{(\R^+)^p}dt_1\cdots dt_p
			\exp\Big\{-\sum_{j=1}^pt_j\Big\}\sum_{l_1+\cdots+l_p=2n}
			\E \prod_{j=1}^pS_{l_j}\big(g_{l_j}(\cdot, t_j,0)\big)\\
			&                  \ge 2\Big((1-\delta)\rho(f)\Big)^2\,. 
		\end{aligned}
		$$
		Letting $\delta\to 0^+$ and taking supremum over all non-negative functions
		$f$ satisfying (\ref{LDP-16})
		on the right hand side,  we  have   
		$$
		\begin{aligned}
			\liminf_{n\to\infty}{1\over n}&\log \int_{(\R^+)^p}dt_1\cdots dt_p
			\exp\Big\{-\sum_{j=1}^pt_j\Big\}\sum_{l_1+\cdots+l_p=2n}
			\E \prod_{j=1}^pS_{l_j}\big(g_{l_j}(\cdot, t_j,0)\big)\\
			&\ge \log 2 \sup_{\|\varphi\|_2=1} \int_{\R^d}\mu(d\xi)\bigg[\int_{\R^d}d\eta {\varphi(\eta)
				\varphi(\eta+\xi)\over
				\sqrt{(1+\vert\eta\vert^2)(1+\vert\xi+\eta\vert^2)}}\bigg]^2\,. 
		\end{aligned}
		$$
		Finally, the proof is completed by Theorem 1.5, \cite{BCR} (with $p=\beta =2$,
		$\sigma=\alpha$ and $\vert\cdot\vert^{-\alpha}$ being replaced by $\gamma(\cdot)$) that states
		\begin{align}\label{LDP-19}
			\sup_{\|\varphi\|_2=1} \int_{\R^d}\mu(d\xi)\bigg[\int_{\R^d}d\eta {\varphi(\eta)\varphi(\eta+\xi)\over
				\sqrt{(1+\vert\eta\vert^2)(1+\vert\xi+\eta\vert^2)}}\bigg]^2
			={\cal M}^{4-\alpha\over 2}\,. 
		\end{align}
		This completes the proof \eqref{LDP-15}.  
	\end{proof}
	
	Combining (\ref{LDP-13}) with  (\ref{LDP-15}) yields 
	\begin{align}\label{LDP-20}
		&\lim_{n\to\infty}{1\over n}\log {1\over n!}\int_{(\R^+)^p}dt_1\cdots dt_p
		\exp\Big\{-\sum_{j=1}^p t_j\Big\}
		\bigg(\sum_{l_1+\cdots+l_p=2n}
		\E \prod_{j=1}^pS_{l_j}\big(g_{l_j}(\cdot, t_j,0)\big)\bigg)\nonumber\\
		&\qquad =\log 2{\cal M}^{4-\alpha\over 2}\,. 
	\end{align}
	
	We point out that we are not able to establish the lower bound correspondent to
	(\ref{LDP-6}) as $\lambda_1,\cdots,\lambda_p$ are not equal,
	although it likely to be true.
	
	Our next goal in this  sub-section is to establish the lower bound
	corresponding to the upper bound (\ref{LDP-14}). 
	
	\begin{lemma}\label{L-5}
		Under the condition in Theorem \ref{th-2},  we have 
		\begin{align}\label{LDP-21}
			\liminf_{n\to\infty}{1\over n}\log
			(n!)^{3-\alpha}&\bigg(\sum_{l_1+\cdots+l_p=2n}
			\E \prod_{j=1}^p S_{l_j}\big(g_{l_j}(\cdot, 1,0)\big)\bigg)\\
			& \ge \log\Big({1\over 2}\Big)^{3-\alpha} p^{4-\alpha}\bigg({2{\cal M}^{1/2}\over 4-\alpha}\bigg)^{4-\alpha}\,. \nonumber
		\end{align}
	\end{lemma}

	\begin{proof} We adopt some idea in the proof for the lower bound of G\"artner-Ellis large deviations (Theorem 2.3.6,, p.44, \cite{DZ}).
		The crucial observation is the concentration behavior
		$t_1,\cdots, t_p\approx  (4-\alpha)n$ (as $n\to\infty$)
		in a dynamics that creates (\ref{LDP-20}).
		To show it, we
		define a probability measures on $(\R^+)^p$ as follows 
		$$
		\mu_n(A)={\displaystyle\int_Adt_1\cdots dt_p\exp\big\{-(t_1+\cdots +t_p)\big\}
			\sum_{l_1+\cdots+l_p=2n}
			\E \prod_{j=1}^p S_{l_j}\big(g_{l_j}(\cdot,t_j,0)\big)\over
			\displaystyle\int_{(\R^+)^p}dt_1\cdots dt_p\exp\big\{-(t_1+\cdots +t_p)\big\}
			\sum_{l_1+\cdots+l_p=2n}\E \prod_{j=1}^p S_{l_j}\big(g_{l_j}(\cdot,t_j,0)\big)}
		$$
		for $n=1,2,\cdots$. 
		Notice that for any $\theta_1,\cdots, \theta_p<1$,
		$$
		\begin{aligned}
			&\int_{(\R^+)^p}dt_1\cdots dt_p\exp\big\{\theta_1t_1+\cdots +\theta_pt_p)\big\}
			\mu_n(dt_1\cdots dt_p)\\
			&={\displaystyle\int_{(\R^+)^p}dt_1\cdots dt_p 
				\exp\Big\{-\sum_{j=1}^p(1-\theta_j) t_j\Big\}
				\bigg(\sum_{l_1+\cdots+l_p=2n}
				\E \prod_{j=1}^pS_{l_j}\big(g_{l_j}(\cdot, t_j,0)\big)\bigg)\over
				\displaystyle\int_{(\R^+)^p}dt_1\cdots dt_p 
				\exp\Big\{-\sum_{j=1}^p t_j\Big\}
				\bigg(\sum_{l_1+\cdots+l_p=2n}
				\E \prod_{j=1}^pS_{l_j}\big(g_{l_j}(\cdot, t_j, 0)\big)\bigg)}
		\end{aligned}
		$$
		and the right hand side
		blows up as long  as $\theta_j\ge 1$ for any $1\le j\le p$.
		
		By (\ref{LDP-6}) and (\ref{LDP-20}),  we see 
		\begin{align}\label{LDP-22}
			\limsup_{n\to\infty}{1\over n}&\log \int_{(\R^+)^p}dt_1\cdots dt_p
			\exp\big\{\theta_1t_1+\cdots +\theta_pt_p)\big\}\mu_n(dt_1\cdots dt_p)\\
			&\le \Lambda(\theta_1,\cdots, \theta_p)\nonumber
		\end{align}
		for any $(\theta_1,\cdots,\theta_p)\in\R^p$, where
		\[ 
		\Lambda(\theta_1,\cdots,\theta_p)
		:=\begin{cases}  
			\displaystyle {4-\alpha\over 2}
			\sum_{j=1}^p{(1-\theta_j)^{-2}\log(1-\theta_j)^{-2}\over (1-\theta_1)^{-2}+\cdots+(1-\theta_p)^{-2}}
			&\hskip.2in \hbox{if}\ \theta_1,\cdots,\theta_p<1\\\\
			\infty&\hskip.2in \hbox{otherwise} \,.  
		\end{cases}
		\]
		
		By the upper bound of G\"artner-Ellis theorem (Theorem 2.3.6 (a), p.44,
		\cite{DZ})
		\begin{align}\label{LDP-23}
			\limsup_{n\to\infty}{1\over n}\log \mu_n(nF)\le -\inf_{(t_1,\cdots,t_p)\in F}
			\Lambda^*(t_1,\cdots, t_p)
		\end{align}
		for any close $F\subset(\R^+)^p$,
		where
		$$
		\Lambda^*(t_1,\cdots, t_p)=\sup_{\theta_1,\cdots,\theta_p<1}\Big\{
		\sum_{j=1}^p\theta_jt_j-\Lambda(\theta_1,\cdots,\theta_p)\Big\}\,,  
		\hskip.2in t_1,\cdots, t_p\ge 0\,. 
		$$
		In fact, the statement of Theorem 2.3.6 (a), p.44,\cite{DZ}
		requires the equality in (\ref{LDP-22}).  However, 
		a careful reading of its proof   finds that (\ref{LDP-22}) is sufficient for
		(\ref{LDP-23}).

		Finding the close form of $\Lambda^*(\theta_1,\cdots,\theta_p)$ might
		not be easy. On the other hand, some  properties of
		$\Lambda^*(\theta_1,\cdots,\theta_p)$ as a rate function exists even
		in the general context. For example, $\Lambda^*(\theta_1,\cdots,\theta_p)$
		is non-negative, lower semi-continuous and has compact level sets 
		(goodness).
		What important to our purpose is that
		\begin{align}\label{LDP-24}
			\Lambda^*(t_1,\cdots, t_p)>0 \,, \hskip.2in \forall (t_1,\cdots, t_p)\not =
			\big(4-\alpha,\cdots, 4-\alpha)\,. 
		\end{align} 
		Indeed, assume that $\Lambda^*(t_1,\cdots, t_p)=0$ for some $(t_1,\cdots, t_p)$. 
		Then we have that
		$$
		\sum_{j=1}^p\theta_jt_j\le {4-\alpha\over 2}
		\sum_{j=1}^p{(1-\theta_j)^{-2}\log(1-\theta_j)^{-2}\over (1-\theta_1)^{-2}+\cdots+(1-\theta_p)^{-2}} 
		$$
		for any $\theta_1,\cdots, \theta_p<1$. For fixed $1\le j\le p$,  taking 
		$\theta_k=0$ for all $k\not =j$, the above inequality gives
		$$
		\theta_jt_j\le {4-\alpha\over 2}\log (1-\theta_j)^{-2}
		=(4-\alpha)\log (1-\theta_j)^{-1}\,, \hskip.2in \forall\theta_j<1\,. 
		$$
		So we have that
		$$
		t_j\le (4-\alpha){1\over\theta_j}\log (1-\theta_j)^{-1}\hskip.2in\hbox{as
			$\theta_j>0$}
		$$
		and
		$$
		t_j\ge (4-\alpha){1\over\theta_j}\log (1-\theta_j)^{-1}\hskip.2in\hbox{as
			$\theta_j<0$}\,. 
		$$
		By the fact that
		$$
		\lim_{\theta\to 0}{1\over\theta}\log (1-\theta)^{-1}=1
		$$
		we have $t_j=4-\alpha$ ($j=1,\cdots, p$).  This shows the claim \eqref{LDP-24}. 
		
		By (\ref{LDP-24}), the  lower semi-continuity and  goodness we have 
		$$
		\inf_{(t_1,\cdots, t_p)\not\in G}\Lambda^*(t_1,\cdots, t_p)>0
		$$
		for any open neighborhood $G$ of $(4-\alpha,\cdots, 4-\alpha)$.
		For any given  small $\delta>0$   taking $G_\delta=(-(4-\alpha), 4-\alpha)^p$
		and $F=G_\delta^c$ in (\ref{LDP-23})  yields 
		$$
		\limsup_{n\to\infty}{1\over n}\log \mu_n(nG_\delta^c)<0\,. 
		$$
		Consequently,
		\begin{align}\label{LDP-25}
			&\int_{(-(4-\alpha-\delta)n, (4-\alpha+\delta)n)^p}
			dt_1\cdots dt_p
			\exp\Big\{-\sum_{j=1}^p t_j\Big\}
			\bigg(\sum_{l_1+\cdots+l_p=2n}
			\E \prod_{j=1}^pS_{l_j}\big(g_{l_j}(\cdot, t_j,0)\big)\bigg)\nonumber\\
			&\qquad \sim \int_{(\R^+)^p}
			dt_1\cdots dt_p\exp\Big\{-\sum_{j=1}^p t_j\Big\}
			\bigg(\sum_{l_1+\cdots+l_p=2n}
			\E \prod_{j=1}^pS_{l_j}\big(g_{l_j}(\cdot, t_j,0)\big)\bigg)
		\end{align}                                    
		as $n\to\infty$.
		
		When
		$(t_1,\cdots, t_p)\in\big(n(4-\alpha-\delta), n(4-\alpha+\delta)\big)^p$,  it is easy to see that 
		$$
		t_j\le (4-\alpha+\delta)n\le {4-\alpha+\delta\over 4-\alpha-\delta}
		\min_{1\le k\le p}t_k\,, 
		\hskip.2in j=1,\cdots, p  
		$$
		and by the scaling property (\ref{LDP-1}) we have
		$$
		\begin{aligned}
			\sum_{l_1+\cdots+l_p=2n}\E \prod_{j=1}^p&
			S_{l_j}\big(g_{l_j}(\cdot, t_j,0)\big)\le
			\sum_{l_1+\cdots+l_p=2n}\E \prod_{j=1}^p
			S_{l_j}\bigg(g_{l_j}\Big(\cdot, {4-\alpha+\delta\over 4-\alpha-\delta}
			\min_{1\le k\le p}t_k, 0\Big)\bigg)\\
			&=\Big({4-\alpha+\delta\over 4-\alpha-\delta}\min_{1\le k\le p}t_k
			\Big)^{(4-\alpha)n}
			\sum_{l_1+\cdots+l_p=2n}\E \prod_{j=1}^p
			S_{l_j}\big(g_{l_j}(\cdot, 1, 0)\big)\,. 
		\end{aligned}
		$$
		
		Therefore,
		$$
		\begin{aligned}
			&\int_{(n(4-\alpha-\delta), n(4-\alpha+\delta))^p} 
			dt_1\cdots dt_p\exp\Big\{-\sum_{j=1}^pt_j\Big\}\sum_{l_1+\cdots+l_p=2n}
			\E \prod_{j=1}^p  S_{lj}\big(g_n(\cdot, t_j,0)\big)\\
			&\qquad\quad  \le \bigg\{\sum_{l_1+\cdots+l_p=2n}\E \prod_{j=1}^p
			S_{l_j}\big(g_{l_j}(\cdot,1,0)\big)\bigg\}
			\Big({4-\alpha+\delta\over 4-\alpha-\delta}\Big)^{(4-\alpha)n}\\
			&\qquad\quad\qquad\quad\times\int_{(\R^+)^p}dt_1\cdots dt_p\exp\Big\{-\sum_{j=1}^pt_j\Big\}\Big(\min_{1\le j\le p}t_j
			\Big)^{(4-\alpha)n}\\
			&\qquad\quad=\bigg\{\sum_{l_1+\cdots+l_p=2n}\E \prod_{j=1}^p
			S_{l_j}\big(g_{l_j}(\cdot, 1,0)\big)\bigg\}
			\Big({4-\alpha+\delta\over 4-\alpha-\delta}\Big)^{(4-\alpha)n}
			\Big({1\over p}\Big)^{(4-\alpha)n}\Gamma\Big(1+(4-\alpha)n\Big)\,, 
		\end{aligned}
		$$
		where the last step follows from (\ref{LDP-11}). 
		Finally, (\ref{LDP-21}) follows from the above inequality together with  (\ref{LDP-20}), (\ref{LDP-25})
		and the Stirling formula.
	\end{proof}
	
	\subsection{Lower bounds for  (\ref{intro-10})}
	
	In this subsection we prove the lower bound part of    (\ref{intro-9}):
	\begin{align}\label{LDP-26}
		\liminf_{p\to\infty}p^{-{4-\alpha\over 3-\alpha}}\log \E\vert u(t,0)\vert^p\ge 
		{3-\alpha\over 2}t^{4-\alpha\over 3-\alpha}\bigg({2\sqrt{\cal M}\over 4-\alpha}\bigg)^{4-\alpha\over 3-\alpha}\,. 
	\end{align}
	It should be pointed out that the G\"artner-Ellis type argument used for the proof of Lemma \ref{L-5} is good only for fixed $p$.
	Different from the approaches used thus far, the treatment below is independent of the
	Stratonovich moment representation developed in Section \ref{S}.
	
	Let ${\cal H}$ be the Hilbert space given as the closure of the  space
	$$
	\bigg\{f: \R^d\to \R;\hskip.1in\int_{\R^d\times\R^d}\gamma(x-y)f(x)f(y)dxdy<\infty\bigg\}
	$$
	under the inner product
	$$
	\langle f,g\rangle_{\cal H}=\int_{\R^d\times\R^d}\gamma(x-y)f(x)g(y)dxdy\,. 
	$$  
	The space $\mathcal{H}$ may contain generalized functions (distributions).  
	For each integer $n\ge 1$, we write ${\cal H}^{\otimes n}$ for the $n$-th product with inner
	product
	\begin{equation}
		\langle f, g\rangle_{{\cal H}^{\otimes n}}=\int_{(\R^d)^{2n}}d{\bf x}d{\bf y}
		\bigg(\prod_{k=1}^n\gamma(x_k-y_k)\bigg)f(x_1,\cdots, x_n)g(y_1,\cdots. y_n)\,.  \label{e.5.27} 
	\end{equation} 
	
	\begin{lemma}\label{L-6} Given any real number $p>1$,
		\begin{align}\label{LDP-27}
			\|u(t,0)\|_p\ge\exp\Big\{-{1\over 2(p-1)}\|f\|_{\cal H}^2\Big\}\sum_{n=0}^\infty
			\langle f^{\otimes n}, g_n(\cdot, t, 0)\rangle_{{\cal H}^{\otimes n}}
		\end{align}
		for any $t>0$ and $f\in {\cal H}$ with $f(\cdot)\ge 0$. 
	\end{lemma}
	
	\begin{proof} Let $q>1$ be the conjugate of $p$. By H\"older's  inequality
		$$
		\|u(t,x)\|_p\ge \E \left[ u(t,x) X\right] 
		$$
		for any random variable $X$ with $\|X\|_q=1$.
		Take
		$$
		\begin{aligned}
			X&=\bigg\|
			\exp\bigg\{\int_{\R^d}f(x)W(dx)\bigg\}\bigg\|_q^{-1}\exp\bigg\{\int_{\R^d}f(x)W(dx)\bigg\}\\
			&=\exp\Big\{-{q\over 2}\|f\|_{\cal H}^2\Big\}\exp\bigg\{\int_{\R^d}f(x)W(dx)\bigg\}\,. 
		\end{aligned}
		$$
		Then  for any $f\in \mathcal{H}$, 
		\begin{align}\label{LDP-28}
			\|u(t,0)\|_p&\ge\exp\Big\{-{q\over 2}\|f\|_{\cal H}^2\Big\}
			\E u(t,0)\exp\bigg\{\int_{\R^d}f(x)W(dx)\bigg\}\\
			&=\exp\Big\{-{q\over 2}\|f\|_{\cal H}^2\Big\}\sum_{n=0}^\infty\bigg\{\sum_{l=0}^n{1\over l!}
			\E \bigg(\int_{\R^d}f(x)W(dx)\bigg)^lS_{n-l}\big(g_{n-l}(\cdot, t,0)\big)\bigg\}\nonumber\\
			&=\exp\Big\{-{q\over 2}\|f\|_{\cal H}^2\Big\}\sum_{n=0}^\infty\bigg\{\sum_{l=0}^{2n}{1\over l!}
			\E \bigg(\int_{\R^d}f(x)W(dx)\bigg)^lS_{2n-l}\big(g_{2n-l}(\cdot, t,0)\big)\bigg\}\nonumber\\
			&\ge\exp\Big\{-{q\over 2}\|f\|_{\cal H}^2\Big\}\sum_{n=0}^\infty\bigg\{\sum_{l=n}^{2n}{1\over l!}
			\E\bigg(\int_{\R^d}f(x)W(dx)\bigg)^lS_{2n-l}\big(g_{2n-l}(\cdot, t,0)\big)\bigg\}\nonumber\\
			&=\exp\Big\{-{q\over 2}\|f\|_{\cal H}^2\Big\}\sum_{n=0}^\infty\bigg\{\sum_{l=0}^n{1\over (n+l)!}
			\E\bigg(\int_{\R^d}f(x)W(dx)\bigg)^{n+l}S_{n-l}\big(g_{n-l}(\cdot, t,0)\big)\bigg\}\,, \nonumber
		\end{align}
		where the second equality follows from (\ref{M-17}), and the second 
		inequality follows from
		the fact that all terms are non-negative.

		For each $0\le l\le n$, by (\ref{M-10}) and (\ref{M-16})
		$$
		\begin{aligned}
			&\E \bigg(\int_{\R^d}f(x)W(dx)\bigg)^{n+l}S_{n-l}\big(g_{n-l}(\cdot, t, 0)\big)\\
			&\quad =\E \int_{(\R^d)^{2n}}g_{n-l}(x_1,\cdots, x_{n-l}, t,0)\bigg(\prod_{k=n-l+1}^{2n}f(x_k)\bigg)
			W(dx_1)\cdots W(dx_{2n})\\
			&\quad =\sum_{{\cal D}\in \Pi_n}
			\int_{(\R^d)^{2n}}d{\bf x}\bigg(\prod_{(j,k)\in {\cal D}}\gamma(x_j-x_k)\bigg)
			g_{n-l}(x_1,\cdots, x_{n-l}, t,0)\bigg(\prod_{k=n-l+1}^{2n}f(x_k)\bigg)\,.
		\end{aligned}
		$$
		We now count
		how many pair partitions ${\cal D}\in \Pi_n$  that make
		\begin{align}\label{LDP-29}
			&\int_{(\R^d)^{2n}}d{\bf x}\bigg(\prod_{(j,k)\in {\cal D}}\gamma(x_j-x_k)\bigg)
			g_{n-l}(x_1,\cdots, x_{n-l}, t,0)\bigg(\prod_{k=n-l+1}^{2n}f(x_k)\bigg)\\
			&=\|f\|_{\cal H}^{2l}\bigg\{\int_{(\R^d)^{2(n-l)}}d{\bf x}d{\bf y}\bigg(\prod_{k=1}^{n-l}\gamma(x_k-y_k)\bigg)
			g_{n-l}(x_1,\cdots, x_{n-l}, t, 0)\bigg(\prod_{k=1}^{n-l}f(y_k)\bigg)\bigg\}\,. \nonumber
		\end{align}
		
		To produce such ${\cal D}$, we first partition $\{n-l+1,\cdots, 2n\}$ into
		two disjoint sets $A_1$ and $A_2$ such that $\#(A_1)=n-l$ and $\#(A_2)=2l$.
		The number of ways to carry out this step is
		$$
		\left(\begin{array}{c} n+l\\ 2l\end{array}\right)\,. 
		$$
		Then we use the elements in $A_1$ to make $n-l$ pairs with the numbers $1,\cdots, n-l$, there are
		$(n-l)!$ ways to do this step. Finally, we pick a pair
		partition ${\cal D}_0$ on $A_2$ together
		with the earlier $n-l$ pairs to form a pair partition ${\cal D}\in \Pi_n$
		---there are ${(2l)!\over 2^l l!}$ ways to finish this step.
		By the Fubini theorem,
		one can see that the pair partitions ${\cal D}$ produced in such way satisfy
		(\ref{LDP-29}). By multiplication principle, there are at least
		$$
		\left(\begin{array}{c} n+l\\ 2l\end{array}\right)
		(n-l)! {(2l)!\over 2^l l!}={(n+l)!\over 2^ll!}
		$$
		pair partitions that make Equation (\ref{LDP-29}) happen.
		
		Write
		$$
		\begin{aligned}
			&\int_{(\R^d)^{2(n-l)}}d{\bf x}d{\bf y}\bigg(\prod_{k=1}^{n-l}\gamma(x_k-y_k)\bigg)
			g_{n-l}(x_1,\cdots, x_{n-l}, t, 0)\bigg(\prod_{k=1}^{n-l}f(y_k)\bigg)\\
			&=\langle f^{\otimes (n-l)}, g_{n-l}(\cdot, t,0)\rangle_{{\cal H}^{\otimes (n-l)}}\,. 
		\end{aligned}
		$$
		In summary,
		$$
		\begin{aligned}
			&\E \left[ \bigg(\int_{\R^d}f(x)W(dx)\bigg)^{n+l}S_{n-l}\big(g_{n-l}(\cdot, t, 0)\big)\right] 
			\ge {(n+l)!\over  2^l l!}\|f\|_{\cal H}^{2l}
			\langle f^{\otimes (n-l)}, g_{n-l}(\cdot, t,0)\rangle_{{\cal H}^{\otimes (n-l)}}\,. 
		\end{aligned}
		$$
		
		Therefore,
		$$
		\begin{aligned}
			\sum_{n=0}^\infty\bigg\{&\sum_{l=0}^n{1\over (n+l)!}
			\E \bigg(\int_{\R^d}f(x)W(dx)\bigg)^{n+l}S_{n-l}\big(g_{n-l}(\cdot, t,0)\big)\bigg\}\\
			&\ge \sum_{n=0}^\infty \bigg\{\sum_{l=0}^n{1\over l!2^l}\|f\|_{\cal H}^{2l}
			\langle f^{\otimes (n-l)},g_{n-l}(\cdot, t, 0)\rangle_{{\cal H}^{\otimes (n-l)}}\bigg\}\\
			&=\bigg\{\sum_{n=0}^\infty{1\over n!2^n}\|f\|_{\cal H}^{2n}\bigg\}
			\bigg\{\sum_{n=0}^\infty\langle f^{\otimes n},g_{n}(\cdot, t, 0)\rangle_{{\cal H}^{\otimes n}}
			\bigg\}\\
			&=\exp\Big\{{1\over 2}\|f\|_{\cal H}^2\Big\}
			\bigg\{\sum_{n=0}^\infty\langle f^{\otimes n},g_{n}(\cdot, t, 0)\rangle_{{\cal H}^{\otimes n}}\bigg\}\,. 
		\end{aligned}
		$$
		In view of (\ref{LDP-28}), we have completed the proof of the lemma. 
	\end{proof}
	
	\begin{proof}[Proof of (\ref{LDP-26})]   Replacing $f(x)$ by
		$$
		f_p(x)=\big((p-1)t\big)^{(2-\alpha +d){1\over 3-\alpha}}
		f\Big((p-1)t\big)^{1\over 3-\alpha}x\Big)
		$$
		in Lemma \ref{L-6} we get 
		$$
		\|u(t,0)\|_p\ge\exp\Big\{-{1\over 2(p-1)}\|f_p\|_{\cal H}^2\Big\}
		\sum_{n=0}^\infty\langle f_p^{\otimes n},g_{n}(\cdot, t, 0)\rangle_{{\cal H}^{\otimes n}}\,. 
		$$
		Set
		$$
		t_p=(p-1)^{1\over 3-\alpha}t^{4-\alpha\over 3-\alpha}\,. 
		$$
		First notice that
		$$
		\|f_p\|_{\cal H}^2=\big((p-1)t\big)^{4-\alpha\over 3-\alpha}\|f\|_{\cal H}^2
		$$
		and by time change and homogeneity  of $\gamma(\cdot)$ and $G(t,x)$,
		$$
		\sum_{n=0}^\infty\langle f_p^{\otimes n},g_{n}(\cdot, t, 0)\rangle_{{\cal H}^{\otimes n}}
		=\sum_{n=0}^\infty \langle f^{\otimes n},g_{n}(\cdot, t_p, 0)\rangle_{{\cal H}^{\otimes n}}\,. 
		$$
		Hence,
		$$
		\begin{aligned}
			\|u(t,0)\|_p&\ge\exp\Big\{-{t_p\over 2}\|f\|_{\cal H}^2\Big\}
			\sum_{n=0}^\infty \langle f^{\otimes n},g_{n}(\cdot, t_p, 0)\rangle_{{\cal H}^{\otimes n}}\\
			&\ge \exp\Big\{-{t_p\over 2}\|f\|_{\cal H}^2\Big\}
			\langle f^{\otimes n},g_{n}(\cdot, t_p, 0)\rangle_{{\cal H}^{\otimes n}}\,, \hskip.2in n=0,1,2,\cdots
		\end{aligned}
		$$
		Let $a>0$ be fixed but arbitrary. Take supremum over $\|f\|_{\cal H}=a$.
		The action can be taken alternatively as $f$ is replaced by
		$af$ and supremum is over $\|f\|_{\cal H}=1$:
		\begin{align}\label{LDP-30}
			\|u(t,0)\|_p&\ge \exp\Big\{-{t_p\over 2}a^2\Big\} a^n\sup_{\|f\|_{\cal H}=1}
			\langle f^{\otimes n},g_{n}(\cdot, t_p, 0)\rangle_{{\cal H}^{\otimes n}}\\
			&=\exp\Big\{-{t_p\over 2}a^2\Big\}a^nt_p^{{4-\alpha\over 2}n}\sup_{\|f\|_{\cal H}=1}
			\langle f^{\otimes n},g_{n}(\cdot, 1, 0)\rangle_{{\cal H}^{\otimes n}}\,, \nonumber
		\end{align}
		where the last step follows from the scaling
		property
		\begin{align}\label{LDP-31}
			\sup_{\|f\|_{\cal H}=1}
			\langle f^{\otimes n},g_{n}(\cdot, t, 0)\rangle_{{\cal H}^{\otimes n}}
			=t^{{4-\alpha\over 2}n}\sup_{\|f\|_{\cal H}=1}\langle f^{\otimes n},g_{n}(\cdot, 1, 0)\rangle_{{\cal H}^{\otimes n}}\,, 
			\hskip.2in \forall t>0\,. 
		\end{align}
		Here we should mention that the supremum should be taken over the functions $f$ with $\|f\|_{\cal H}=1$
		and $f\ge 0$ where the constraint ``$f\ge 0$'' is inherited from Lemma \ref{L-6}.
		We removed ``$f\ge 0$'' from the above discussion as $g_{n}(\cdot, 1, 0)\ge 0$ and therefore
		$$
		\sup_{\stackrel{\scriptstyle\|f\|_{\cal H}=1}{\scriptstyle f\ge 0}}
		\langle f^{\otimes n},g_{n}(\cdot, t, 0)\rangle_{{\cal H}^{\otimes n}}
		=\sup_{\|f\|_{\cal H}=1}\langle f^{\otimes n},g_{n}(\cdot, t, 0)\rangle_{{\cal H}^{\otimes n}}.
		$$
		
		Let $0<\theta<1$ be fixed but arbitrary. Multiplying $(1-\theta)\theta^n$ on the both sides of
		(\ref{LDP-30})
		and summing up both sides over $n=0,1,2,\cdots$,
		\begin{align}\label{LDP-32}
			\|u(t,0)\|_p\ge (1-\theta)\exp\Big\{-{t_p\over 2}a^2\Big\} \sum_{n=0}^\infty
			(\theta a)^nt_p^{{4-\alpha\over 2}n}\sup_{\|f\|_{\cal H}=1}
			\langle f^{\otimes n},g_{n}(\cdot, 1, 0)\rangle_{{\cal H}^{\otimes n}}\,. 
		\end{align}

		On the other hand, 
		$$
		\begin{aligned}
			\int_0^\infty dt  & e^{-t}\sup_{\|f\|_{\cal H}=1}
			\langle f^{\otimes n},g_{n}(\cdot, t, 0)\rangle_{{\cal H}^{\otimes n}}
			\ge \sup_{\|f\|_{\cal H}=1}\int_0^\infty dt e^{-t} 
			\langle f^{\otimes n},g_{n}(\cdot, t, 0)\rangle_{{\cal H}^{\otimes n}}  \\
			&=\sup_{\|f\|_{\cal H}=1}\int_{(\R^d)^n }\mu^{\otimes n}(d\xi)\bigg(\prod_{k=1}^n {\cal F}(f)(\xi_k)\bigg)
			\prod_{k=1}^n\bigg\{1+\Big\vert\sum_{j=k}^n\xi_j\Big\vert^2\bigg\}^{-1}\,, 
		\end{aligned}
		$$
		where
		$$
		{\cal F}(f)(\xi)=\int_{ \R^d }e^{i\xi\cdot x}f(x)dx
		$$
		is the Fourier transform of $f$ and the last step follows from a treatment similar to  the one conducted
		in (\ref{LDP-17}). In view of the scaling identity (\ref{LDP-31}), this inequality
		can be written as
		$$
		\begin{aligned}
			\sup_{\|f\|_{\cal H}=1}&
			\langle f^{\otimes n},g_{n}(\cdot, 1, 0)\rangle_{{\cal H}^{\otimes n}}\\
			&\ge \bigg(\int_0^\infty e^{-t}t^{{4-\alpha\over 2}n}dt\bigg)^{-1}
			\sup_{\|f\|_{\cal H}=1}\int_{(\R^d)}\mu^{\otimes n}(d\xi)\bigg(\prod_{k=1}^n {\cal F}(f)(\xi_k)\bigg)
			\prod_{k=1}^n\bigg\{1+\Big\vert\sum_{j=k}^n\xi_j\Big\vert^2\bigg\}^{-1}\\
			&=\Gamma\Big(1+{4-\alpha\over 2}n\Big)^{-1}\sup_{\|f\|_{\cal H}=1}\int_{(\R^d)}\mu^{\otimes n}(d\xi)
			\bigg(\prod_{k=1}^n {\cal F}(f)(\xi_k)\bigg)
			\prod_{k=1}^n\bigg\{1+\Big\vert\sum_{j=k}^n\xi_j\Big\vert^2\bigg\}^{-1}\,. 
		\end{aligned}
		$$
		By (\ref{LDP-18}), (\ref{LDP-19}) and the Stirling formula
		$$
		\liminf_{n\to\infty}{1\over n}\log (n!)^{4-\alpha\over 2}\sup_{\|f\|_{\cal H}=1}
		\langle f^{\otimes n},g_{n}(\cdot, 1, 0)\rangle_{{\cal H}^{\otimes n}}
		\ge \log \bigg({2{\cal M}^{1/2}\over 4-\alpha}\bigg)^{4-\alpha\over 2}\,. 
		$$
		Hence,
		$$
		\begin{aligned}
			\liminf_{p\to\infty}{1\over t_p}&\log\sum_{n=0}^\infty
			(\theta a)^nt_p^{{4-\alpha\over 2}n}\sup_{\|f\|_{\cal H}=1}
			\langle f^{\otimes n},g_{n}(\cdot, 1, 0)\rangle_{{\cal H}^{\otimes n}}\cr
			&\ge \lim_{p\to\infty}{1\over t_p}\log\sum_{n=0}^\infty (n!)^{-{4-\alpha\over 2}}
			\bigg((\theta a)\Big({2{\cal M}^{1/2}\over 4-\alpha}\Big)^{4-\alpha\over 2}\bigg)^n t_p^{{4-\alpha\over 2}n}\\
			&={4-\alpha\over 2}\bigg(\theta a
			\Big({2{\cal M}^{1/2}\over 4-\alpha}\Big)^{4-\alpha\over 2}\bigg)^{2\over 4-\alpha}
			=(\theta a)^{2\over 4-\alpha}{\cal M}^{1/2}\,, 
		\end{aligned}
		$$
		where the second step follows from (\ref{LDP-5}) with $\gamma=\displaystyle {4-\alpha\over 2}$
		and  $b=\displaystyle t_p^{4-\alpha\over 2}$.
		
		By (\ref{LDP-32}), therefore,
		$$
		\liminf_{p\to\infty}{1\over t_p}\log\|u(t,0)\|_p
		\ge  -{1\over 2}a^2+(\theta a)^{2\over 4-\alpha}{\cal M}^{1/2}\,. 
		$$
		Letting $\theta\to 1^-$ yields 
		$$
		\liminf_{p\to\infty}{1\over t_p}\log\|u(t,0)\|_p
		\ge  -{1\over 2}a^2+a^{2\over 4-\alpha}{\cal M}^{1/2}\,. 
		$$
		Taking the supremum over $a>0$ on the right hand side,
		\begin{align}\label{LDP-33}
			\liminf_{p\to\infty}{1\over t_p}\log\|u(t,0)\|_p\ge
			{3-\alpha\over 2}\bigg({2{\cal M}^{1/2}\over 4-\alpha}
			\bigg)^{4-\alpha\over 3-\alpha}\,. 
		\end{align}
		By definition of $t_p$  this is   (\ref{LDP-26}). 
	\end{proof}
	
	\begin{remark}\label{re-10}
		Under an obvious modification, the same proof also leads to (\ref{LDP-33})
		with fixed $p\ge 1$ and with $t\to\infty$. Consequently, it leads to
		the lower bound for (\ref{intro-9}) in the special case when $p$ is an even integer.
	\end{remark}
	
	\section{Appendix}\label{s.a} 
	
	\subsection{Moment bounds for Brownian intersection local times}
	
	Let $B(t), B_1(t), B_2(t)$ be  independent $d$-dimensional Brownian motions.
	
	\begin{lemma}\label {L-10} Assume   Dalang's condition (\ref{intro-6}).
		There is a constant $C>0$, independent of $n$ and $t$, such that
		\begin{align}\label{A-1}
			\E_0\bigg[\int_0^t\!\!\int_0^t\gamma\big(B(s)-B(r)\big)dsdr\bigg]^n
			\le C (n!)^2 (t\vee t^2)^n\,, \hskip.2in n=1,2,\cdots
		\end{align}
		\begin{align}\label{A-2}
			\E_0\bigg[\int_0^t\!\!\int_0^t\gamma\big(B_1(s)-B_2(r)\big)dsdr\bigg]^n
			\le C (n!)^2 (t\vee t^2)^n\,, \hskip.2in n=1,2,\cdots
		\end{align}
	\end{lemma}
	
	\proof Write
	$$
	Z_t=\bigg(\int_0^t\!\!\int_0^t\gamma\big(B(s)-B(r)\big)dsdr\bigg)^{1/2}\,, \hskip.2in t\ge 0\,. 
	$$
	To prove (\ref{A-1}) all we need is the bound
	\begin{align}\label{A-3}
		\E_0 Z_t^n\le n! C^n (\sqrt{t}\vee t)^n\,, \hskip.2in n=1,2,\cdots
	\end{align}
	First,  $Z_t$ is non-decreasing, almost surely continuous with $Z_0=0$. From (A.9), \cite{Chen-2}
	$Z_t$ is sub-additive: For any $t_1,t_2>0$, there is a random variable $Z_{t_2}'$ such that
	$Z_{t_2}'\buildrel d\over =Z_{t_2}$ and $Z_{t_2}'$ is independent of $\{Z_s;\hskip.1in s\le t_1\}$.
	By (1.3.7), p.21, \cite{Chen-1}, therefore,
	$$
	\P_0\big\{Z_{t_0}\ge a+b\big\}\le \P_0\big\{Z_{t_0}\ge a\big\}\P_0\big\{Z_{t_0}\ge b\big\}
	$$
	for any $t_0, a, b>0$. Thus, for any integer $m\ge 1$,
	$$
	\begin{aligned}
		\E_0 Z_t^n&=(e\E_0 Z_t)^n\E_0\Big({Z_t\over e\E_0 Z_t}\Big)^n\\
		&=(e\E_0 Z_t)^nn\int_0^\infty b^{n-1}\P_0\{Z_t\ge e b\E_0Z_t\}db\\
		&=(e\E_0 Z_t)^n\bigg\{n\int_0^1b^{n-1}db+n\int_1^\infty \P_0\{Z_t\ge e b\E_0Z_t\}db\bigg\}\\
		&\le(e\E_0 Z_t)^n\bigg\{1+n\int_1^\infty b^{n-1}\Big(\P_0\{Z_t\ge e \E_0Z_t\}\Big)^{b-1}db\bigg\}\,. 
	\end{aligned}
	$$
	The claimed bound (\ref{A-3}) follows from the following estimation
	$$
	\int_1^\infty b^{n-1}\Big(\P_0\{Z_t\ge e \E_0Z_t\}\Big)^{b-1}db
	\le e\int_0^\infty b^{n-1}e^{-b}db=en!
	$$
	and the bound ((A.6), Appendix, \cite{Chen-2})
	$$
	\E_0 Z_t\le \Big(\E_0 Z_t^2\Big)^{1/2}\le \Big(C(t\vee t^2)\Big)^{1/2}\,. 
	$$
	
	We now prove (\ref{A-2}). Let $\dot{W}(x)$ be a Gaussian noise 
	independent of $B, B_1, B_2$ and having covariance $\gamma(\cdot)$.
	Conditioning on the Brownian motions
	$$
	\E\bigg[\int_0^t\dot{W}\big(B_1(s)\big)ds\bigg]
	\bigg[\int_0^t\dot{W}\big(B_2(s)\big)ds\bigg]
	=\int_0^t\!\!\int_0^t\gamma
	\big(B_1(s)-B_2(r)\big)dsdr\,.
	$$
	In addition, by the Cauchy-Schwartz inequality
	$$
	\begin{aligned}
		\E\bigg[\int_0^t&\dot{W}\big(B_1(s)\big)ds\bigg]
		\bigg[\int_0^t\dot{W}\big(B_2(s)\big)ds\bigg]\\
		&\le\bigg\{\E\bigg[\int_0^t\dot{W}\big(B_1(s)\big)ds\bigg]^2\bigg\}^{1/2}
		\bigg\{\E\bigg[\int_0^t\dot{W}\big(B_2(s)\big)ds\bigg]^2\bigg\}^{1/2}\\
		&=\bigg\{\int_0^t\!\!\int_0^t\gamma\big(B_1(s)-B_1(r)\big)dsdr\bigg\}^{1/2}
		\bigg\{\int_0^t\!\!\int_0^t\gamma\big(B_2(s)-B_2(r)\big)dsdr\bigg\}^{1/2}\,. 
	\end{aligned}
	$$
	Hence,
	$$
	\begin{aligned}
		\int_0^t\!\!\int_0^t\gamma
		\big(B_1(s)-B_2(r)\big)dsdr
		\le&\bigg\{\int_0^t\!\!\int_0^t\gamma\big(B_1(s)-B_1(r)\big)dsdr\bigg\}^{1/2}\\
		&\quad \times 
		\bigg\{\int_0^t\!\!\int_0^t\gamma\big(B_2(s)-B_2(r)\big)dsdr\bigg\}^{1/2}\,. 
	\end{aligned}
	$$
	By the independence between $B_1$ and $B_2$,
	$$
	\begin{aligned}
		&\E_0\bigg[\int_0^t\!\!\int_0^t\gamma\big(B_1(s)-B_2(r)\big)dsdr
		\bigg]^n
		\le\bigg\{\E_0\bigg[\int_0^t\!\!\int_0^t\gamma\big(B(s)-B(r)\big)dsdr\bigg]^{n/2}
		\bigg\}^2\\
		&\le \E_0\bigg[\int_0^t\!\!\int_0^t\gamma\big(B(s)-B(r)\big)dsdr\bigg]^n\,. 
	\end{aligned}
	$$
	Therefore, (\ref{A-2}) follows from (\ref{A-1}).
	\qed
	
	\subsection{Hu-Meyer formula}
	
	Although Lemma \ref{L-0} gives a way for us  
	to show the existence of a multiple Stratonovich integral we also need to know   what kind 
	general conditions to impose  on $f $ so that its multiple Stratonovich 
	integral $S_n(f)$ exists, namely the approximation in \eqref{M-6} has a limit in ${\cal L}^2(\Omega, {\cal F}, \P)$.  If the multiple Stratonovich integral $S_n(f )$ exists in ${\cal L}^2(\Omega, {\cal F}, \P)$, then according to general It\^o-Wiener's chaos expansion theorem it admits a chaos expansion and it is interesting  to find this  chaos expansion.  
	For this we shall establish a Hu-Meyer formula along the line of 
	\cite{humeyer, humeyer93}.  If $f\in {\cal H}^{\otimes n} $ is a (generalized) symmetric function of $n$-variables such that
	\begin{equation}
		\begin{split}
			\|f\|_{{\cal H}^{\otimes n}}^2 :=&\int_{(\R^d)^{2n}} f (x_1, \cdots, x_n) 
			f(y_1, \cdots, y_n)\\
			&\qquad \times  \gamma(x_1-y_1) \cdots \gamma(x_n-y_n) dx_1dy_1\cdots dx_ndy_n<\infty\,,  
		\end{split}	
		\label{e.6.4a} 
	\end{equation} 
	then its multiple It\^o-Skorohod integral 
	exists and is denoted by 
	\[
	I_n(f )=\int_{(\R^d)^{n}} f (x_1, \cdots, x_n) \delta W(x_1)\cdots \delta W(x_n)\,, 
	\]
	where $\delta W$ denotes the It\^o-Skorohod stochastic integral. 
	To precisely define  ${\cal H}^{\otimes n}$, we can complete 
	the set of all  symmetric  smooth  functions  
	with compact supports under the Hilbert norm defined by \eqref{e.6.4a}. It is well-known that the Hilbert space  ${\cal H}^{\otimes n}$ contains generalized functions (see e.g. \cite{pipiras}).  
	
	Recall our definition \eqref{M-5}  that   $W_\varepsilon(x)=\int_{\R^d} 
	p_\varepsilon(x-y) W(dy)=I_1(p_\varepsilon(x-\cdot))$. 
	From \cite[Corollary 5.1, Equation 5.3.15]{hubook},  it follows that the chaos expansion of $\prod_{k=1}^n\dot{W}_{\varepsilon}(x_k)$ is
	\begin{eqnarray}
		\prod_{k=1}^n\dot{W}_{\varepsilon}(x_k)
		&=&\sum_{k\le n/2} \sum_{i_1 < j_1, \cdots, i_k < j_k} \prod_{\ell=1}^k \int_{\R^{2d}} p_\varepsilon(x_{i_\ell}
		-y)\gamma(y-z) p_\varepsilon(x_{j_\ell}
		-z) dydz \nonumber \\
		&&\qquad\qquad   I_{n-2k}(\Lambda_ {i_1, j_1, \cdots, i_k, j_k} \otimes_{m =1}^n p_\varepsilon(x_m-\cdot))\nonumber\\
		&=&\sum_{k\le n/2} \sum_{i_1 < j_1, \cdots, i_k < j_k} \prod_{\ell=1}^k \gamma_{2\varepsilon}(x_{i_\ell}-x_{j_\ell})  \nonumber\\
		&&\qquad \qquad     I_{n-2k}(\Lambda_ {i_1, j_1, \cdots, i_k, j_k} \otimes_{m  =1}^n p_\varepsilon(x_m -\cdot))\,, \label{e.6.4} 
	\end{eqnarray}
	where 
	\begin{enumerate}
		\item[(i)] The set of distinct elements 
		$i_1 < j_1, \cdots, i_k < j_k$  is a subset of
		$\{1,2, \cdots, n\}$ and the summation $\sum_{i_1 < j_1, \cdots, i_k < j_k}$ is   over all such distinct pairs; 
		\item[(ii)] The function $\Lambda_ {i_1, j_1, \cdots, i_k, j_k}
		\otimes_{m=1}^n p_\varepsilon(x_m-\cdot)$ is defined as the symmetrization of
		the function
		$$
		\prod_{m \in [1,n]\setminus\{i_1, j_1, \cdots, i_k, j_k\} }p_\epsilon
		(x_m -y_m)
		$$
		over the variables $\big(y_m; \hskip.05in
		m\in [1,n]\setminus\{i_1, j_1, \cdots, i_k, j_k\}\big)$, i.e.,
		$$
		\Lambda_ {i_1, j_1, \cdots, i_k, j_k}
		\otimes_{m=1}^n p_\varepsilon(x_m- y_m) 
		={1\over (n-2k)!}\sum_{\sigma}
		\prod_{m \in [1,n]\setminus\{i_1, j_1, \cdots, i_k, j_k\} }p_\epsilon
		(x_m -y_{\sigma(m)} )
		$$
		where the summation is over all permutations $\sigma$ on 
		$[1,n]\setminus\{i_1, j_1, \cdots, i_k, j_k\}$. When $k=0$,
		in particular, we follow the natural convention that
		\[
		\Lambda_ {i_1, j_1, \cdots, i_k, j_k}\otimes_{m=1}^n p_\varepsilon(x_m-y_m)={1\over n!}\sum_{\sigma\in\Sigma_n}\prod_{m=1}^n 
		p_\varepsilon(x_m-y_{\sigma(m)})\,, 
		\]
		where $\Sigma_n$ is the permutation group on $\{1,\cdots, n\}$.
		\item[(iii)]  $I_{n-2k}(\cdots)$ is  the  multiple  It\^o-Wiener (It\^o-Skorohod) integral 
		with the integration variables
		$\left\{y_m\,;\ m\in [1,n]\setminus\{i_1, j_1, \cdots, i_k, j_k\}\right\}$.
	\end{enumerate} 
	With the above chaos expansion 
	\eqref{e.6.4} we see that    the chaos expansion of
	the approximated Stratonovich integral is
	\begin{eqnarray}
		S_{n, \varepsilon}(f ) & =& \int_{(\R^d)^{n}}f (x_1,\cdots, x_{n})\bigg(\prod_{k=1}^{n}\dot{W}_{\varepsilon}(x_k)\bigg)
		dx_1\cdots dx_{n}\nonumber\\
		& 
		=&
		\sum_{k\le n/2} \sum_{i_1 < j_1, \cdots, i_k < j_k} \int_{(\R^d)^{n}}f (x_1,\cdots, x_{n}) \bigg(\prod_{\ell=1}^k \gamma_{2\varepsilon}(x_{i_\ell}-x_{j_\ell}) \bigg)\nonumber\\
		&&\qquad \qquad      \times I_{n-2k}\Big(\Lambda_ {i_1, j_1, \cdots, i_k, j_k} \otimes_{m=1}^n p_\varepsilon(x_m-\cdot)\Big)dx_1\cdots dx_n\,.
	\end{eqnarray}
	By the symmetry of $f $ on $x_1, \cdots, x_n$ 
	and with a combinatorial   analysis as in \cite{humeyer93} 
	the above equation can be written 
	\begin{eqnarray}
		S_{n, \varepsilon}(f )	 
		&=&
		\sum_{k\le n/2}  \frac{n!}{2^k k! (n-2k)!}  I_{n-2k}
		\left(  \int_{(\R^d)^{2n}}f(x_1,\cdots, x_{ n}) \prod_{\ell=1}^k \gamma_{2\varepsilon}(x_{2\ell-1}-x_{2\ell})
		\right. \nonumber\\
		&&\qquad \qquad  \qquad \qquad   \left.      \times \prod_{j={2k+1}}^n    p_\varepsilon(x_j-\cdot))dx_1\cdots dx_n\right) \,.
		\label{e.2.20} 
	\end{eqnarray}
	
	Since the approximated multiple integral can be decomposed to finite sum of   multiple It\^o-Wiener integrals which are  orthogonal, we see that  the convergence in ${\cal L}^2(\Omega, {\cal F}, \P)$ of \eqref{M-6} is 	equivalent to that each of the multiple It\^o-Wiener integrals in \eqref{e.2.20} converges in ${\cal L}^2(\Omega, {\cal F}, \P)$. Thus, we have the following theorem which is used to justify    \eqref{M-16}. 
	\begin{theorem}\label{t.6.2} 
		Let $f \in {\cal H}^{\otimes n}$ be deterministic
		and symmetric.
		If the trace  
		\begin{eqnarray}
			\tr^k f (y_{2k+1}, \cdots, y_n) 
			&:= & \lim_{ \varepsilon\to 0}\int_{(\R^d)^{n}}f (x_1,\cdots, x_{ n})   \prod_{\ell=1}^k \gamma_{2\varepsilon}(x_{2\ell-1}-x_{2\ell})
			\nonumber\\
			&&\qquad  \times\prod_{j={2k+1}}^n    p_\varepsilon(x_j-y_j))dx_1\cdots dx_n\label{e.6.7} 
		\end{eqnarray} 
		exists in ${\cal H}^{\otimes (n-2k)}$  for all $k\le n/2$, then the Stratonovich 
		integral $S_n(f )$ exists as an ${\cal L}^2(\Omega, {\cal F}, \P)$ limit of 
		$S_{n, \varepsilon}(f )$ as $\varepsilon\to 0$ and we have the following Hu-Meyer formula: 
		\begin{equation}
			S_n(f )=\sum_{k\le n/2}  \frac{n!}{2^k k! (n-2k)!} 
			I_{n-2k}(\tr^k f )\,. \label{e.2.22} 
		\end{equation}
		Conversely,    if   
		$S_{n, \varepsilon}(f )$ is a Cauchy sequence in 
		${\cal L}^2(\Omega, {\cal F}, \P)$, then the right hand side of \eqref{e.6.7} is a Cauchy sequence in
		${\cal H}^{\otimes (n-2k)}$  for all $k\le n/2$, whose limit is denoted by the left hand side of 
		\eqref{e.6.7}  and $S_{n, \varepsilon}(f )$ converges to $S_n(f)$ defined by \eqref{e.2.22} in ${\cal L}^2(\Omega, {\cal F}, \P)$.  
		Moreover, if   
		$S_{n, \varepsilon}(f )$ converges to $S_n(f )$ in   ${\cal L}^2(\Omega, {\cal F}, \P) $,  then
		this convergence also takes place in 
		${\cal L}^p(\Omega, {\cal F}, \P)$  for any $p\in [1, \infty)$.  This means that
		$S_n(f )$ is in   ${\cal L}^p(\Omega, {\cal F}, \P)$  for any $p\in [1, \infty)$. 
	\end{theorem} 
	\begin{remark}\label{r.6.3}
		It is obvious that if $  f $ is the symmetrization 
		of $\hat f $,  then by the above definition it is easy to verify that
		$S_n(f )=S_n(\hat f )$.  
	\end{remark} 
	\begin{proof}[Proof of the theorem] 
		Denote 
		\[
		g_{n,k, \varepsilon}(y_{2k+1}, \cdots, y_n):=	\int_{(\R^d)^{2n}}f (x_1,\cdots, x_{ n})   \prod_{\ell=1}^k \gamma_{2\varepsilon}(x_{2\ell-1}-x_{2\ell}) \prod_{j={2k+1}}^n    p_\varepsilon(x_j-y_j))dx_1\cdots dx_n\,. 
		\]
		Equation \eqref{e.6.7} means $\|g_{n,k, \varepsilon}-\tr ^{k} f \|_{{\cal H}^{\otimes (n-2k}}\to 0$ as $\varepsilon\to 0$.
		By the It\^o isometry,
		\begin{equation}
			\begin{split}
				\E|I_{n-2k}(g_{n,k, \varepsilon} )
				&-I_{n-2k}(\tr^k f )|^2=
				\E|I_{n-2k}(g_{n,k, \varepsilon}-\tr^k f )|^2
				\\
				=&(n-2k)! \|g_{n,k, \varepsilon}-\tr^k f \|_
				{{\cal H}^{\otimes (n-2k)}}^2\to 0   
			\end{split} \label{e.6.9}
		\end{equation}
		by \eqref{e.6.7}.  
		Equation \eqref{e.2.20}  tells   that 
		$S_{n,\varepsilon}(f )$ converges to
		$S_n(f )$  given by  \eqref{e.2.22}.  
		
		Now we assume that  $	S_{n, \varepsilon}(f )$ is a Cauchy sequence in ${\cal L}^2(\Omega, {\cal F}, \P)$.
		With our notation $g_{n,k, \varepsilon} $  we can write
		\begin{eqnarray*} 
			S_{n, \varepsilon}(f )	 
			&=&
			\sum_{k\le n/2}  \frac{n!}{2^k k! (n-2k)!}  I_{n-2k}
			\left(  g_{n,k, \varepsilon}  \right) \,. 
		\end{eqnarray*}
		Thus, by the orthogonality of multiple It\^o-Wiener integrals, 
		\[
		\begin{split} 
			\E\left[ S_{n, \varepsilon}(f )	-S_{n, \varepsilon'}(f )	\right]^2= &
			\sum_{k\le n/2}  \left(\frac{n!}{2^k k! (n-2k)!}\right)^2
			\E \left[   I_{n-2k}
			\left(  g_{n,k, \varepsilon}   \right)
			-I_{n-2k}
			\left(  g_{n,k, \varepsilon'}   \right)\right]^2\\
			= &
			\sum_{k\le n/2}  \left(\frac{n!}{2^k k! (n-2k)!}\right)^2
			(n-2k)! \|  g_{n,k, \varepsilon}   
			-  g_{n,k, \varepsilon'}   \|_{{\cal H}^{\otimes (n-2k)}}^2\,. 
		\end{split}
		\]
		This can be used to prove  the second part of the theorem easily. 
		
		Recall that if $F=\sum_{n=0}^\infty I_n(f_n)$ is the chaos expansion of $F$, then 
		the   second quantization operator 
		(e.g. \cite{hubook}) of a number $\alpha\in [-1, 1]$
		is defined as
		\[
		\Gamma(\alpha) F=\sum_{n=0}^\infty \alpha^n I_n(f_n)\,. 
		\] 
		Now for any $p>2$, let   $\alpha=\sqrt{\frac{1}{p-1}}$ and let  
		\begin{eqnarray} 
			F_{t, n,  \varepsilon}
			&=&
			\sum_{k\le n/2}  (1/\alpha)^{n-2k} \frac{n!}{2^k k! (n-2k)!} \left[ I_{n-2k}
			\left( g_{n,k, \varepsilon}\right) -	I_{n-2k}(\tr^k f )\right]\,.   
		\end{eqnarray}  
		Then
		by the hypercontractivity inequality (e.g. \cite[p. 54, Theorem 3.20]{hubook}, we have
		\[
		\begin{split}
			\left(\E |	S_{n, \varepsilon}(f )	\right.
			&\left. - 	S_n(f )|^p\right)^{1/p}
			=  \left(\E |\Gamma(\alpha)	F_{t, n,  \varepsilon}|^p\right)^{1/p} \le \left(\E | 	F_{t, n,  \varepsilon}|^2\right)^{1/2} \\
			= &
			\left( \E| \sum_{k\le n/2}  (1/\alpha)^{n-2k} \frac{n!}{2^k k! (n-2k)!} \left[ I_{n-2k}
			\left( g_{n,k, \varepsilon}\right) -	I_{n-2k}(\tr^k f )\right]|^2\right)^{1/2} \\
			\le &
			\left(  \sum_{k\le n/2}  (1/\alpha)^{2n-4k} \frac{(n!)^2}{2^{2k} (k!)^2 ((n-2k)!)^2 } 
			\E \left[ I_{n-2k}
			\left( g_{n,k, \varepsilon}\right) -	I_{n-2k}(\tr^k f )\right] ^2\right)^{1/2}    \,, 
		\end{split}
		\]
		which converges to $0$ by \eqref{e.6.9}.  
		This proves the theorem. 
	\end{proof}

	
	

	\vskip 1.in
	
	\begin{tabular}{lll}
		Xia Chen  \\
		Department of Mathematics  \\
		University of Tennessee \\
		Knoxville TN 37996, USA  \\
		{\tt xchen@math.utk.edu} 
	\end{tabular}
	\hskip 18mm \begin{tabular}{lll}
		Yaozhong Hu  \\
		Department of Math and  Stat Sciences  \\
		University of Alberta at Edmonton \\
		Edmonton, Canada,  T6G 2G1  \\
		{\tt yaozhong@ualberta.ca} 
	\end{tabular}

\end{document}